\documentclass[11pt]{amsart}

\usepackage{amsmath, amssymb,amsthm,color}
\usepackage[alphabetic]{amsrefs}
\usepackage[all]{xy}
\usepackage{tikz-cd}
\usetikzlibrary{arrows.meta,decorations.markings,arrows}

\usepackage{pgfplots}
\usepgfplotslibrary{polar}
\pgfplotsset{compat=newest}

\newcommand{\nc}{\newcommand}

\newcommand{\uz}{{\underline{z}}}

\newcommand{\uu}{{\underline{u}}}
\newcommand{\uk}{{\underline{k}}}

\newcommand{\C}{\mathbb{C}}
\newcommand{\R}{\mathbb R}
\newcommand{\PP}{\mathbb P}

\renewcommand{\sl}{\mathfrak{sl}}

\newcommand{\wt}{\widetilde}

\newcommand{\Cx}{\C^\times}
\nc{\vareps}{\varepsilon}
\nc{\arrg}{\operatorname{arg}}
\nc{\Cf}{\mathcal{T}}

\nc{\circQ}{Q^\circ}
\nc{\circCQ}{\CQ^\circ}
\nc{\circcircQ}{Q^{\circ \circ}}
\nc{\circcircCQ}{\CQ^{\circ \circ}}
\nc{\circY}{\ft}
\nc{\Ycirc}{\overline \ft^\circ}
\nc{\Cfcirc}{\overline \Cf^\circ}
\nc{\circCY}{\Cf}
\nc{\circcircY}{\ft^{\circ}}
\nc{\circcircCY}{\Cf^{\circ}}
\nc{\bV}{\tilde{V}}
\nc{\bU}{\tilde{U}}
\nc{\bCU}{\tilde{\CU}}
\nc{\bCV}{\tilde{\CV}}
\nc{\bbV}{V}
\nc{\bbU}{U}
\nc{\bbCU}{{\CU}}
\nc{\bbCV}{{\CV}}
\nc{\circcircM}{M^{\circ \circ}}

\nc{\Gr}{\mathbb{G}}

\nc{\A}{{\mathcal A}}
\nc{\ol}{\overline}
\nc\tboxtimes{\wt{\boxtimes}}
\nc{\alp}{\alpha}
\nc{\Wh}{\operatorname{Wh}}
\nc{\IC}{{\mathcal{IC}}}
\nc{\Uhg}{{U_\hbar \fg}}
\nc{\tW}{{\tilde{\Gr}}}
\nc{\bM}{\mathbf M}
\nc{\la}{\lambda}
\nc{\tM}{\widetilde{M}}
\nc{\tCM}{\widetilde{\mathcal M}}

\nc{\BA}{{\mathbb{A}}} \nc{\BC}{{\mathbb{C}}}
\nc{\BQ}{{\mathbb{Q}}}
\nc{\BM}{{\mathbb{M}}} \nc{\BN}{{\mathbb{N}}}
\nc{\BP}{{\mathbb{P}}} \nc{\BR}{{\mathbb{R}}}
\nc{\BZ}{{\mathbb{Z}}} \nc{\BS}{{\mathbb{S}}}
\nc{\BG}{{\mathbb{G}}}

\nc{\CA}{{\mathcal{A}}} \nc{\CB}{{\mathcal{B}}}
\nc{\CD}{{\mathcal{D}}}
\nc{\CE}{{\mathcal{E}}} \nc{\CF}{{\mathcal{\overline F}}}
\nc{\CG}{{\mathcal{G}}} \nc{\CH}{{\mathcal{H}}}
\nc{\CI}{{\mathcal{I}}}  \nc{\CJ}{{\mathcal{J}}}
\nc{\CL}{{\mathcal{L}}}
\nc{\CM}{{\mathcal{M}}} \nc{\CN}{{\mathcal{N}}}
\nc{\CO}{{\mathcal{O}}} \nc{\CP}{{\mathcal{P}}}
\nc{\CQ}{{\mathcal{Q}}} \nc{\CR}{{\mathcal{R}}}
\nc{\CS}{{\mathcal{S}}} \nc{\CT}{{\mathcal{T}}}
\nc{\CU}{{\mathcal{U}}} \nc{\CV}{{\mathcal{V}}}
\nc{\CX}{{\mathcal{X}}}
\nc{\CW}{{\mathcal{W}}} \nc{\CZ}{{\mathcal{Z}}}

\nc{\ff}{{\mathfrak{f}}} \nc{\fv}{{\mathfrak{v}}}
\nc{\fa}{{\mathfrak{a}}} \nc{\fb}{{\mathfrak{b}}}
\nc{\fd}{{\mathfrak{d}}} \nc{\fe}{{\mathfrak{e}}}
\nc{\fg}{{\mathfrak{g}}} \nc{\fgl}{{\mathfrak{gl}}}
\nc{\fh}{{\mathfrak{h}}} \nc{\fri}{{\mathfrak{i}}}
\nc{\fj}{{\mathfrak{j}}} \nc{\fk}{{\mathfrak{k}}}
\nc{\fm}{{\mathfrak{m}}} \nc{\fn}{{\mathfrak{n}}}
\nc{\ft}{{\mathfrak{t}}} \nc{\fu}{{\mathfrak{u}}}
\nc{\fw}{{\mathfrak{w}}} \nc{\fz}{{\mathfrak{z}}}
\nc{\fl}{{\mathfrak{l}}}
\nc{\fp}{{\mathfrak{p}}} \nc{\frr}{{\mathfrak{r}}}
\nc{\fs}{{\mathfrak{s}}} \nc{\fsl}{{\mathfrak{sl}}}
\nc{\hsl}{{\widehat{\mathfrak{sl}}}}
\nc{\hgl}{{\widehat{\mathfrak{gl}}}}
\nc{\hg}{{\widehat{\mathfrak{g}}}}
\nc{\chg}{{\widehat{\mathfrak{g}}}{}^\vee}
\nc{\hn}{{\widehat{\mathfrak{n}}}}
\nc{\chn}{{\widehat{\mathfrak{n}}}{}^\vee}
\nc{\Xiset}{\Xi\text{-}\mathtt{Set}}
\nc{\Lset}{\Lambda_+\text{-}\mathtt{Set}}
\nc{\set}{\mathtt{Set}}
\nc{\gcrys}{\fg\text{-}\mathtt{Crys}}
\nc{\sqcu}{\bigsqcup\limits}
\nc{\FR}{\overline{FM}^\BR}

\usepackage{accents}
\newcommand{\dbtilde}[1]{\accentset{\approx}{#1}}

\DeclareMathOperator{\id}{id}

 \DeclareMathOperator{\Spec}{Spec}

\DeclareMathOperator{\bhr}{\bar{\mathfrak h}(\mathbb R)}
\DeclareMathOperator{\bh}{\bar{\mathfrak h}}

\newtheorem{cor}{Corollary}[section]

\newtheorem{lem}[cor]{Lemma}
\newtheorem{prop}[cor]{Proposition}

\newtheorem{conj}[cor]{Conjecture}
\newtheorem{thm}[cor]{Theorem}

\theoremstyle{definition}
\newtheorem{defn}[cor]{Definition}
\newtheorem{rem}[cor]{Remark}
\newtheorem{eg}[cor]{Example}
\newtheorem{example}[cor]{Example}

\theoremstyle{remark}

\title{The moduli space of cactus flower curves and the virtual cactus group}

\author{Aleksei Ilin}
\address{Skolkovo Institute of Science and Technology, Moscow, Russia \\
	National Research University Higher School of Economics, Moscow, Russia \\
	{\tt alex.omsk2@gmail.com}}
\author{Joel Kamnitzer}
\address{Department of Mathematics and Statistics, McGill University \\ Montreal QC, Canada\\
{\tt joel.kamnitzer@mcgill.ca}}
\author{Yu Li}
\address{Department of Mathematics, University of Toronto \\
	Toronto, ON, Canada \\ {\tt liyu@math.toronto.edu}}
\author{Piotr Przytycki}
\address{Department of Mathematics and Statistics, McGill University \\ Montreal QC, Canada \\ {\tt piotr.przytycki@mcgill.ca}}
\author{Leonid Rybnikov}
\address{ Department of Mathematics, Massachusetts Institute of Technology\\
Cambridge MA, USA \\
On leave from HSE University, Moscow\\
{\tt leo.rybnikov@gmail.com}}

\begin{document}
	\maketitle

\begin{abstract}
	The space $ \ft_n = \C^n/\C $ of $n$ points on the line modulo translation has a natural compactification $ \overline \ft_n $ as a matroid Schubert variety.  In this space, pairwise distances between points can be infinite; it is natural to imagine points at infinite distance from each other as living on different projective lines.  We call such a configuration of points a ``flower curve'', since we picture the projective lines joined into a flower.  Within $ \ft_n $, we have the space $ F_n  = \C^n \setminus \Delta / \C $ of $ n$ distinct points.  We introduce a natural compatification $ \overline F_n $ along with a map $ \overline F_n \rightarrow \overline \ft_n $, whose fibres are products of genus 0 Deligne-Mumford spaces. We show that both $\overline \ft_n$ and $\overline F_n$, are special fibers of $1$-parameter families whose generic fibers are, respectively, Losev-Manin and Deligne-Mumford moduli spaces of stable genus $0$ curves with $n+2$ marked points.
	
	We find combinatorial models for the real loci $ \overline \ft_n(\BR) $ and $ \overline F_n(\BR) $.  Using these models, we prove that these spaces are aspherical and that their equivariant fundamental groups are the virtual symmetric group and the virtual cactus groups, respectively. The degeneration of a twisted real form of the Deligne-Mumford space to $\overline F_n(\mathbb{R})$ gives rise to a natural homomorphism from the affine cactus group to the virtual cactus group.
\end{abstract}
	

\section{Introduction}

\subsection{Moduli space of stable genus 0 curves}

The Deligne-Mumford space $ \overline M_{n+1}$ of stable genus 0 curves with $ n+1 $ marked points (here we will call these ``cactus curves'') has been intensively studied in algebraic geometry, representation theory, and algebraic combinatorics.

Going back to the work of Kapranov \cite[Theorem 4.3.3]{Kap}, $ \overline M_{n+1} $ can be constructed as an iterated blowup of projective space along a certain family of subspaces.  This construction was generalized by de Concini-Procesi \cite{dCP}, who defined a wonderful compactification of the complement of any hyperplane arrangement, so that $ \overline M_{n+1} $ is the wonderful compactification of the projectivized type $A_{n-1}$ root hyperplane arrangement.

Losev-Manin \cite{LM} introduced an alternate construction of $ \overline M_{n+1} $.  They began with the permutahedral toric variety $ \overline T_n $ (also called the Losev-Manin space), where $ T_n = (\Cx)^n / \Cx $, and then performed a series of blowups to get $ \overline M_{n+2} $.

In representation theory, Aguirre-Felder-Veselov \cite{AFV} proved that $ \overline M_{n+1} $ parametrizes maximal commutative subalgebras of the Drinfeld-Kohno Lie algebra.  Their result was used by the fifth author \cite{Ryb1} to prove that $ \overline M_{n+1} $ parametrizes Gaudin subalgebras in $ (U \fg)^{\otimes n} $, where $ \fg $ is any semisimple Lie algebra.

The real locus $ \overline M_{n+1}(\BR) $ is a beautiful combinatorial space of independent interest.  Kapranov \cite[Proposition 4.8]{Kap2} and Devadoss \cite[Theorem 3.1.3]{Dev} proved that it is tiled by $ n!/2 $ copies of the associahedron.  Dual to this tiling is a cube complex studied by Davis-Januszkiewicz-Scott \cite{DJS}, who proved that $ \overline M_{n+1}(\BR) $ is the classifying space for the cactus group $C_n$, a finitely generated group analogous to the braid group.

The second and fifth authors, along with Halacheva and Weekes, studied the monodromy of eigenvectors for Gaudin algebras \cite{HKRW} over this real locus.  They proved that this monodromy is given by the action of the cactus group on tensor product of crystals, as defined in \cite{HK}.

\subsection{The moduli space of cactus flower curves}

In this paper, we will study the moduli space $\overline F_n $ of cactus flower curves, an additive analog of the Deligne-Mumford space. Much as $ \overline M_{n+2} $ is a compactification of
$$
M_{n+2} := (\BP^1)^{n+2} \setminus \Delta / PGL_2 = (\Cx)^n \setminus \Delta / \Cx $$
the space of $n $ distinct points in the multiplicative group,
our space $ \overline F_n $ will be a compactification of
$$
F_n := \C^n \setminus \Delta / \C
$$
the space of $n $ distinct points in the additive group.

Given a point $ (z_1, \dots, z_n) \in F_n $, we can consider all possible distances between points $ \delta_{ij} = z_i - z_j $. These distances are non-zero and non-infinite and obey the ``triangle equalities'' $ \delta_{ij} + \delta_{jk} = \delta_{ik} $.  In Section~\ref{se:flower}, we define $ \overline \ft_n $ to be the scheme defined by these triangle equations, but where we allow $\delta_{ij} $ to take any value in $ \BP^1$.

When distances between points are infinite, it is natural to view the points as living on different projective lines.  As we explain in Remark \ref{rem:Ynparam}, we will view these lines as glued together at a single point, thus forming a ``flower'' with many petals. More precisely, a flower curve is a curve $ C = C_1 \cup \cdots \cup C_m $ with $ n $ marked points $ z_1, \dots, z_n \in C $, such that each component $ C_j $ is isomorphic to $ \BP^1 $, all meet at a common distinguished point $ z_{n+1} $, and each carry a non-zero tangent vector at $ z_{n+1}$ (the marked points are not required to be distinct). We will regard $ \overline \ft_n $ as the moduli space of flower curves.  In a followup paper with Zahariuc, we will precisely formulate a moduli functor for these flower curves and prove that it is represented by $ \overline \ft_n $.

\begin{figure}
	\includegraphics[trim=0 80 60 40, clip,width=\textwidth]{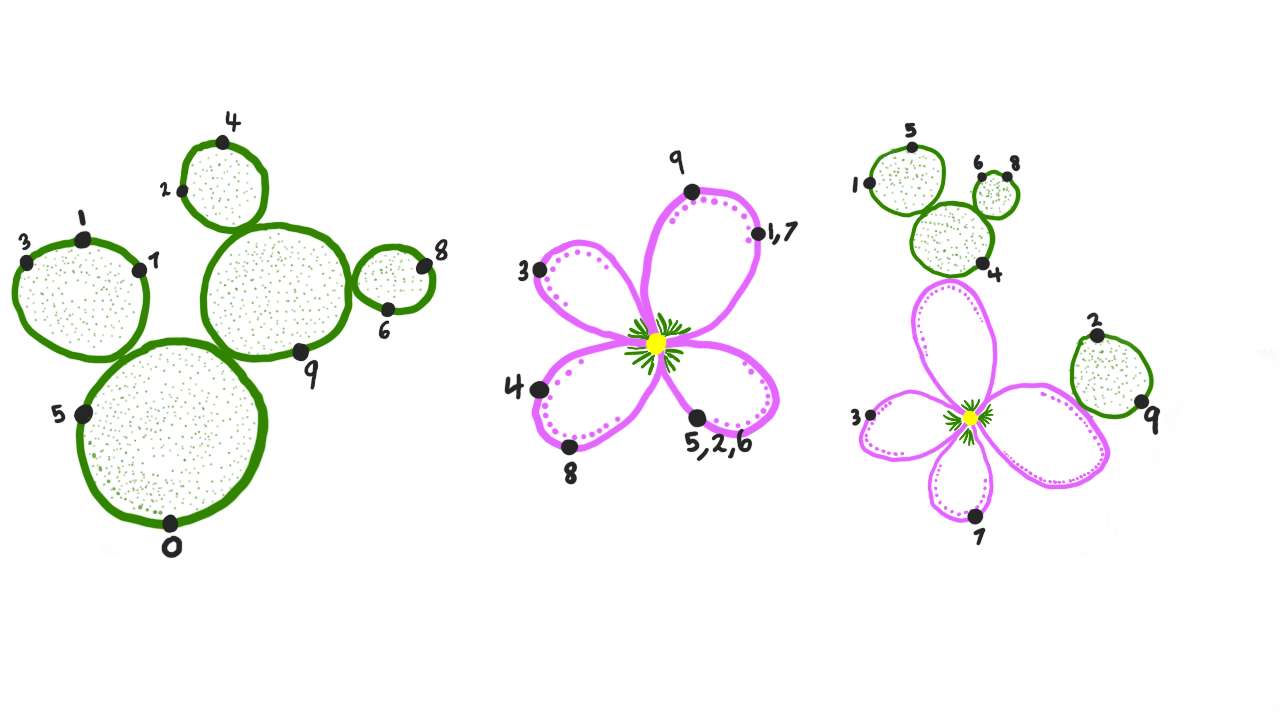}
	\caption{A point of $ \overline M_{9+1}$ (a cactus curve), a point of $ \overline \ft_9$ (a flower curve), and a point of $ \overline F_9$ (a cactus flower curve).}
\end{figure}

The space $ \overline \ft_n $ is a special case of a matroid Schubert variety.  Ardila-Boocher \cite{AB} defined this compactification of a vector space, depending on a hyperplane arrangement; in our case, the hyperplane arrangement is the type A root arrangement.

The space $ \overline \ft_n $ is an analog of the Losev-Manin space, so it is natural to construct $ \overline F_n $ from $ \overline \ft_n $ by iterated blowups.  Due to the singularities of the space, this is technically difficult, so we follow a different approach.  We cover $ \overline \ft_n $ by a collection of open affine subschemes $ U_\CS $ (these are indexed by set partitions $ \CS $ of $ \{1, \dots, n\} $).  In Section \ref{se:defFn}, over each open set $ U_\CS$, we give an explicit construction of its desired preimage $ \bU_\CS $.  Then we glue these $\bU_\CS $ together to form our scheme $ \overline F_n $.  For example if $ \CS = \{ \{1, \dots, n\} \} $ (the set partition with one part), then $ U_\CS = \ft_n $ and $ \bU_\CS $ is $ \tM_{n+1} $, the total space of the natural line bundle over $ \overline M_{n+1} $.

From our construction of $ \overline F_n $, there is a natural morphism $ \gamma: \overline F_n \rightarrow \overline \ft_n $ whose fibres are products of Deligne-Mumford spaces.  For this reason, we regard $ \overline F_n $ as a moduli space of cactus flower curves; in a cactus flower curve, all the marked points are distinct and each $C_j $ is a usual cactus curve (a stable genus 0 curve with distinct marked points).

A different compactification of $ F_n $ was previously defined by Mau-Woodward \cite{MW}.  Their space $ Q_n $ has the advantage that it can be defined directly as a subscheme of a product of projective lines, see Section \ref{se:MWdef}.  However, their space is too big for our purposes, as the fibres of $ Q_n \rightarrow \overline \ft_n $ are larger than we desire.  From our construction of $ \overline F_n$, we are able to construct a morphism $ Q_n \rightarrow \overline F_n $.

We emphasize that $ \overline F_n $ is not a resolution of $ \overline \ft_n$.  In fact, the locus of the most singular point in $ \overline \ft_n $ (called the \textbf{reciprocal plane}) embeds into $ \overline F_n $ as well.  As for any matroid Schubert variety, the \textbf{augmented wonderful variety} $W_n $ (see \cite{BHMPW}) provides a resolution of $ \overline \ft_n$.  The maps $ W_n \rightarrow \overline \ft_n $ and $ \overline F_n \rightarrow \overline \ft_n $ are ``orthogonal''.  The morphism $ W_n \rightarrow \overline \ft_n $ blows up the strata at $ \infty $  (corresponding to $ \delta_{ij} = \infty$), while the morphism $ \overline F_n \rightarrow \ft_n $ blows up the strata in $ \ft_n $ (corresponding to $ \delta_{ij} = 0 $).

\subsection{Deformation}
There is a degeneration of the multiplicative group $ \Cx $ to the additive group $ \C $.  This leads to a degeneration of the Losev-Manin space $ \overline T_n $ to $ \overline \ft_n $.  We write $\overline \Cf_n \rightarrow \BA^1 $ for the total space of the degeneration which we will regard as a deformation of $ \overline \ft_n $. (We were inspired by closely related degenerations studied by Zahariuc \cite{Z}).

In a similar way, we will define a deformation $ \CF_n $ of $ \overline F_n$, whose general fibre is $ \overline M_{n+2}$.  Similarly, the Mau-Woodward space $ Q_n $ also admits a deformation $ \CQ_n $ whose general fibre is $ \overline M_{n+2} $. This leads to the following diagram
 \begin{equation} \label{eq:maindiagram}
	\begin{tikzcd}
		\overline M_{n+2} \ar[r] \ar[equal]{d} & \CQ_n \ar[d] & \ar[l] \ar[d] Q_n \\
		\overline M_{n+2} \ar[r] \ar[d] & \CF_n \ar[d]  & \ar[l] \ar[d] \overline F_n \\
		\overline T_n \ar[r] \ar[d] & \overline \Cf_n \ar[d] & \ar[l] \overline \ft_n \ar[d] \\
		\{\vareps \ne 0 \} \ar[r] & \BA^1 & \ar[l] \{0 \}
	\end{tikzcd}
\end{equation}
Geometrically, the degeneration of $ \overline M_{n+2} $ to $ \overline F_n $ parametrizes a family of marked curves where two marked points come together to form a distinguished point with a tangent vector.

\subsection{Trigonometric and inhomogeneous Gaudin algebras}
Our main motivation for this paper was the theory of Gaudin subalgebras.  Let $ \fg $ be a semisimple Lie algebra.  As mentioned above, the compactification of the moduli space of Gaudin subalgebras of $ (U \fg)^{\otimes n} $ is given by $ \overline M_{n+1} $.  In a companion paper \cite{IKR}, we will study trigonometric and inhomogenous Gaudin subalgebras of $  (U \fg)^{\otimes n} $ (with fixed element of the Cartan).  The non-compactified moduli space of these algebras is $ M_{n+2} $ (for trigonometric) and $ F_n $ (for inhomogeneous).  In \cite{IKR}, we will prove that these families of commutative subalgebras extend to $ \overline M_{n+2} $ and $\overline{F_n}$, respectively. Moreover, these families of subalgebras join into the one parametrized by our scheme $ \CF_n $.

\subsection{Real locus} \label{se:introreal}
As with $ \overline M_{n+1}$, the real locus of the cactus flower space, $ \overline F_n(\BR) $, is a beautiful combinatorial space.  In Section~\ref{se:CombIso}, we prove that $ \overline F_n(\BR) $ is homeomorphic to a cube complex $ \widehat D_n $, whose cubes are labeled by planar forests.  Similarly, $ \overline \ft_n(\BR) $ has a combinatorial description as the quotient of the permutahedron by the equivalence relation which identifies all parallel faces.  This quotient of the permutahedron was previous considered in \cite{BEER}.

These combinatorial descriptions allow us to identify the fundamental groups of the real loci.  The \textbf{cactus group} $ C_n $ is defined to be the group with generators $ s_{ij} $ for $ 1 \le i< j \le n $ and relations
	 \begin{enumerate}
	\item $ s_{ij}^2 = 1 $
	\item $ s_{ij} s_{kl} = s_{kl} s_{ij} $ if $ [i,j] \cap [k,l] = \emptyset $
	\item $ s_{ij} s_{kl} = s_{ w_{ij}(l)  w_{ij}(k)} s_{ij} $ if $ [k,l] \subset [i,j] $
\end{enumerate}
Here $ w_{ij} \in S_n $ is the element of $ S_n $ which reverses $ [i,j] $ and leaves invariant the elements outside this interval. From \cite{DJS}, we have an isomorphism $ \pi_1^{S_n}(\overline M_{n+1}(\BR)) \cong C_n $.

Taking inspiration from the virtual braid group \cite{virtual} (see also the virtual twin group \cite{NNS} and virtual Artin group \cite{BPT}), we introduce the \textbf{virtual cactus group} and the \textbf{virtual symmetric group} (this latter group, or its pure variant, has previously appeared in the literature under the names ``flat braid group'', ``upper virtual braid group'', ``triangular group'').  The virtual cactus group $vC_n$ is generated by a copy of the cactus group $ C_n$ and the symmetric group $S_n $, subject to the relations
$$ w s_{ij} = s_{w(i) w(j)} w, \text{ if $ w \in S_n $ and $ w(i+k) = w(i) + k $ for $ k = 1, \dots, j -i $}$$
The virtual symmetric group, $vS_n$, has a similar presentation involving two copies of the symmetric group.

Using the combinatorial descriptions of these spaces, we prove the following result (Theorem \ref{th:pi1geom1}).
\begin{thm}
We have isomorphisms $ \pi_1^{S_n}(\overline F_n(\BR)) \cong vC_n $ and $ \pi_1^{S_n}(\overline \ft_n(\BR)) \cong vS_n $.  Moreover the higher homotopy groups of these spaces vanish.
\end{thm}

We also study a twisted real form $ \overline M^\sigma_{n+2}(\BR) $ of $ \overline M_{n+2} $ which is compatible with the compact form $ U(1)^n/U(1) $ of $T_n $. Geometrically, $ \overline M^\sigma_{n+2}(\BR) $ parametrizes $ (C, \uz) $ where $ \bar z_0 = z_{n+1} $ and $ \bar z_i = z_i $ for $i = 1, \dots, n $.  This real form (among others) was studied by Ceyhan \cite{C}.  Using his results, we prove the following (Theorem \ref{th:pi1geom2}).

\begin{thm}
	We have an isomorphism $ \pi^{S_n}_1( \overline M_{n+2}^\sigma(\BR) ) \cong \widetilde{AC}_n $.
\end{thm}
Here $ \widetilde{AC}_n $ is the extended affine cactus group.  It is defined by starting with the affine cactus group $ AC_n$, which has generators $ s_{ij} $ for $ 1 \le i \ne j \le n $ (corresponding to intervals in the cyclic order on $ \BZ/n $), and then forming the semidirect product with $ \BZ/n $ (see Section~\ref{se:ac} for the precise definition).

There is a twisted real form $  \CF^\sigma_n(\BR)  $ of $ \CF_n $ whose generic fibre is  $ \overline M^\sigma_{n+2}(\BR) $ and whose special fibre is $ \overline F_n(\BR) $.  We prove the following result (Thoerems \ref{th:deformretract} and \ref{th:pi1ACvC}) concerning its fundamental group.
\begin{thm}
	$\overline F_n(\BR) $ is a deformation retract of $ \CF^\sigma_n(\BR) $.  The resulting morphism $$ \widetilde{AC}_n \cong \pi^{S_n}_1( \overline M_{n+2}^\sigma(\BR) ) \rightarrow \pi^{S_n}_1( \CF^\sigma_n(\BR)) \cong \pi^{S_n}_1(\overline F_n(\BR)) \cong vC_n $$
	can be described explicitly on generators.
\end{thm}

As mentioned above, $ \CF_n $ is a moduli space of trigonometric and inhomogeneous Gaudin algebras in $ (U \fg)^{\otimes n} $ where $ \fg $ is a semisimple Lie algebra.  As in \cite{HKRW}, we can study a cover of $ \CF_n^\sigma(\BR)$ whose fibres are eigenvectors for these algebras acting on a tensor product $ V_1 \otimes \cdots \otimes V_n $ of representations of $ \fg$.  Generalizing our work \cite{HKRW}, in a future paper we will prove that the monodromy of this cover gives an action of the virtual cactus group which is isomorphic to its action on the tensor products of crystals $ B_1 \otimes \cdots \otimes B_n$ for these representations (see \cite{IKR} for a precise statement).

\subsection{Generalizations}
Many of the construction presented here have natural generalizations to other root systems and to arbitrary hyperplane arrangements.  To begin, we can study the compactification $\overline \fh $ of the Cartan subalgebra of any semisimple Lie algebra (the matroid Schubert variety of the root hyperplane arrangement). Such a study was initiated by Evens-Li \cite{EL}.  In the appendix, we study and give a combinatorial description of the real locus $ \overline \fh(\BR) $, proving that it is the quotient of the corresponding permutahedron by the equivalence relation of parallel faces.

In future work, we will define an analog of $ \overline F_n $ for any hyperplane arrangement.  As in this paper, the definition of this space will combine aspects of the Ardila-Boocher and the de Concini-Procesi constructions.  We will also give a combinatorial description of its real locus.

\subsection{Acknowledgements}
We would like to thank Ana Balibanu, Dror Bar-Natan, Laurent Bartholdi, Paolo Bellingeri, Matthew Dyer, Pavel Etingof, Evgeny Feigin, Davide Gaiotto, Victor Ginzburg, Iva Halacheva, Yibo Ji, Leo Jiang, Michael McBreen, Sam Payne, Nick Proudfoot, and Adrian Zahariuc for helpful conversations. We thank Bella Kamnitzer for Figure 1. The work of A.I. is an output of a research project implemented as part of the Basic Research Program at the National Research University Higher School of Economics (HSE University). The work was accomplished during L.R.'s stay at the Institut des Hautes Études Scientifiques (IHÉS) and at Harvard University. L.R. would like to thank IHÉS, especially Maxim Kontsevich, and Harvard University, especially Dennis Gaitsgory, for their hospitality. Part of this work was done during the stay of Y.L. at the Max Planck Institute for Mathematics (MPIM).  The hospitality of MPIM is gratefully acknowledged.

\section{Some combinatorics}
Throughout this paper, $ n $ will be a natural number and we write $ [n] := \{1, \dots, n\}$.

For each finite set $ S $, let $ p(S) $ denote the set of pairs $ (i,j)$ of distinct elements of $ S $.  We will abuse notation by abbreviating $ (i,j) $ to $ ij$.  Similarly, we write $ t(S) $ for the set of triples $ (i,j,k) $ (abbreviated to $ ijk$) of distinct elements of $ S $.

	Here are some combinatorics which will be useful for labelling strata in the flower space. A \textbf{set partition} of $ [n] $ is a set $ \CS = \{S_1, \dots, S_m \}$ of subsets of $ [n] $ such that $ S_1 \sqcup \cdots \sqcup S_m = [n] $ (the order of the subsets is not important).  Such a set partition defines an equivalence relation $ \sim_\CS $ on $ [n] $ where these are the equivalence classes.  Conversely, an equivalence relation on $ [n] $ determines a set partition of $ [n]$.  The two extreme set partitions are $ \{ [n]\} $, the unique set partition with 1 part, and $ [[n]] := \{ \{1\}, \dots, \{n\} \} $, the unique set partition with $n $ parts.

For labelling strata in the cactus flower space, we will need some finer combinatorics.  A \textbf{tree} is a connected graph without cycles.  A \textbf{forest} is a graph without cycles, or equivalently a disjoint union of trees.  Given two vertices $ v, w $ of a tree, there is a unique embedded edge path connecting them: we call this \textbf{the path} between $v$ and $w$.

A \textbf{rooted tree} is a tree with a distinguished vertex, called the \textbf{root}, contained in exactly one edge, called the \textbf{trunk}, and with no vertices contained in exactly two edges. Given a non-root vertex $v$ of a rooted tree $\tau$,
the unique edge containing $v$ that lies on the path between $v$ and the root is \textbf{descending at $v$}. The remaining edges containing $v$ are \textbf{ascending at $v$}. A vertex with no ascending edges is a \textbf{leaf}. A vertex $v$  is \textbf{above} an edge $ e $ (resp.\ a vertex $ w$), if the path between $v$ and the root contains $e$ (resp.\ $w$).  This defines a partial order on the set of vertices.

A \textbf{rooted forest} is a disjoint union of rooted trees.  The above notions extend to rooted forests.  If $ v, w $ are two vertices of a rooted forest, for $ v $ to be above $ w $, we must have $ v, w $ on the same component of the forest.

Let $ S $ be a finite set, often we take $ S = [n] $.  We say a rooted forest is $S$-\textbf{labelled} (or \textbf{labelled by $S$}) if we are given a bijection between $ S $ and the set of leaves of $ \tau $.  Any $ S$-labelled forest $ \tau = \{\tau_1, \dots, \tau_m\} $ (where $ \tau_1, \dots, \tau_m $ are rooted trees) determines a set partition of $ S = S_1 \sqcup \dots \sqcup S_m $ where $ S_j $ is the set of labels of $\tau_j $.

Let $ \tau $ be an $S$-labelled tree.  Given $ i, j \in S $, their \textbf{meet} is the unique vertex that lies on all three paths between: the root, and the two leaves corresponding to $i$ and $j$. Equivalently, it is maximal vertex (with respect to the above partial order) which lies below $ i $ and $ j $.

A vertex of a rooted forest is \textbf{internal} if it is neither a leaf, nor a root. A \textbf{binary tree} is a rooted tree in which every internal vertex is contained in exactly three edges (one descending and two ascending).

\section{The Losev-Manin and flower spaces}

\subsection{Losev-Manin space}
Let $ T_n = (\Cx)^n / \Cx $ be the space of $ n$  points on $ \Cx $ modulo scaling.  Let $ \overline T_n $ be the Losev-Manin space, also known as the permutahedral toric variety.  It is a toric variety for $ T_n $, so is equipped with an action of $ T_n $ having an open dense orbit.

For each $ ij \in p([n]) $, we consider the character $ \alpha_{ij} : T_n \rightarrow \Cx $ defined by $ [z_1, \dots, z_n] \mapsto z_i z_j^{-1} $.  This extends to a map $ \overline T_n \rightarrow (\BP^1)^{p([n])}$.

The following result is \cite[Cor. 1.16]{BB}, but was perhaps known earlier.
\begin{lem} \label{le:BBTn}
	The above maps identify $ \overline T_n $ with the subscheme of $ \alpha \in (\BP^1)^{p([n])} $ defined by the equations $$ \alpha_{ij} \alpha_{jk} = \alpha_{ik} \quad \alpha_{ij} \alpha_{ji} = 1 $$
	for distinct $ i,j,k $.
\end{lem}
In this paper, we will not use the toric variety description of $ \overline T_n$, so the reader can take these equations as the definition of $ \overline T_n $.

\begin{rem}
	Here and below, we will consider equations inside a product of $ \BP^1$s.  The meaning of these equations has be interpreted carefully.  For example, when we write $ ab = c $ for $a,b,c \in \BP^1 $, we really mean $ a_1 b_1 c_2 = c_1 a_2 b_2 $, where $a = [a_1 : a_2] $, etc are homogeneous coordinates.  In particular, this equation $ ab = c $ is solved by $ a = 0, b = \infty $ and $ c $ arbitrary.
\end{rem}

\begin{rem} \label{rem:useS1}
	It will be convenient to consider these spaces and later ones as depending on a finite set $ S $ (other than $[n]$).  More precisely, we write $ T_S $ for $ (\Cx)^S / \Cx $, and $ \overline T_S $ for the subscheme of $(\BP^1)^{p(S)} $ defined by the above equations.
\end{rem}

%

A \textbf{caterpillar curve} is a curve $ C = C_1 \cup \dots \cup C_m $, where each $ C_k $ is a projective line, and where each pair $C_k, C_{k+1} $ meet transversely at a single point (with no other intersections); we also assume we are given distinguished smooth points $ z_0 \in C_1 $ and $ z_{n+1} \in C_m $.  A caterpillar curve with $ n$ marked points is a pair $ (C, \uz) $ where $ C $ is a caterpillar curve and each $ z_i \in C $ is a smooth point not equal to $ z_0, z_{n+1} $ (but we allow other points $ z_i, z_j $ to be equal).  Losev-Manin \cite[(2.6.3)]{LM} proved that $ \overline T_n $ is the moduli space of caterpillar curves with $ n$ marked points (see also \cite{BB}).

 There is an open subset consisting of those $ (C, \uz) $ where $ C $ has one component.  Identifying $ C = \BP^1$, we use the $ PGL_2 $ action to fix $ z_0 = 0, z_{n+1} = \infty $, and then we see that this open subset is our torus $ (\Cx)^n / \Cx = T_n$.  In terms of the coordinates above, $ T_n = \{ \alpha :\alpha_{ij} \ne 0, \infty \}$.  In turn, we can consider the locus $(C, \uz) $ where $ C $ has one component and all $ z_i $ are distinct.  It is easy to see that this locus is given by $ \{ \alpha : \alpha_{ij} \ne 0, 1, \infty \} $ and can be identified with $ M_{n+2} := ((\Cx)^n \setminus \Delta) / \Cx $.

\subsection{The flower space} \label{se:flower}
We will now study the \textbf{flower space} $\overline \ft_n$, an additive version of $ \overline T_n$.  By definition this is the subscheme of $ \nu \in (\BP^1)^{p([n])} $ defined by
\begin{equation} \label{eq:ftn}
\nu_{ij} \nu_{jk} = \nu_{ik} \nu_{jk} + \nu_{ij} \nu_{ik} \qquad \nu_{ij} + \nu_{ji} = 0
\end{equation}
for all distinct $ i,j,k$.

Equivalently, we can set $ \delta_{ij} = \nu_{ij}^{-1} $.  In these coordinates, the defining equations of $ \overline \ft_n $ become
$$
\delta_{ij} + \delta_{jk} = \delta_{ik} \qquad \delta_{ij} + \delta_{ji} = 0
$$

Let $ \ft_n = \{ \nu \in \overline \ft_n : \nu_{ij} \ne 0 \text{ for all } i,j \} $, and $ \Ycirc_n = \{ \nu \in \overline \ft_n : \nu_{ij} \ne \infty \text{ for all } i, j \}$. These are two open affine subschemes of $ \overline \ft_n$.  Their intersection will be denoted $ F_n := \Ycirc_n \cap \ft_n $.

On $ \ft_n $, the coordinates $ \delta_{ij} $ are finite, and the following result is immediate.
\begin{lem} \label{le:onepetal}
There is an isomorphism $ \ft_n \cong \C^n / \C $ defined by $ \delta \mapsto (x_1, \dots, x_n) $ where $ \delta_{ij} = x_i -x_j$. This restricts to an isomorphism $ F_n \cong (\C^n \setminus \Delta) / \C $.
\end{lem}

\begin{rem} \label{rem:useS2}
	Following Remark \ref{rem:useS1}, for any finite set $ S $, we will write $ \ft_S := \C^S / \C $ and $ \overline \ft_S $ for the subscheme of $ (\BP^1)^S $ defined by the above equations.
\end{rem}

\begin{rem} \label{rem:vector}
		We may identify $ \C^n/\C $ with $ \C^n \times \Cx / B $ where $ B = \Cx \ltimes \C $ is the Borel subgroup of $ PGL_2 $, and where $ B $ acts on $ \Cx $ by inverse scaling and on each copy of $ \C $ by an affine linear transformation.  Because $ B $ is the stabilizer of $ \infty \in \BP^1$ and acts with weight -1 on its tangent space $ T_{\infty} \BP^1 $, a point $(z_1, \dots, z_n, a) \in \C^n \times \Cx $ can be considered as $ n+1 $ points $z_1, \dots, z_n,  z_{n+1} = \infty  \in \BP^1 $ along with a non-zero tangent vector $ a \in T_{z_{n+1}}\BP^1$.
		
		Thus, we obtain identifications
		\begin{equation} \label{eq:withTangent}
		\begin{gathered}
		\ft_n = \{(z_0, \dots, z_n, a) : z_i \in \BP^1, z_0 \ne z_i, a \in T_{z_{n+1}} \BP^1, a \ne 0 \} / PGL_2
		\\
		F_n = \{(z_0, \dots, z_n, a) : z_i \in \BP^1, z_i \ne z_j, a \in T_{z_{n+1}} \BP^1, a \ne 0 \} / PGL_2
		\end{gathered}
	\end{equation}
 We call $ z_1, \dots, z_n $ marked points and $ z_{n+1}$ the distinguished point.  In $\ft_n $, the marked points are allowed to coincide with each other, but not with the distinguished point.  In $ F_n$, the marked points are all distinct.  In either case, the distinguished point always carries a non-zero tangent vector.
\end{rem}

Consider the family of hyperplanes $ \{ x_i = x_j \} $, for $ ij \in p([n]) $, inside $ \ft^n = \C^n/\C$.  This is the type $A_{n-1}$ root hyperplane arrangement, also known as the \textbf{braid arrangement}.  Associated to this hyperplane arrangement, we consider the inclusion $ \ft_n \rightarrow \C^{p([n])} $ given by $ x \mapsto (x_i - x_j) $ (as in Lemma \ref{le:onepetal}).  The closure of the image of $ \ft_n $ inside of $ (\BP^1)^{p([n])} $ is called the \textbf{matroid Schubert variety} of braid arrangement (this is a special case of the construction of Ardila-Boocher \cite{AB}).  The matroid Schubert variety has an open subset containing $ \infty \in (\BP^1)^{p([n])} $ called the \textbf{reciprocal plane}.
(Here $ \infty $ is the point of $ \overline \ft_n $ defined by $ \delta_{ij} = \infty $ for all $ i,j$.)  The reciprocal plane is defined (by Proudfoot-Speyer \cite{PS}) as $ \Spec OT_n $, where $ OT_n $ is the subalgebra of $ \C(x_1, \dots, x_n) $ generated by $ \frac{1}{x_i - x_j} $ (the Orlik-Terao algebra \cite{Terao}).

\begin{thm} \label{th:ArdilaBoocher}
\begin{enumerate}
  \item The affine scheme $ \Ycirc_n $ is the reciprocal plane.
  \item The scheme $ \overline \ft_n $ is the matroid Schubert variety of the braid arrangement; in particular, it is reduced.
\end{enumerate}
\end{thm}

We thank Sam Payne and Nick Proudfoot for help with the following proof.
\begin{proof}
By Proudfoot-Speyer \cite{PS}, the ideal defining $ OT_n $ is generated by relations coming from all the circuits of the matroid of this hyperplane arrangement.  By Schenck-Tohaneanu \cite[Prop. 2.7]{Schenck}, it suffices to use circuits of size 2 and 3, which correspond to the relations (\ref{eq:ftn}). Thus, $ \Ycirc_n $, which is the affine scheme defined by (\ref{eq:ftn}), is the reciprocal plane.

Now, let $ \CI $ be the ideal sheaf of $ \overline \ft_n $ as a subscheme of $ (\BP^1)^{p([n])} $.  More precisely, $ \CI $ is the ideal sheaf associated to the multihomogeneous ideal $ I \subset \C[\nu_{ij}, \delta_{ij} : ij \in p([n])] $ generated by
\begin{equation*}
\nu_{ij} \nu_{jk} \delta_{ik} - \nu_{ik} \nu_{jk} \delta_{ij} - \nu_{ij} \nu_{ik} \delta_{jk} \qquad \nu_{ij}\delta_{ji} - \nu_{ji}\delta_{ij}
\end{equation*}
On the other hand, let $ \CJ $ be the ideal sheaf of the matroid Schubert variety.  By Ardila-Boocher \cite[Theorem 1.3(a)]{AB}, $\CJ $ is the ideal sheaf associated to the ideal $J$ generated by the homogenization of relations coming from all circuits.  So $ I \subset J $; we do not expect that $ I = J $.  On the other hand, we will show that $ \CI = \CJ $.  To this end, consider the quotient $ \CI/\CJ $, a coherent sheaf on $ (\BP^1)^{p([n])} $.  There is an action of the group $ \ft_n $ on $  (\BP^1)^{p([n])} $ by translation. Both $ \CI $ and $ \CJ $ are equivariant for this action; hence so is the quotient $ \CI/\CJ $.  Thus, the support of $ \CI/\CJ $ is a closed $ \ft_n $-invariant subset of $ (\BP^1)^{p([n])} $.  Hence if it is non-empty, it must contain the point $ \infty \in  (\BP^1)^{p([n])} $.  Thus, it is enough to prove that $ \CI = \CJ $ on an open affine subset of $ \infty$.  By (1), on the natural open affine neighbourhood of $ \infty $ (given by $ \delta_{ij} \ne 0 $), $ \CI $ and $ \CJ $ are both the ideal sheaves of the reciprocal plane and thus are equal.  So the result follows.
\end{proof}

\begin{rem} \label{re:EvensLi}
Independently, Evens-Li \cite{EL} have studied $ \overline{\fg^*} $, a compactification of the dual of a semisimple Lie algebra, analogous to the wonderful compactification $ \overline G$ of a semisimple group $ \overline G $ defined by de Concini-Procesi.  Within $ \overline{\fg^*} $, Evens-Li considered $ \overline \fh $, the closure of the Cartan subalgebra.  They proved that $ \overline \fh $ coincides with matroid Schubert variety associated to the root hyperplane arrangement in $ \fh $.  This explains our notation $ \overline \ft_n $, where $ \ft_n = \C^n/\C $ is the Cartan subalgebra of $ \mathfrak{pgl}_n $.
\end{rem}

Let $ \CS $ be a set partition of $ [n] $.  Let
$$
\bbV_\CS = \{\delta \in \overline \ft_n : \delta_{ij} \ne \infty \text{ if and only if } i \sim_\CS j  \}
$$
Note that $ \bbV_{\{[n]\}} = \ft_n $ and $ \bbV_{[[n]]} = \{\infty \}  $.

\begin{prop} \label{prop:strataYn}
	\begin{enumerate}
		\item
		This defines a decomposition of $ \overline \ft_n $ into locally closed subsets $ \bbV_\CS $.
		\item
		There is an isomorphism $\bbV_\CS \cong \circY_{S_1} \times \cdots \times \circY_{S_m} $ given by
		$$ \delta \mapsto (\delta|_{p(S_1)}, \dots, \delta|_{p(S_m)})$$
	\end{enumerate}
\end{prop}

\begin{proof}
	It is clear that $\bbV_\CS $ are locally closed subsets and that they are disjoint.  We must show that their union is $ \overline \ft_n $. Let $ \delta \in \overline \ft_n $.   Define an equivalence relation on $ \{1, \dots, n\} $ by setting $ i \sim j $ if $ \delta_{ij} \ne \infty$.  The relation $ \delta_{ij} + \delta_{jk} = \delta_{ik}$  implies that if $ i \sim j$ and $ j \sim k $, then $ i \sim k$.  Thus, this defines an equivalence relation.  Hence $ \delta \in V_\CS $ where $ \CS $ is the set of equivalence classes.
	
	The second part is clear because all other $ \delta_{ij} $ equal $ \infty $ by definition.
\end{proof}

\begin{rem} \label{rem:Ynparam}
Because of Proposition~\ref{prop:strataYn}, Lemma \ref{le:onepetal} and Remark \ref{rem:vector}, we can think of a point of $ \overline \ft_n $ as parametrizing a set of projective lines, each carrying a non-empty collection of (possibly non-distinct) marked points along with a non-zero tangent vector at one distinguished point.  If $ z_i, z_j $ are marked points on distinct components, then they have infinite distance from each other, i.e. $\delta_{ij} = \infty $, and they live in different parts of the set partition $ \CS $.

We will think of these projective lines as being attached together at their distinguished points, hence forming the petals of a flower.  This explains the origin of the name flower space.  We write $ (C, \uz) $ for the resulting curve and marked points.

In future work with Zahariuc, we will prove that $ \overline \ft_n $ is the fine moduli space for such flower curves.  In fact, we will show that the universal curve for this moduli space is $\overline \ft_{n+1} \rightarrow \overline \ft_n$.  In particular, to each point $ \delta \in \overline \ft_n $, the corresponding flower curve $ C $ is the fibre of $ \overline \ft_{n+1} \rightarrow \overline \ft_n $ over $ \delta $.  For example, the fibre over the point $ \infty \in \overline \ft_n $ is the maximal flower curve, which consists of $ n $ $ \BP^1$s, each carrying a marked point and all meeting at a single point.  (It is the compactification of the union of the coordinate axes in $ \C^n $.)  For this reason, we call $ \infty \in \overline \ft_n $, the \textbf{maximal flower point}.
\begin{center}
	\begin{tikzpicture}[scale=0.7]
	\begin{polaraxis}[grid=none, axis lines=none,
		ymax = 1.2]
		\addplot[mark=none,domain=0:360,samples=300] {sin(5*x)};
				\filldraw (18,1) circle [radius=2pt]  node[right] {$z_1$};
				\filldraw (90,1) circle [radius=2pt]  node[right] {$z_2$};
				\filldraw (90+72,1) circle [radius=2pt]  node[left] {$z_3$};
				\filldraw (90+2*72,1) circle [radius=2pt]  node[left] {$z_4$};
				\filldraw (90 + 3*72,1) circle [radius=2pt]  node[right] {$z_5$};
	\end{polaraxis}
	\end{tikzpicture}
\end{center}
\end{rem}

Now, we consider a different stratification.  Let $ \CB $ be another set partition of $ [n] $ and define
$$
\bbV^\CB = \{ \delta \in \overline \ft_n : \delta_{ij} = 0 \text{ if and only if } i \sim_\CB j  \}
$$
This is the locus where two marked points $ z_i, z_j $ are equal if and only if $i, j$ lie in the same part of the set partition $ \CB $.  The proof of the following result is very similar to Proposition \ref{prop:strataYn}.

\begin{prop} \label{prop:strataYn2}
	\begin{enumerate}
		\item
		This defines a decomposition of $ \overline \ft_n $ into locally closed subsets $ \bbV^\CB $.
		\item
		There is an isomorphism $\bbV^\CB \cong \Ycirc_r $, where $ r $ is the number of parts in $ \CB $.
	\end{enumerate}
\end{prop}
Note also that $ \bbV^{\{[n]\}} = \{ 0 \} $ where $ 0 $ is the point where $ \delta_{ij} = 0 $ for all $ i,j $ and $ \bbV^{[[n]]} = \Ycirc_n $.

We set $ \bbV^\CB_\CS := \bbV_\CS \cap \bbV^\CB $.

\begin{prop} \label{prop:strataYn3}
	\begin{enumerate}
		\item
		$ \bbV^\CB_\CS $ is non-empty if and only if $ \CB $ refines $ \CS $.
		\item
We have $$		\bbV^\CB_\CS \cong F_{r_1}  \times \cdots \times F_{r_m}
$$
where $r_k $ is the number of parts of $ \CB $ contained in $ S_k $.
	\end{enumerate}
\end{prop}

In particular, $ \bbV^{[[n]]}_{{[n]}} = \ft_n \cap \Ycirc_n = F_n $.

\begin{rem} \label{rem:scaling}
	There is a $ \Cx $ action on $ \overline \ft_n $ acting by weight $1$ on each $ \BP^1 $ in the $ \nu $ coordinates (and so with weight $ -1 $ in the $\delta $ coordinates).  The fixed points of this action are labelled by set partitions of $ [n] $; to each set partition $ \CS $, we associate the point $ \delta(S) $, where $ \delta(S)_{ij} = 0 $ if $ i \sim_\CS j $ and $ \delta(S)_{ij} = \infty $ if $ i \nsim_\CS j $.  (These are precisely those points where all the marked points on a given petal are equal.)
	
	The strata $ \bbV^\CB $ and $ \bbV_\CS $ are the attracting and repelling sets for these fixed points with respect to this $ \Cx $ action.
\end{rem}

\subsection{Degeneration of multiplicative group to additive group} \label{se:groupscheme}
Following Zahariuc \cite{Z}, let $ \BG $ be the group scheme over $ \BA^1 $ defined as
$$ \BG = \{(x; \vareps) \in \C^2 : 1 -  \vareps x \ne 0 \}
$$
Multiplication in this group scheme is defined by $$ (x_1; \vareps) (x_2; \vareps) = (x_1 + x_2 - \vareps x_1 x_2; \vareps)  $$
For convenience, we write $ x_1 *_\vareps x_2 := x_1 + x_2 -  \vareps x_1 x_2$.

Note that if we specialize $\vareps \ne 0 $, then $ \BG(\vareps) \cong \Cx$, via the map $ x \mapsto 1 - \vareps x $. On the other hand, $ \BG(0) \cong \C$.  (Here and below, we write $ X(\vareps) := X \times_{\BA^1} \{\vareps \} $ for the fibre of a scheme $ X $ defined over $ \BA^1 $.)

We can realize $ \BG $ as a family of abelian subgroups of $ PGL_2$ as follows
$$
\BG = \left\{ \begin{bmatrix} 1 - \vareps x & x \\ 0 & 1 \end{bmatrix}   : x, \vareps \in \C, 1 - \vareps x \ne 0 \right\} \subset PGL_2
$$
For $ \vareps \ne 0 $, $ \BG(\vareps) $ is the stablizer of points $ \vareps^{-1}, \infty $ in $ \PP^1$.  Alternatively, for any $ \vareps$ we can see that $ A_\vareps$ is the centralizer of $ \begin{bmatrix} -\vareps  & 1 \\ 0 & 0 \end{bmatrix} $ and thus $ \BG $ is the group scheme of regular centralizers in $PGL_2 $ (after base change).

The group scheme $ \BG$ acts on $ \BG^n $ by $$ (x; \vareps) \cdot (x_1, \dots, x_n; \vareps) = (x *_\vareps x_1, \dots, x *_\vareps x_n; \vareps)$$
(Since $ \BG $ is scheme over $ \BA^1 $, when we define $\BG^n $, we form fibre products over $\BA^1$, so $\BG^n = \{(x_1, \dots, x_n; \vareps): 1 - \vareps x_i \ne 0 \text{ for all $ i$ } \} $.)

The quotient $ \BG^n / \BG $ is a scheme over $ \BA^1 $ whose generic fibre is $ (\Cx)^n / \Cx $ and whose special fibre is $ \C^n / \C $.

We also consider $$ \mathcal F_n := (\BG^n \setminus \Delta) / \BG = (x_1, \dots, x_n, \vareps) : x_i \ne x_j, x_i \ne \vareps^{-1} \}$$

This is a scheme over $ \BA^1 $ whose generic fibre is $ M_{n+2} = ((\Cx)^n \setminus \Delta) / \Cx $ and whose special fibre is $ F_n = (\C^n \setminus \Delta) / \C $.

\subsection{Degeneration of Losev-Manin to the flower space}

We define the family $ \overline \Cf_n $ as the subscheme of $ (\nu, \vareps) \in (\BP^1)^{p([n])} \times \C $ defined by
\begin{equation} \label{eq:Cfnu} \vareps \nu_{ik} + \nu_{ij} \nu_{jk} = \nu_{ik} \nu_{jk} + \nu_{ij} \nu_{ik} \qquad \nu_{ij} + \nu_{ji} = \vareps \end{equation}
As before, let $ \delta_{ij} = \nu_{ij}^{-1} $.  In these coordinates, the equations become
\begin{equation} \label{eq:Cfdelta}
\vareps \delta_{ij}\delta_{jk} + \delta_{ik} = \delta_{ij} + \delta_{jk} \qquad \delta_{ij} + \delta_{ji} = \vareps \delta_{ij} \delta_{ji}
\end{equation}

It is clear that the fibre $ \overline \Cf_n(0) $ over $ 0 \in \BA^1 $ is isomorphic to $ \overline \ft_n $.

\begin{rem}
We can rewrite the first equation of (\ref{eq:Cfnu}) without a ``linear'' term by using the relation $ \nu_{ij} = \vareps - \nu_{ji} $ on the right hand side and then simplifying to obtain
$$ \nu_{ij}\nu_{jk} + \nu_{ji}\nu_{ik} = \nu_{ik} \nu_{jk}
$$
\end{rem}

\begin{rem}
    Let $ \fg $ be any semisimple Lie algebra with Cartan subalgebra $ \fh $ and subgroup $ H$.  In this context, Balibanu-Crowley-Li \cite{BCL} have studied a flat degeneration of $ \overline{H}$ (the toric variety associated to the root hyperplane fan) to a scheme which has $ \overline \fh $ (the matroid Schubert variety associated to the root hyperplane arrangement) as the reduced subscheme of one of its irreducible components.  In the particular case $ \fh = \ft_n $ studied in the present paper, the central fibre is reduced and irreducible.  This generalizes our space $ \overline \Cf_n $.  More generally, \cite{BCL} show that for any rational central hyperplane arrangement (with zero and parallel normal vectors allowed) there is a flat degeneration of a toric variety to a scheme such that the reduced subscheme of each of its irreducible components is a matroid Schubert variety.
\end{rem}

\begin{prop} \label{pr:YnLn}
For any $ \vareps \ne 0$, we have an isomorphism $ \overline \Cf_n(\vareps) \cong \overline T_n$ via the identification $ \alpha_{ij} = 1 - \vareps \delta_{ij} = \frac{\nu_{ij} - \vareps}{ \nu_{ij}}$.
\end{prop}

\begin{proof}
	First, we observe that $ (1 - \vareps \delta_{ij})(1 - \vareps \delta_{ji}) = 1 $ is equivalent to $ \delta_{ij} + \delta_{ji} = \vareps \delta_{ij} \delta_{ji} $.
	
	Next note that
	$
	(1-\vareps \delta_{ij})(1-\vareps \delta_{jk}) = 1-\vareps \delta_{ik}
	$
	is equivalent to $ -\vareps \delta_{ij} \delta_{jk} + \delta_{ij} + \delta_{jk} = \delta_{ik} $.
\end{proof}


Let
$$\Cf_n:= \{(\nu, \vareps) \in \overline \Cf_n : \nu_{ij} \ne 0, \vareps \text{ for all $ i,j$} \} \quad  \Cf_n^\circ := \{ (\nu, \vareps) : \nu_{ij} \ne 0, \vareps, \infty \text{ for all $ i,j$} \} $$

Recall the group scheme $ \BG $ defined in Section~\ref{se:groupscheme}. From the isomorphism given in Proposition \ref{pr:YnLn}, the following result is immediate.

\begin{prop}
	There are isomorphisms $\Cf_n \cong \BG^n / \BG $ and $ \Cf_n^\circ \cong (\BG^n \setminus \Delta) / \BG = \mathcal F_n $, defined on coordinates by $\nu_{ij} \mapsto \frac{ 1 - \vareps x_j}{x_i - x_j}$
\end{prop}

The following results show that this family has good properties.
\begin{lem}
	\begin{enumerate}
		\item $\overline \Cf_n $ is flat over $ \C $.
		\item $ \overline \Cf_n $ is reduced.
	\end{enumerate}
\end{lem}
\begin{proof}
	\begin{enumerate}
		\item From Proposition \ref{pr:YnLn}, it is clear that $ \overline \Cf_n(\Cx) $ is isomorphic to $ \overline T_n \times \Cx $ and that $ \overline \Cf_n $ is the scheme theoretic closure of $\overline T_n \times \Cx $ inside $ (\BP^1)^{p([n])} \times \C $.  This implies that it is flat over $ \C $.
		\item Since $ \overline \Cf_n \rightarrow \C $ is a flat family with reduced fibres (see Lemma \ref{le:BBTn} and Theorem \ref{th:ArdilaBoocher}), we conclude that $ \overline \Cf_n $ is reduced.
	\end{enumerate}
\end{proof}

\subsection{Strata and an open cover}
As for $ \overline \ft_n $, for any set partitions $ \CS, \CB$, we can define strata
\begin{gather*}
	\bbCV_\CS = \{ (\nu; \varepsilon) \in \overline \Cf_n: \nu_{ij} \in \{0, \vareps\} \text{ if and only if $i \nsim_\CS j $} \} \\
	\bbCV^\CB = \{ (\nu; \varepsilon) \in \overline \Cf_n: \nu_{ij} = \infty \text{ if and only if $i \sim_\CB j $} \} \\
	\bbCV^\CB_\CS = \bbCV^\CB \cap \bbCV_\CS
\end{gather*}

However, the strata $ \bbCV_\CS $ are not generally irreducible; their irreducible components are described by possible orderings of the parts of $ \CS$.  Thus, given any ordered set partition $ \CS = (S_1,\dots, S_m) $, we define
$$
\vec{\bbCV}_\CS = \{ (\nu,\varepsilon) \in \overline \Cf_n: \nu_{ij} = 0 \text{ if $ i \in S_k, j \in S_l $ and $ k < l $ } \}
$$

These strata can be described as follows
\begin{enumerate}
	\item $ \vec{\bbCV}_\CS $ parameterizes caterpillar curves (for $ \vareps \ne 0$), resp. flower curves (for $\vareps = 0 $), with $ m $ components $ C_1, \dots, C_m $, with marked points labelled $ S_1, \dots, S_m $.
	\item $ \bbCV^\CB $ parameterizes caterpillar curves (for $ \vareps \ne 0$), resp. flower curves (for $\vareps = 0 $), with equal marked points $ z_i = z_j $  if and only if $ i \sim_\CB j $.
\end{enumerate}
\begin{rem} \label{rem:scaling2}
	The $ \Cx $ action on $ \overline \ft_n $ described in Remark \ref{rem:scaling} extends to a $ \Cx $ action on $ \overline \Cf_n $ where $ \Cx $ acts by weight $ 1 $ on the $ \vareps $ coordinate (note that the defining equations are homogeneous in $ \nu $ and $\vareps $).  The fixed points of this action all lie in the $ \vareps = 0 $ fibre and were described in Remark \ref{rem:scaling}.  The strata $ \bbCV^\CB $ are the attracting sets for this action.
\end{rem}

Let $ \CS $ be a set partition of $ [n] $.  We define an open affine subscheme $ \bbCU_\CS $, of $ \overline \Cf_n $ by
\begin{equation} \label{eq:defUS}
\bbCU_\CS = \{ (\nu, \vareps) \in \overline \Cf_n: \nu_{ij} \ne \infty \text{ if $ i \nsim_\CS j $}, \ \nu_{ij} \ne 0, \vareps \text{ if $i \sim_\CS j) $}  \}
\end{equation}
This contains the stratum $ \bbCV_\CS $.  Moreover, $ \bbCU_\CS $ parameterizes caterpillar curves (for $ \vareps \ne 0$), resp. flower curves (for $\vareps = 0 $), such that $ z_i \ne z_j $ if $ i \nsim_\CS j $, and $ z_i $ and $ z_j $ are on the same component if $ i \sim_S j $.

\begin{eg} \label{eg:bbCU}
	Suppose that $ m = 1$ and thus $ S_1 = [n] $.  This open subset corresponds to the locus where there is a unique component of the caterpillar curve (when $ \vareps \ne 0 $) or a unique petal of the flower (when $ \vareps = 0 $).  We have $ \bbCU_{\{[n]\}} = \Cf_n \cong \BG^n / \BG$.
	
	Suppose that $ m = n $, and thus $ S_k = \{k \}$.  In this case $ \Cfcirc_n := \bbCU_{[[n]]} $ is the locus of marked curves $ (C, \uz) $ where all the $ z_i $ are distinct.  This is a singular affine scheme.  The special fibre of $ \Cfcirc_n \rightarrow \BA^1 $ is the reciprocal plane $ \Ycirc_n $; it seems an interesting problem to study the general fibre of this family.
\end{eg}

Since every stratum is contained in an open set, the following is immediate.
\begin{prop}
	These open sets $ \bbCU_\CS $ cover $ \overline \Cf_n $.
\end{prop}

We can partially describe these open sets in the following way.  Fix a set partition $ \CS = \{S_1, \dots, S_m \} $.  For each part $ S_k $ of $\CS $, fix some $ i_k \in S_k$.

\begin{prop} \label{le:CUCSembed}
	The map \begin{align*} \bbCU_\CS &\rightarrow \prod_k \Cf_{S_k} \times \overline \Cf_m^\circ  \\
		\nu &\mapsto ((\nu|_{p(S_1)}, \dots, \nu|_{p(S_m)}), \nu|_{p(\{i_1, \dots, i_m\})})
	\end{align*}
is an open embedding.
\end{prop}	

We begin with the following elementary lemma.

\begin{lem} \label{le:xyz}
	Let $ x,y, \vareps \in \C $ with $ y \vareps \ne 1$.  Then the equation
	$$
	y \vareps z + x = x yz +z
	$$
	has at most one solution for $ z \in \C $.
\end{lem}
\begin{proof}
	We find $ x = (xy + 1 - \vareps y) z$.  The only way this can fail to have at most one solution for $ z $ is if $ x = 0 $ and $ xy +1 - \vareps y = 0 $.  But this is impossible, since $ y \vareps \ne 1 $.
\end{proof}

\begin{proof}[Proof of Prop \ref{le:CUCSembed}]
Let $ \nu \in \prod_k \Cf_{S_k} \times \overline \Cf_m^\circ  $.  We must prove that there is at most one way to extend $ \nu $ to a point of $ \CU_\CS $.

We already have the data of $ \nu_{ij} $ for $ ij \in p(S_k) $ for some $k$, and the data of $ \nu_{i_k i_l} $.  So we are missing the data of $ \nu_{ab} $ where $ a \in S_k, b \in S_l $ with $ k \ne l $, but at least one of $ a, b $ is not the chosen elements $ i_k, i_l$.

Suppose for the moment that $ a = i_k $, but $ b\ne i_l $. Then the defining equation of $ \overline \Cf_m $ gives
$$
\vareps \nu_{i_k b} + \nu_{i_k i_l} \nu_{i_l b} = \nu_{i_k b} \nu_{i_k i_l} + \nu_{i_k b} \nu_{i_l b}
$$
We already have the values of $ \nu_{i_k i_l}  $ and $ \nu_{i_l b} $ (which lies in $ \BP^1 \setminus \{0, \vareps \} $) and wish to determine $ \nu_{i_k b} $.  Setting $ y = \nu_{i_l b}^{-1} $ and $ x = \nu_{i_k i_l} $ brings us to the situation of Lemma \ref{le:xyz}, with $ z = \nu_{i_k b} $.

For the general case of $ a \ne i_k, b \ne i_l $, we can proceed similarly, using the fact that we have already determined the value of $ \nu_{i_k b} $.
\end{proof}

\section{Line bundle on the Deligne-Mumford spaces}
\subsection{Deligne-Mumford space}
We now consider the usual Deligne-Mumford moduli space of points on $ \BP^1$.

We write $ \overline M_{n+1} $ for the moduli space of genus 0, stable, nodal curves with $ n+1 $ marked points, denoted $ (C, \uz) $.  We have the open locus $$ M_{n+1} = ((\BP^1)^{n+1} \setminus \Delta) / PGL_2 = ((\Cx)^{n-1} \setminus \Delta) / \Cx $$

For $ z \in  (\BP^1)^{n+1} \setminus \Delta $ and four distinct indices $ i, j, k, l \in \{1, \dots, n+1 \}$, the cross ratio
$$
\frac{(z_i - z_k)(z_l - z_j)}{(z_i - z_j)(z_l - z_k)}
$$
is a well-defined function on $M_{n+1} = (\BP^1)^{n+1} \setminus \Delta / PGL_2$.  Even if one of the points is $ \infty$, the cross ratio still makes sense.  For example if $ z_l = \infty$, we get
$$
\mu_{ijk} := \frac{z_i - z_k}{z_i - z_j}
$$

The cross ratio extends to a map $ \overline M_{n+1} \rightarrow \BP^1$. We will use these cross ratios to embed $ \overline M_{n+1} $ inside a product of projective lines.  Rather than working with all the cross ratios, we will always take $ l = n+1$.  Usually, we will choose $z_{n+1} = \infty $, so that the cross ratio will reduce to the above simple ratio.

 The following result is due to Aguirre-Felder-Veselov \cite[Theorem A.2]{AFV}, building on earlier work of Gerritzen-Herrlich-van der Put \cite{GHP}.
\begin{thm} \label{th:embedP1}
	The maps $ \mu_{ijk} $ embed $ \overline M_{n+1} $ as the subscheme of $ (\BP^1)^{t([n])} $ defined by
	\begin{align*}
		\mu_{ijk} \mu_{ikj} = 1 \quad \mu_{ijk} + \mu_{jik} = 1  \quad \mu_{ijk}\mu_{ilj} = \mu_{ilk}
	\end{align*}
for distinct $ i, j, k, l$.
\end{thm}

\begin{rem}
	More generally, for any finite set $ S$, we can consider the moduli space $ \overline M_{S+1} $ where the points are labelled by $ S \sqcup \{0\}$.  Theorem \ref{th:embedP1} then gives an embedding of $ \overline M_{S +1} $ into $ (\BP^1)^{t(S)} $.
\end{rem}

Every point of $ \overline M_{n+1} $ defines a $[n] $-labelled rooted tree. More precisely, given a genus 0, stable, nodal curve $C$ with $ n+1 $ marked points, we consider the tree whose internal vertices are the irreducible components of $ C $ and whose leaves are the marked points of $ C$ other than $ z_{n+1} $.  We use the marked point $ z_{n+1} $ as the root.

Thus we have a stratification of $ \overline M_{n+1} $ indexed by $[n]$-labelled trees.  For example, $ M_{n+1} $ is the stratum indexed by the unique tree with one internal vertex.

\begin{rem} \label{rem:dCP1}
	Given any hyperplane arrangement and a building set, de Concini-Procesi \cite{dCP} constructed a wonderful compactification of the projectivization of the complement of the hyperplane arrangement.  For the braid arrangement and the minimal building set, the de Concini-Procesi wonderful compactification is $ \overline M_{n+1} $.  From their construction, we obtain an embedding $ \overline M_{n+1} \rightarrow \prod \BP(\C^n / E_B ) $ where the product ranges over subsets  $ B \subset [n] $ of size at least 3, where
	$$ E_B = \{ \uz \in \C^n : z_i = z_j \text{ for all $ i, j \in B$ } \} $$
 From this perspective, the map $ \mu_{ijk} : \overline M_{n+1} \rightarrow \BP^1 $ can be regarded as the projection onto the factor $ \BP(\C^n / E_{\{i,j,k\}}) $.
\end{rem}

\subsection{Morphism to Losev-Manin space}
We now consider $ \overline M_{n+2}$, with the marked points $ z_0, z_{n+1} $ distinguished. Given a point $(C, \uz)\in \overline M_{n+2} $, we can collapse the curve $ C $ in a unique way to a caterpillar curve with $n$ marked points $ (C', \uz')$.  More precisely, we let $ C' $ be the union of those components of $ C $ along the unique path between the component containing $ z_0 $ and the one containing $ z_{n+1}$.  We then define $ z'_k $ to be the point on $ C' $ closest to $ z_k $, for $ k = 1, \dots, n $.

The map $ (C, \uz) \mapsto (C', \uz') $ defines a morphism $ \overline M_{n+2} \rightarrow \overline T_n $.   In coordinates, this morphism is given by $ \alpha_{ij} = \mu_{n+1 \, i j}$.

\begin{eg}
	Here is an example of the morphism $ \overline M_{6+2} \rightarrow \overline T_6 $.
$$
	\begin{tikzpicture}[scale=0.7]
		\draw (1,0) ellipse [x radius=1, y radius=0.8];
			\draw (3,0) ellipse [x radius=1, y radius=0.8];
			\draw (5,0) ellipse [x radius=1, y radius=0.8];
		\filldraw  (0,0) circle [radius=2pt] node[left] {$z_0$};
		\filldraw (1,-0.8) circle [radius=2pt] node[below] {$z_4$};
		\draw (1, 1.4) circle [radius=0.6] ;
		\draw ({1 + sqrt(2)/2}, {1.4 + sqrt(2)/2}) circle [radius=0.4];
		\filldraw  (0.4,1.4) circle [radius=2pt] node[left] {$z_1$};
		\filldraw ({1 + sqrt(2)/2},  {1.4 + sqrt(2)/2 + 0.4}) circle [radius=2pt] node[above] {$z_2$};
			\filldraw ({1 + sqrt(2)/2 + 0.4},  {1.4 + sqrt(2)/2}) circle [radius=2pt] node[right] {$z_3$};
		\filldraw (3,-0.8) circle [radius=2pt] node[below] {$z_5$};
		\filldraw (5,-0.8) circle [radius=2pt] node[below] {$z_6$};
		\filldraw (6, 0) circle [radius=2pt] node[right] {$z_7$};
\draw[|->] (7.5,0) to (8.5,0);
\begin{scope}[xshift=290]
	\draw (1,0) ellipse [x radius=1, y radius=0.8];
\draw (3,0) ellipse [x radius=1, y radius=0.8];
\draw (5,0) ellipse [x radius=1, y radius=0.8];
\filldraw  (0,0) circle [radius=2pt] node[left] {$z_0$};
\filldraw (1,-0.8) circle [radius=2pt] node[below] {$z_4$};
\filldraw (1, 0.8) circle [radius=2pt] node[above] {$z_1 z_2 z_3$} ;
\filldraw (3,-0.8) circle [radius=2pt] node[below] {$z_5$};
\filldraw (5,-0.8) circle [radius=2pt] node[below] {$z_6$};
\filldraw (6, 0) circle [radius=2pt] node[right] {$z_7$};
\end{scope}
\end{tikzpicture}
$$
\end{eg}
		
\begin{eg} \label{eg:M5L3}
	Suppose that $ n =3$, so we are considering the map $ \overline M_5 \rightarrow \overline T_3 $.  This map is 1-1 except over the point $ \{ \alpha_{ij} = 1 \} \in \overline T_3 $ (the curve with marked points $ z_1 = z_2 = z_3$).  The fibre over this point is $ \overline M_4 = \BP^1 $.  Over this point, the morphism $ \overline M_5 \rightarrow \overline T_3 $ maps a curve with 2 or 3 components to one with a single component.  (Collapsing also occurs at other points where two of the $ z_i $ coincide, but this does not give non-trivial fibres to the morphism since $ \overline M_3 $ is a point.)
\end{eg}

The following result is immediate from the definition of this morphism.
\begin{prop} \label{pr:collapse}
 Let $ (C, \uz) \in \overline T_n \cap \bbCV^\CB$.  (So $ z_i = z_j $ iff $ i \sim_\CB j $.)  The fibre of $ \overline M_{n+2} \rightarrow \overline T_n $ over $(C, \uz) $ is $ \overline M_{B_1 + 1} \times \cdots \times \overline M_{B_r +1}$.
\end{prop}

\subsection{A line bundle} \label{se:linebundle}
There is a natural line bundle $ \tM_{n+1}$ over $ \overline M_{n+1} $, defined as follows.  We consider the morphism $ \overline M_{n+1} \rightarrow  \BP(\ft_n) $ given by collapsing a curve $ (C, \uz) $ to the component containing $ z_{n+1} $, identifying $ z_{n+1} = \infty $, and remembering the positions of all the other marked points (recall that $ \ft_n = \C^n/\C$).  We define $ \tM_{n+1} $ by pulling back $ \CO(-1) $ along this morphism.

As $ \CO(-1) $ is the tautological line bundle over $ \BP(\ft_n) $, it comes equipped with a morphism to $ \ft_n $ and thus by pullback, we have a morphism $ \gamma:\tM_{n+1} \rightarrow \ft_n  $.  In particular, this means that $ z_i - z_j $ are well-defined functions on $ \tM_{n+1} $ for any $ ij \in p([n]) $.

$$
\begin{tikzcd}
	\tM_{n+1} \arrow{r} \arrow{d} & \CO(-1) \arrow{d} \arrow{r} & \ft_n \\
	\overline M_{n+1} \arrow{r} & \BP(\ft_n)
\end{tikzcd}
$$

We now study the fibres of the map $ \tM_{n+1} \rightarrow \ft_n $.   Let $ \uz \in \ft_n$.  This point determines a partition $ [n] = B_1 \sqcup \cdots \sqcup B_r $, where $ i, j \in B_l$ for some $ l $ if and only if $ z_i = z_j $; equivalently we have a point $ \nu \in \bbV^\CB \cap \ft_n $.

\begin{prop} \label{pr:fibrecircY}
	There is an isomorphism $ \gamma^{-1}(\uz) \cong \overline M_{B_1 +1} \times \cdots \times \overline  M_{B_r +1} $.

We obtain a stratification of $ \tM_{n+1} $ with strata
$$
 \gamma^{-1}(V^\CB \cap \ft_n) = \overline M_{B_1 +1} \times \cdots \times \overline M_{B_r +1} \times F_{r} $$
\end{prop}

\begin{rem} \label{rem:tMparam}
	Combining the above proposition with Remark \ref{rem:vector}, we see that $ \tM_{n+1} $ parametrizes genus 0 stable nodal curves $ C$, carrying $ n $ distinct marked points $ z_1, \dots, z_n$, one distinguished point $ z_{n+1}$, and a non-zero tangent vector at the distinguished point $ a\in T_{z_{n+1}} C $.  From this point of view, the zero section of the line bundle $ \tM_{n+1} \rightarrow \overline M_{n+1} $ corresponds to the locus where the component of $ C $ containing $ z_{n+1} $ contains no other marked points (this corresponds to $ r = 1$ in the above proposition).
\end{rem}

\begin{proof}
	We are studying the fibres of $ \tM_{n+1} = \overline M_{n+1} \times_{\BP(\ft_n)} \CO(-1) \rightarrow \ft_n$.  There are two cases, since $ \CO(-1) \rightarrow \ft_n $ has one exceptional fibre.
	
	If $ \uz = 0 $ (which is equivalent to $\CB = \{[n]\} $), then the fibre of $ \CO(-1) \rightarrow \ft_n $ over $ \uz $ is $ \BP(\ft_n) $.  Thus the fibre of $ \tM_{n+1} \rightarrow \C^n/ \C$ over $ \uz $ is $ \overline M_{n+1} $, as desired.
	
	If $ \uz \ne 0 $, then the fibre of $ \CO(-1) \rightarrow \ft_n $ over $ \uz $ is a single point, and thus the fibre of $ \tM_{n+1} \rightarrow \ft_n $ over $ \uz $ is the same as the fibre of $ \tM_{n+1} \rightarrow \BP(\ft_n) $ over $ [\uz] $.  As this morphism is collapsing components, we deduce the desired result.
\end{proof}

\begin{rem}
In Remark \ref{rem:dCP1}, we explained that $ \overline M_{n+1} $ is the closure of $ M_{n+1} $ in a product of projective spaces, a special case of the construction of de Concini-Procesi.  Following their work \cite[\S 1.1 and 4.1]{dCP}, we can see that $ \tM_{n+1} $ is the closure of $ M_{n+1} $ in the product $ \ft_n \times \prod \BP(\C^n / E_B ) $ where the second factor is the same as in Remark \ref{rem:dCP1}.
\end{rem}

The stratification of $ \overline M_{n+1} $ by rooted trees can be extended to $ \tM_{n+1} $ as follows.  We define a \textbf{bushy rooted tree} to be the same as a rooted tree, except that we allow the root to be contained in more than one edge (have degree greater than 1).  To each point $ C \in \tM_{n+1} $ we associate a bushy rooted tree as follows.  First, under the map $ \tM_{n+1} \rightarrow \ft_n $, we obtain a point in $ V^\CB $ for some set partition $ \CB $.  Then using Proposition \ref{pr:fibrecircY}, we obtain a point in $ \overline M_{B_1 +1} \times \cdots \overline M_{B_r +1} $, which gives us rooted trees $ \tau_1, \dots, \tau_r $.  We glue these trees together at their roots to obtain a bushy rooted tree.

Alternatively, following Remark \ref{rem:tMparam}, we regard $ C $ as $ (C, \uz, a) $ where $ C $ is a genus 0 nodal curve, $ \uz $ is a collection of marked points, and $ a \in T_{z_{n+1}} C $ is a non-zero tangent vector.  We then consider the component graph of $ C $ where the root labels the component containing $ z_{n+1} $ (rather than $z_{n+1} $ itself).

\subsection{A deformation of the line bundle} \label{se:deformline}
We will now define a deformation $ \tCM_{n+1} $ of $ \tM_{n+1} $ which will map to $ \overline \Cf_n $.  Intuitively, we will consider curves with $ n+2 $ marked points and bring two points together.  Here is a more precise definition.

To begin, let $ B = \Cx \ltimes \C $ be the Borel subgroup of $ PGL_2 $ and note that  $ \BP(\ft_n) = \C^n \setminus \C(1, \dots, 1) / B $.

Define an action of $ B $ on $ \C^n  \times \BP^1 \times \BA^1 $ by
\begin{equation} \label{eq:actB}
(s,u) \cdot (\uu, y; \vareps) = ((s u_i + u), sy - \vareps u; \vareps)
\end{equation}

It is not hard to see that the action of $ B $ preserves the set $$(\C^n \setminus \C \times \BP^1 \times \BA^1)^\circ := \{ (\uu, y, \vareps) : y \ne -\vareps u_i \text{ for all } i \}$$
and we let
$$\CE_{n+1} := (\C^n \setminus \C \times \BP^1 \times \BA^1)^\circ / B $$
To be more precise, we consider the ring
$$
R = \C[u_1, \dots, u_n, (y + \vareps u_1)^{-1}, \dots, (y + \vareps u_n)^{-1}, \vareps] $$
regarded as a subalgebra of $ \C[u_1, \dots, u_n, \vareps](y) $.  Then $ R $ carries an action of $ B $, given by (\ref{eq:actB}), and we define a graded algebra $ S = \oplus_{m \in \BN} S_m  $, by
$$ S_m = \{ r \in R : br = \chi(b)^m r \text{ for all $ b \in B $ } \}$$
where $ \chi : B \rightarrow \Cx $ is given by $ (s,u) \mapsto s $.  Then we define $ \CE_{n+1} = \operatorname{Proj} S $.

There is an obvious morphism $ \CE_{n+1} \rightarrow \BP(\ft_n) $ defined by $ (\uu, y, \vareps) \mapsto \uu $.

There is a less-obvious morphism $\CE_{n+1} \rightarrow \BG^n / \BG \cong \Cf_n $ defined by $(\uu, y, \vareps) \mapsto (\frac{u_i}{y + \vareps u_i}; \vareps) $ or in terms of $\delta$-coordinates, by $ \delta_{ij} = \frac{u_i - u_j}{y + \vareps u_i } \in S_0$.  This map is proper birational and nearly an isomorphism.  Its only exceptional fibres are over the identity section $ \BA^1 \subset \BG^n / \BG $, where the fibre is $ \BP(\ft_n) $.

\begin{lem}
	This map realizes $ \CE_{n+1} $ as the blowup of $ \Cf_n $ along the identity section.
\end{lem}

\begin{proof}
	By definition $ \C[\Cf_n] $ is the quotient of $ \C[\delta_{ij}] $ by the ideal generated by equations (\ref{eq:Cfdelta}).  Let $ I \subset \C[\Cf_n] $ be the ideal generated by all $ \delta_{ij} $; this is precisely the ideal of the zero section.  Hence $ \operatorname{Proj} \oplus_{m \in \mathbb N} I^m $ is the blowup of $ \Cf_n $ along the zero section.
		
	We claim that $  \oplus_{m \in \mathbb N} I^m \cong S $ as graded algebras.
	 To begin with there is an isomorphism $  \C[\Cf_n] \cong S_0 $ given by sending $ \delta_{ij} $ to $  \frac{u_i - u_j}{y + \vareps u_i} $ as above.  Next, there is an isomorphism $ I \cong S_1 $ as $S_0 $-modules; this isomorphism takes $ \delta_{ij} $ to $ u_i - u_j $.  From here the result follows.
\end{proof}

Note that the $ \vareps = 0 $ fibre, $ \CE_{n+1}(0)$, is isomorphic to  $ \CO(-1) $ over $ \BP(\ft_n) $, where the map is given by $ (\uu,y) \mapsto (y^{-1} u_i) $ (here  $ (\uu, y) \in \C^n \times \BP^1 \setminus \{0 \} $).

Then we define $ \tCM_{n+1} :=  \overline M_{n+1} \times_{\BP(\ft_n)} \CE_{n+1} $.  By construction, we have $ \tCM_{n+1} \rightarrow \overline M_{n+1} $ and $ \gamma: \tCM_{n+1} \rightarrow \Cf_n $.

\begin{rem} \label{rem:strangemap}
	The formula above for the map $ \CE_{n+1} \rightarrow \Cf_n $ looks a bit strange.  In order to understand it, take $ \vareps \ne 0 $, and consider the composition
	$$ \CE_{n+1}(\vareps) \rightarrow \Cf_n(\vareps) \cong (\Cx)^n / \Cx \rightarrow (\BP^1)^{n+2} / PGL_2 $$  The resulting map is $$ (\uu, y, \vareps) \mapsto (\frac{u_i}{y + \vareps u_i}; \vareps) \mapsto (0, \frac{y}{y + \vareps u_i}, \infty) = (\infty , \uu , -\vareps^{-1} y) $$
	where we use the isomorphism $ \BG(\vareps) \cong \Cx$, and where the last equality in an equality in $ (\BP^1)^{n+2} / PGL_2 $.
\end{rem}

\begin{prop}
	For $ \vareps \ne 0 $, we have $ \tCM_{n+1}(\vareps) \cong \overline M_{n+2}^\circ$, the open subset of $ \overline M_{n+2} $ such that the marked points $ z_0, z_{n+1} $ lie on the same component, and the morphism $ \tCM_{n+1}(\vareps) \rightarrow \overline M_{n+1} $ is identified with the restriction of the universal curve morphism.
\end{prop}

\begin{proof}
	We define the map $ \tCM_{n+1}(\vareps) \rightarrow \overline M_{n+2}^\circ $ as follows.  A point of $ \tCM_{n+1}(\vareps) $ is a pair $(C, \uz), (\uu, y, \vareps) $ mapping to the same point in $ (z_1, \dots, z_n) = \uu \in \BP(\ft_n) $.  Assume that $y \ne \infty $.  Then by Remark \ref{rem:strangemap}, the additional data of $ y $ gives us a new point $ z_{n+1} = -\vareps^{-1}y $ on the component of $ C $ containing $ z_0 = \infty $.  By the open condition in the definition of $ \CE_{n+1} $, this new point is distinct from $ z_0 $ and all the $ z_i $.   If $ y = \infty $, we need to do something a bit different.  We take the curve $(C,\uz) $ and we attach a new $ \BP^1$ component at the marked point $ z_{n+1}$.  Then we place $ z_0 $ and $ z_{n+1} $ on this new component.  In both cases, we get a point of $ \overline M_{n+2}^\circ$.
	
	Conversely, given a point $ (C, \uz) \in \overline M_{n+2}^\circ $, we can forget $ z_{n+1} $ to get a point $ (C, z_0, \dots, z_n) \in \overline M_{n+1} $.  Or we can collapse $ C $ to the component containing $ z_0 $ and $ z_{n+1} $ to get a point of $ (\Cx)^n/ \Cx $, by identifying $z_0 = 0, z_{n+1} = \infty $ and recording the images of $ z_1, \dots, z_n $.  In this way, we can produce a point in the fibre product $ \tCM_{n+1}(\vareps) = \overline M_{n+1} \times_{\BP(\ft_n)} \CE_{n+1}(\vareps) $.
\end{proof}
\section{Mau-Woodward space}

In \cite{MW}, Mau-Woodward introduced a space $ Q_n $ by explicit equations in the product of projective lines, following earlier work by Ziltener \cite{Zil}.  This space is a compactification of $ F_n $, but is not our desired space $ \overline F_n $, as we will see below.

\subsection{Mau-Woodward space} \label{se:MWdef}

Following \cite{MW}, we define the Mau-Woodward space $Q_n$ to be the subscheme of $ (\nu, \mu) \in (\PP^1)^{p([n])} \times (\PP^1)^{t([n])}$ defined by the equations
\begin{gather*}
	 \mu_{ijk} \mu_{ikj} = 1 \quad
	 \mu_{ijk} + \mu_{jik} = 1 \quad
	 \mu_{ijk} \mu_{ikl} = \mu_{ijl} \\
	  \mu_{ijk} \nu_{ik} = \nu_{ij} \quad
	 \nu_{ij} + \nu_{ji} = 0\quad
	 \nu_{ij} \nu_{jk} = \nu_{ik} \nu_{jk} + \nu_{ij} \nu_{ik}
\end{gather*}
for distinct $ i,j,k,l$.



%
%

\begin{eg} \label{eg:Q3}
	Take $ n = 3$.  Set $ a = \nu_{23}, b = \nu_{13}, c = \nu_{12}, \mu = \mu_{123}$. We have
	$$
	\mu b = c \quad  a (\mu -1) = c \quad a(\mu - 1) = \mu b \quad ab + bc = ac
	$$
	
	There is a map $ Q_3 \rightarrow \overline \ft_3 := \{(a,b,c) \in (\BP^1)^3 : ab + bc = ac \} $. This map has $ \BP^1 $ fibres over $ (\infty, \infty, \infty) $ and $(0,0,0) $ and point fibres elsewhere.  Compare with Example \ref{eg:M5L3}, where there was only one special fibre for $ \overline M_5 \rightarrow \overline T_3$.
\end{eg}

\subsection{Deformation of Mau-Woodward}

Now, we define the total space of the deformation
$\mathcal Q_n $ to be the subscheme of $ (\nu, \mu, \vareps) \in (\PP^1)^{p([n])} \times (\PP^1)^{t([n])} \times \BA^1 $ defined by
\begin{equation} \label{eq:defCQn}
\begin{gathered}
	\mu_{ijk} \mu_{ikj} = 1 \quad
	\mu_{ijk} + \mu_{jik} = 1 \quad
	\mu_{ijk} \mu_{ikl} = \mu_{ijl} \\
	\mu_{ijk} \nu_{ik} = \nu_{ij} \quad
	\nu_{ij} + \nu_{ji} = \vareps  \quad
	\vareps \nu_{ik} + \nu_{ij} \nu_{jk} = \nu_{ik} \nu_{jk} + \nu_{ij} \nu_{ik}
\end{gathered}
\end{equation}

It is easy to see that $(\nu, \mu) \mapsto \nu$ defines a map $ \CQ_n \rightarrow \overline \Cf_n $ so this deformation sits over the previous one.
%

Clearly, we have $ \CQ_n(0) = Q_n $, while the general fibre is the Deligne-Mumford space.
\begin{prop}
	For $ \vareps \ne 0$,  we have an isomorphism $ \CQ_n(\vareps) \cong \overline M_{n+2} $ given by setting $ \mu_{ij \, n+1} =  \vareps^{-1} \nu_{ij}$.
\end{prop}
\begin{proof}
Under this change of coordinates, the defining equations (\ref{eq:defCQn}) of $ \CQ_n(\vareps) $ become the equations from Theorem \ref{th:embedP1}.

	The only non-trivial calculation is that the final equation of (\ref{eq:defCQn}) becomes
	 $$ -\mu_{jl \, n+1} +  \mu_{jk \, n+1} \mu_{kl \, n+1} =  \mu_{jl \, n+1} \mu_{kl \, n+1} +  \mu_{jk \, n+1} \mu_{jl \, n+1}$$ which is equivalent to
	$$ (1 - \mu_{jk \, n+1}^{-1})(1- \mu_{kl \, n+1}^{-1}) = 1 - \mu^{-1}_{jl \, n+1} $$
	which is then equivalent to $\mu_{n+1 \, j k} \mu_{n+1 \, k l} = \mu_{n+1 \, j l}$ as desired.
\end{proof}


We now study the open subset of this deformation $$ \circCQ_n := \{ (\nu, \mu, \vareps) : \nu_{ij} \ne 0, \vareps \text{ for all } i,j \} $$  It is nothing but the degeneration of the universal curve to the line bundle from Section~\ref{se:deformline}.


\begin{thm} \label{th:QtM}
	There is an isomorphism $ \circCQ_n \cong \tCM_{n+1} $ compatible with their maps to $ \Cf_n $.
\end{thm}

\begin{proof}
	Recall that $  \tCM_{n+1} = \CE_{n+1} \times_{\BP(\ft_n)} \overline M_{n+1} $ and $ \CE_{n+1} = (\C^n \setminus \C \times \BP^1 \times \BA^1)^\circ / B $ (as defined in Section~\ref{se:deformline}).
	
	We define $ \tCM_{n+1} \rightarrow \circCQ_n $ by $ ((\uu, y, \vareps), \mu) \mapsto (\nu, \mu, \vareps) $ where as before $ \nu_{ij} = \frac{\vareps u_i - y}{u_i - u_j} $.
 	
	The inverse map $ \circCQ_n \rightarrow \tCM_{n+1} = \CE_{n+1} \times_{\BP(\ft_n)} \overline M_{n+1} $ is given by $ (\nu, \mu) \mapsto ((\uu, y, \vareps), \mu) $ where $ u_1 = 0, y = 1, u_i = \nu_{1 i}^{-1} $.
	
	It is easy to see that these maps are inverses.
\end{proof}

In particular, over the open locus $ \mathcal F_n = \BG^n \setminus \Delta / \BG$, we have $$ \mu_{ijk} = \frac{x_i - x_k}{x_i - x_j} \qquad \nu_{ij} = \frac{1 - \vareps x_j}{x_i - x_j} $$ where $ (x_1, \dots, x_n, \vareps) \in \BG^n \setminus \Delta / \BG $ and $ \mu, \nu $ are the coordinates on $ \CQ_n $.

\subsection{Strata in the Mau-Woodward space}
Let $ \CS $ be a set partition of $[n]$.  Recall the subset $ \bbV_\CS \subset \overline \ft_n $.  Its preimage in $ Q_n $ will be denoted $ \dbtilde{V}_\CS $, so
$$
\dbtilde{V}_\CS = \{(\nu, \mu) :   i \sim_\CS j \text{ if and only if } \nu_{ij} \ne 0 \}
$$
Let $ Q_n^\circ = \dbtilde{V}_{[n]}$ be the locus where none of the $ \nu_{ij} $ vanish.
 \begin{prop} \label{pr:strataMW}
 There is an isomorphism $\dbtilde{V}_\CS \cong \circQ_{S_1} \times \cdots \circQ_{S_m} \times \overline M_{m+1} $ given by
 $$ (\nu, \mu) \mapsto ((\nu, \mu)|_{S_1}, \dots, (\nu,\mu)|_{S_m}, \mu|_{\{i_1, \dots, i_m\}})$$
 This isomorphism is compatible with the isomorphism $ \bbV_\CS \cong \circY_{S_1} \times \cdots \times \circY_{S_m} $ given in Proposition \ref{prop:strataYn}.
\end{prop}
(Here as in Lemma \ref{le:CUCSembed}, $i_k \in S_k $ is a fixed choice.  We also slighly abuse notation by writing $ \mu|_S $ for the restriction of $ \mu $ to triples lying in $ S $.)  This will not be used in what follows, so we omit the proof.

\begin{eg}
The fibre of $ Q_n \rightarrow \overline \ft_n $ over the maximal flower point is $\overline M_{n+1}$.  In this case $ \CS = [[n]] $ and each $ S_k $ is of size $ 1$.
\end{eg}

\begin{rem} \label{rem:Qnparam}
From this proposition, a point $ Q_n $ gives the data of $ (C_1, \uz^1,\dots, C_m, \uz^m, \tilde C) $ where $ (C_r, \uz^r) \in \circQ_{S_r} $ and $ \tilde C \in \overline M_{m+1} $.  In other words, we have a collection $ C_1, \dots C_m $ stable nodal curves, each carrying collection distinct marked points $ \uz^1, \dots, \uz^m $, and a tangent vector at one distinguished point, along with another curve $ \tilde C $ with $m+1 $ marked points which is used to bind together these curves.  This differs from the description of Remark \ref{rem:Ynparam} in two ways: the curves $C_r $ carry distinct marked points and can have multiple components, and there is the extra data of $ \tilde C $.  We will define the space $ \overline F_n $ so that we keep the all the data of the $ C_r$, but forget $ \tilde C $.
\end{rem}

We can use the above results to completely describe the fibres of the map $ Q_n \rightarrow \overline \ft_n $.

Consider a point $ \nu \in \overline \ft_n $. Assume that it lies in the stratum $ \bbV_\CS^\CB $ where $ \CS, \CB $ are two set partitions of $ [n] $ and $ \CB $ refines $ \CS $.  Let $ m $ be the number of parts of $ \CS $.

\begin{cor} \label{cor:fibreQ}
   The fibre of $ Q_n \rightarrow \overline \ft_n $ over $ \nu \in \bbV_\CS^\CB  $ is isomorphic to
	$$
	\overline M_{B_1 +1} \times \cdots \times \overline M_{B_r + 1} \times \overline M_{m+1}
	$$
\end{cor}

\begin{proof}
This follows by combining Propositions \ref{pr:fibrecircY} and \ref{pr:strataMW}.
\end{proof}

\section{The cactus flower space}

\subsection{The open cover} \label{se:defFn}

If we compare Proposition \ref{pr:collapse} and Corollary \ref{cor:fibreQ} which describe the fibres of $ \CQ_n \rightarrow \overline \Cf_n$, we see that in the latter result there is an extra factor of $ \overline M_{m+1} $. We will now define a new space $ \CF_n$, intermediate between $ \CQ_n $ and $ \overline \Cf_n $, in order to correct this defect.

Recall the open cover $ \bbCU_\CS $ of $ \overline \Cf_n $, indexed by set partitions $ \CS $ of $ [n]$.  We define $ \dbtilde{\CU}_\CS \subset \CQ_n $ to be its preimage under the map $ \CQ_n \rightarrow \overline \Cf_n $.

We define a morphism
\begin{align*}
	\bbCU_\CS &\rightarrow \prod_{k=1}^m \circCY_{S_k} \\
 \nu &\mapsto \nu|_{p(S_1)} \times \cdots \times \nu|_{p(S_m)}
\end{align*}

Let $ \bCU_\CS := \bbCU_\CS \times_{\prod_k \circCY_{S_k}} \prod_k \tCM_{S_k +1}$.  These schemes will be the building blocks for our new space $ \CF_n $.

More explicitly, $ \bCU_\CS $ is the subscheme of $ (\nu, \mu^1, \dots, \mu^m, \vareps) \in (\BP^1)^{p([n])} \times \prod_k (\BP^1)^{t(S_k)} \times \BA^1 $ such that
\begin{itemize}
	\item $ \nu_{ij} $ satisfy the ``non-vanishing'' conditions given in (\ref{eq:defUS}) and
	\item all the equations (\ref{eq:defCQn}) in the definition of $ \CQ_n $ hold, whenever they make sense.
\end{itemize}

\begin{lem} \label{le:bCUCSembed}
	$ \bCU_\CS$ is an open subscheme of $ \prod_k \tCM_{S_k +1} \times \overline \Cf_m^\circ $.  In particular it is reduced, and $ \mathcal F_n $ is dense in $ \bCU $.
\end{lem}
\begin{proof}
	By Lemma \ref{le:CUCSembed}, we have an embedding
	$$
	\CU_\CS \hookrightarrow \prod_k \Cf_{S_k} \times \overline \Cf_m^\circ $$
	Applying the fibre product in the definition of $ \bCU_\CS $ gives that $ \bCU_\CS $ is open.
	
	Since it is an open subscheme of a reduced scheme, it is automatically reduced.  And since $ \mathcal F_n $ is dense in $ \prod_k \tCM_{S_k +1} \times \overline \Cf_m^\circ $, we see that $ \mathcal F_n $ is dense in $ \bCU_\CS $.
\end{proof}

There is an obvious morphism $ \gamma:\bCU_\CS \rightarrow \bbCU_\CS $ and we can also define a morphism $ \dbtilde{\CU}_\CS \rightarrow \bCU_\CS $, as follows.  There is a restriction map $ \CQ_n \rightarrow \CQ_{S_k} $ given by $ (\nu, \mu) \mapsto (\nu, \mu)|_{S_k} $ and this leads to a map $ \dbtilde{\CU}_\CS \rightarrow \circCY_{S_k}$, for each $k $.  Combining these together yields the morphism $ \dbtilde{\CU}_\CS \rightarrow \bCU_\CS$.

\begin{eg}
	Suppose that $ \CS = \{[n]\} $ (the unique set partition with $ m = 1$).  Then $ \bCU_{\{[n]\}} = \bbCU_{\{[n]\}} \times_{\Cf_n} \tCM_{n+1}$, and so $ \bCU_{\{[n]\}} = \tCM_{n+1} = \dbtilde{\CU}_{\{[n]\}}$.
	
	Suppose that $ \CS = [[n]] $ (the unique set partition with $ m = n $).  Then each $  \circCY_{S_k} $ and $ \tCM_{S_k +1} $ is a point and $ \bCU_{[[n]]} = \bbCU_{[[n]]} $.  In this case, $ \bCU_{[[n]]} =  \Cfcirc_n $ is the deformation of the reciprocal plane as discussed in Example \ref{eg:bbCU}.
\end{eg}

Now let $ \nu \in \bbCU_\CS \cap \bbCV^\CB$.

\begin{lem} \label{le:fibreF}
	The fibre of $ \gamma : \bCU_\CS \rightarrow \bbCU_\CS $ over $ \nu $ is given by $ \overline M_{B_1 + 1} \times \cdots \times \overline M_{B_r +1}$.
\end{lem}

\begin{proof}
	For $ k = 1, \dots, m $, let $ \nu^k = \nu|_{S_k} $ be the image of $ \nu $ under the map $ \bbCU_\CS \rightarrow \circCY_{S_k} $.
	Since $ \bCU_\CS = \bbCU_\CS \times_{\prod_k \circCY_{S_k}} \prod_k \tCM_{S_k +1}$, the fibre $\gamma^{-1}(\nu) $ is the product of the fibres $ \gamma^{-1}(\nu^k) $ of $ \tCM_{S_k +1} \rightarrow  \circCY_{S_k} $.
	
	Now, $ \nu^k$ lies in a stratum $\bbCV^{\CB^k}$ of $ \Cf_{S_k}$  given by a set partition $ \CB^k $ of $ S_k $.  Examining the definitions, we see that the set partition $ \CB $ of $ [n] $ must refine $ \CS $ and that in fact $ \CB $ is made by collecting all the parts of $ \CB^1, \dots, \CB^m $.
	
	The fibre  $ \gamma^{-1}( \nu^k)$ is $ \overline M_{T^k_1 + 1} \times \cdots \times \overline M_{T^k_{r_k} + 1} $ by Propositions \ref{pr:collapse} and \ref{pr:fibrecircY}.  Thus taking the product over $ k = 1, \dots, r $ gives us the desired result.
\end{proof}

\begin{lem} \label{le:bCUCU}
	For $ \vareps \ne 0 $, $\bCU_\CS(\vareps) = \dbtilde{\CU}_\CS(\vareps) $.
\end{lem}

\begin{proof}
	Let $ (\nu, \vareps) \in \bbCU_\CS  $.  Choose $ \CB $ so that $ (\nu, \vareps) \in \bbCV^\CB$.   The point $ (\nu, \vareps) $ corresponds to a point $ (C, \uz) \in \overline T_n $ under the isomorphism $ \overline \Cf_n(\vareps) \cong \overline T_n $.  Moreover, we see that $ z_i = z_j $ if and only if $ i \sim_\CB j $.
	
	Then by Proposition \ref{pr:collapse}, the fibre of $ \overline M_{n+2} \rightarrow \overline T_n $ over $ (C, \uz) $ is $ \overline M_{B_1 + 1} \times \cdots \times \overline M_{B_r + 1} $.
	
By the previous lemma, this is the same as the fibre of $ \bCU_\CS \rightarrow \bbCU_\CS $.  It is easy to see that the natural map $ \dbtilde{\CU}_\CS \rightarrow \bCU_\CS $ induces this isomorphism between fibres and thus this map gives an isomorphism $\bCU_\CS(\vareps) \cong \bbCU_\CS(\vareps) $.
\end{proof}

%

Let $\CS, \CS' $ be  two set partitions and let $ \bbCU_{\CS \CS'} = \bbCU_\CS \cap \bbCU_{\CS'} $.  Let $ \bCU_{\CS \CS'} $ be the preimage of $ \bbCU_{\CS \CS'} $ under the map $ \bCU_\CS \rightarrow \bbCU_{\CS'} $.

\begin{lem}
There is a natural isomorphism $ \bCU_{\CS \CS'} \cong \bCU_{\CS' \CS} $, compatible with the projections to $ \bbCU_{\CS \CS'}$.

\end{lem}

\begin{proof}
	Let $ \nu \in \bbCU_{\CS \CS'} $. Lemma \ref{le:fibreF} shows that the fibres of $\bCU_\CS \rightarrow \bbCU_\CS $ and $ \bCU_{\CS'} \rightarrow \bbCU_{\CS'} $ over this point are equal, since the fibre doesn't depend on $ \CS $ or $ \CS'$.  Thus, we deduce the desired isomorphism.
\end{proof}

\begin{cor} \label{cor:proper}
	There is a reduced scheme $ \CF_n \rightarrow \BA^1 $ which has an open cover by $ \bCU_\CS $. It contained $ \mathcal F_n $ as a dense subset.  It is equipped with proper morphisms $ \CQ_n \rightarrow \CF_n \xrightarrow{\gamma} \overline \Cf_n \rightarrow \BA^1 $ fitting into the diagram (\ref{eq:maindiagram}).
\end{cor}
\begin{proof}
	We have established everything except the properness of the morphisms.  To see this, note that for each set partition $ \CS$, the morphisms $ \dbtilde{\CU}_\CS \rightarrow \bCU_\CS $ and  $ \bCU_\CS \rightarrow \bbCU_\CS $ are proper, since they both involve forgetting some of the $ \mu $ coordinates (which take values in $ \BP^1 $).  Since properness is local on the base, this proves that $ \CQ_n \rightarrow \CF_n $ and $ \CF_n \rightarrow \overline \Cf_n $ are proper.  Finally, $ \CF_n \rightarrow \BA^1 $ is proper since all the $ \nu $ coordinates take values in $ \BP^1 $.
\end{proof}

We define $ \overline F_n = \CF_n(0)$, which we call the \textbf{cactus flower moduli space}.  As with $ \CF_n $, it is covered by open subschemes $ \CU_\CS $.  There are obvious versions of Lemmas \ref{le:CUCSembed} and \ref{le:bCUCSembed} and so we can also conclude that $ \overline F_n $ is reduced.

\begin{rem}
	The $ \Cx $ action on $ \overline \Cf_n $ defined in Remark \ref{rem:scaling2} lifts to a $ \Cx $ action on $ \CF_n $.  To see this, we just note that the defining equations (\ref{eq:defCQn}) of $ \CQ_n $ (and hence the equations for each $ \bCU_\CS $) are homogeneous, where we give $ \mu_{ijk} $ weight 0 (recall that $ \nu_{ij} $ and $ \vareps$ are each given weight 1).
\end{rem}

\subsection{Strata of $ \overline F_n $}
Let $\CS $ be a set partition.  We define $ \bV_\CS := \gamma^{-1}(V_\CS) \subset \overline F_n $ to be the preimage of $ \bbV_\CS \subset \ft_n $, and similar for $ \bV^\CB$ and $ \bV^\CB_\CS $.

Note that $ \bV_\CS \subset \bU_\CS $ and so we have a morphism $ \bV_\CS \rightarrow \tM_{S_k + 1} $ for each $ k$.

From Propositions \ref{prop:strataYn} and \ref{pr:strataMW}, and the definition of $ \overline F_n $, we deduce the following.

\begin{prop} \label{pr:strataFn}
	\begin{enumerate}
		\item For a set partition $ \CS $ with $m$ parts,
		there is an isomorphism $\bV_\CS \cong \tM_{S_1 +1} \times \cdots \times \tM_{S_m+1} $.  In particular $ \dim \bV_\CS = n  - m $.
		
		\item The following diagram
		\begin{equation}
			\begin{tikzcd}
				\dbtilde{V}_\CS \arrow[r,"\cong"] \arrow[d] & \tM_{S_1+1} \times \cdots \times \tM_{S_m+1} \times \overline M_{m+1} \arrow[d] \\
				\bV_\CS \arrow[r,"\cong"] \arrow[d] & \tM_{S_1 + 1} \times \cdots  \times \tM_{S_m + 1} \ar[d] \\
					\bbV_\CS \arrow[r,"\cong"]  & \circY_{S_1} \times \cdots \times \circY_{S_m}
			\end{tikzcd}
		\end{equation}
		commutes.  In particular, the fibre of $ Q_n \rightarrow \overline F_n $ over any point of $ \bV_\CS $ is isomorphic to $ \overline M_{m+1} $.
		\item For a set partition $ \CB $ with $ r $ parts, there is an isomorphism
		$$ \bV^\CB \cong \Ycirc_r \times \overline M_{B_1+1} \times \cdots \times \overline M_{B_r + 1} $$	
		In particular $ \dim \bV^\CB = n - 1 -r + p $, where $ p $ is the number of parts of $ \CB $ of size $ 1 $.
		\item The locus $ \bV^\CB_\CS $ is non-empty only when $ \CB $ refines $ \CS $ and we have
		$$
			\bV^\CB_\CS \cong F_{r_1} \times \cdots \times F_{r_m} \times \overline M_{B_1 + 1} \times \cdots \times \overline M_{B_r + 1}
			$$
			where $ r_k $ is the number of parts of $ \CB $ lying in $ S_k $.
	\end{enumerate}
\end{prop}

We may also take the closures of the above strata and we find
$$
\overline{\bV_\CS} \cong \overline F_{S_1} \times \cdots \times \overline F_{S_m} \qquad \overline{\bV^\CB} \cong \overline F_r \times  \overline M_{B_1 + 1} \times \cdots \times \overline M_{B_r + 1}
$$

\begin{eg}
	Take $ \CS = [[n]] $, the set partition with $n $ parts. Then $ \bbV_{[[n]]} $ consists of the unique point $ \nu = 0 $, which we call the maximal flower point.  The fibre over this point (and hence the stratum $ \bV_{[[n]]}$) is isomorphic to a point.
\end{eg}

\begin{figure} \label{fig:2}
	\includegraphics[trim=0 120 40 80, clip,width=\textwidth]{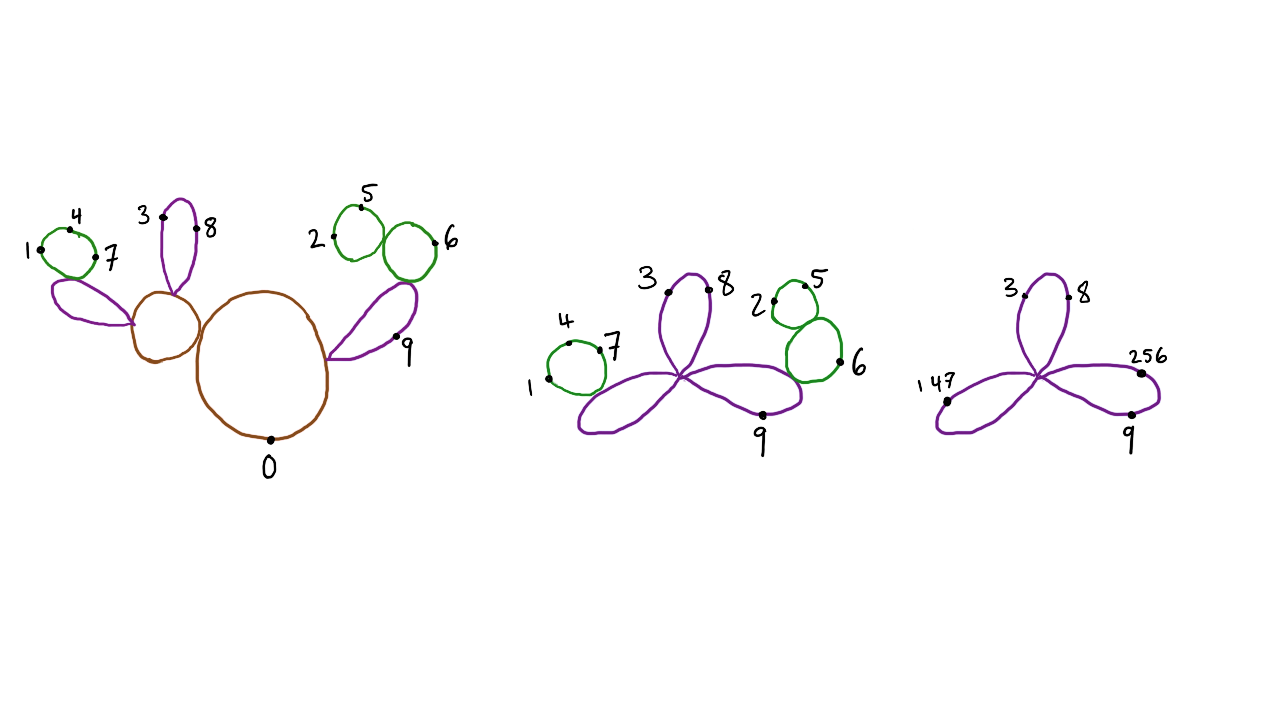}
	\caption{A point of $ Q_9$ and its images in $ \overline F_9 $ and $ \overline \ft_9 $.  In the notation from Proposition \ref{pr:strataFn} we have $ m = 3 $ and $ \CS = \{\{1,4,7\}, \{3,8\}, \{2, 5, 6, 9\}\} $; we also have $ r =5 $ and $ \CB = \{ \{1,4,7\}, \{ 3 \}, \{8 \}, \{2,5,6\}, \{9 \} \} $.}
	\end{figure}

\begin{rem} \label{rem:Fnparam}
	From this proposition, we can picture a point of $ \overline F_n $ as a collection $ (C_1, \uz^1, a_1),$ $\dots, (C_m, \uz^m, a_m) $ where $ (C_k, \uz^k,a_k) \in \tM_{S_k + 1} $.  In other words, we have a collection $ C_1, \dots C_m $ stable nodal curves with distinct marked points $ \uz^1, \dots, \uz^m $, and non-zero tangent vectors $a_1, \dots, a_m$ at one distinguished point.  We attach these curves together at their distinguished points to form a \textbf{cactus flower curve} $ C = C_1 \cup \dots \cup C_m $.
	
	If we compare this description with Remarks \ref{rem:Ynparam} and \ref{rem:Qnparam}, we see the following which is illustrated in figure \ref{fig:2}.
	\begin{enumerate}
		\item The space $ \overline \ft_n $ parametrizes projective lines $ C_1, \dots, C_m $, carrying a total of $ n $ (possibly non-distinct) marked points, and each of which carries a tangent vector at their distinguished point.  We imagine these lines attached together to form a flower with purple petals.
		\item The space $ \overline F_n $ parametrizes genus 0 stable nodal curves $ C_1, \dots, C_m $, carrying a total of $ n $ distinct marked points, and each of which carries a tangent vector at their distinguished point.  We imagine these curves attached together to form a flower of green cacti which have a purple base.
		\item The space $ Q_n $ parametrizes genus 0 stable nodal curves $ C_1, \dots, C_m $, carrying a total of $ n $ distinct marked points, each of which carries a tangent vector at their distinguished point, as well as a $m+1$-marked genus 0 stable nodal curve $ \tilde C$.   We imagine a maroon cactus with green/purple cacti attached to it.
	\end{enumerate}
	Examining this list, we see that there is a fourth possibility: projective lines $ C_1, \dots, C_m $, carrying a total of $ n $ (possibly non-distinct) marked points, and each of which carries a tangent vector at their distinguished point, as well as a $m+1$-marked genus 0 stable nodal curve $ \tilde C$.   We imagine a red cactus with purple petals attached to it.  In \cite{Z}, Zahariuc studied a space $P_n $ of marked nodal curves with vector fields and proved that it was a degeneration of $ \overline T_n $.  We believe that his space parametrizes these red cacti with purple petals.
\end{rem}

\begin{rem}
We emphasize that it remains an open problem to prove that $ \overline F_n $ is actually a moduli space of curves.  For this purpose, we need to rigorously define a moduli problem which is solved by $\overline F_n $.  Moreover, for this moduli problem, the data of the non-zero tangent vector is probably the wrong choice.
\end{rem}

\subsection{Strata of $ \CF_n $}
Let $\CS $ be a set partition.  We define $ \bCV_\CS := \gamma^{-1}(\CV_\CS) \subset \CF_n $ to be the preimage of $ \bbCV_\CS \subset \overline \Cf_n $, and similar for $ \tilde{\vec{\CV}}_\CS $ and $\bCV^\CB$.

\begin{prop} \label{pr:strataCFn}
	\begin{enumerate}
		\item For a set partition $ \CS $ with $m$ parts,
		there is an isomorphism $$ \tilde{\vec{\CV}}_\CS \cong \tCM_{S_1+1} \times_{\BA^1} \cdots \times_{\BA^1} \tCM_{S_m+1}. $$  In particular $ \dim \bCV_\CS = n +1 - m $.
		%
		\item For a set partition $ \CB $ with $ r $ parts, there is an isomorphism
		$$ \bCV^\CB \cong  \Cfcirc_r \times \overline M_{B_1+1} \times \cdots \times \overline M_{B_r + 1} $$	
		In particular $ \dim \bCV^\CB = n -r + p $, where $ p $ is the number of parts of $ \CB $ of size $ 1 $.
	\end{enumerate}
\end{prop}

\begin{rem} \label{rem:codim1}
	Among $ \bCV_\CS, \bCV^\CB $, the codimension 1 strata are given as follows
	\begin{enumerate}
		\item	We choose $ \CS = (A, B) $ where $A \sqcup B = [n] $ is a set partition with two parts.  In this case, we have
		$$
		\tilde{\vec{\CV}}_{(A, B)} \cong \tCM_{A+1} \times_{\BA^1} \tCM_{B+1} \quad \overline{\tilde{\vec{\CV}}_{(A, B)}} \cong \CF_A \times_{\BA^1} \CF_{B}
		$$
		and we have $ \nu_{ij} = 0 $ for $ i \in A, j \in B $.
		
		If we fix $ \vareps \ne 0 $, a generic point of $ \bCV_{\{A, B \}} \cap \CF_n(\vareps) \subset \overline M_{n+2} $ consists of two component curves $ C = C_1 \cup C_2 $ such that the marked points on $ C_1 $ are labelled by $ A \sqcup \{0\} $ and the marked points on $ C_2 $ are labelled by $ B \sqcup \{n+1\} $.  A generic point of $ \bV_{\{A, B\}} = \bCV_{\{A, B \}} \cap \CF_n(0) $ is shown in figure \ref{fig:3}.

		\item We choose $ \CB = \{\{a_1\}, \dots, \{a_p\}, B\} $, a set partition with one part of size not equal to $ 1$.  In this case, we have
		$$
		\bCV^\CB \cong  \Cfcirc_{p+1} \times \overline M_{B+1} \quad \overline{\bCV^\CB} \cong  \CF_{p+1} \times \overline M_{B+1}
		$$
		and we have $ \nu_{ij} = \infty $ for $ i,j \in B $.
		
				If we fix $ \vareps \ne 0 $, a generic point of $ \bCV^\CB \cap \CF_n(\vareps) \subset \overline M_{n+2} $ consists of two component curves $ C = C_1 \cup C_2 $ such that the marked points on $ C_1 $ are labelled by $\{0,a_1, \dots, a_p, n+1\} $ and the marked points on $ C_2 $ are labelled by $ B $. A generic point of $ \bV^\CB = \bCV^\BC \cap \CF_n(0) $ is shown in figure \ref{fig:3}.
	\end{enumerate}
	Their closures are precisely the irreducible components of $ \CF_n \setminus \mathcal F_n $.  To see this we note that every other stratum is contained in the closure of one of these strata.
\end{rem}

\begin{figure} \label{fig:3}
 	\includegraphics[trim=220 130 0 60, clip,width=0.6\textwidth]{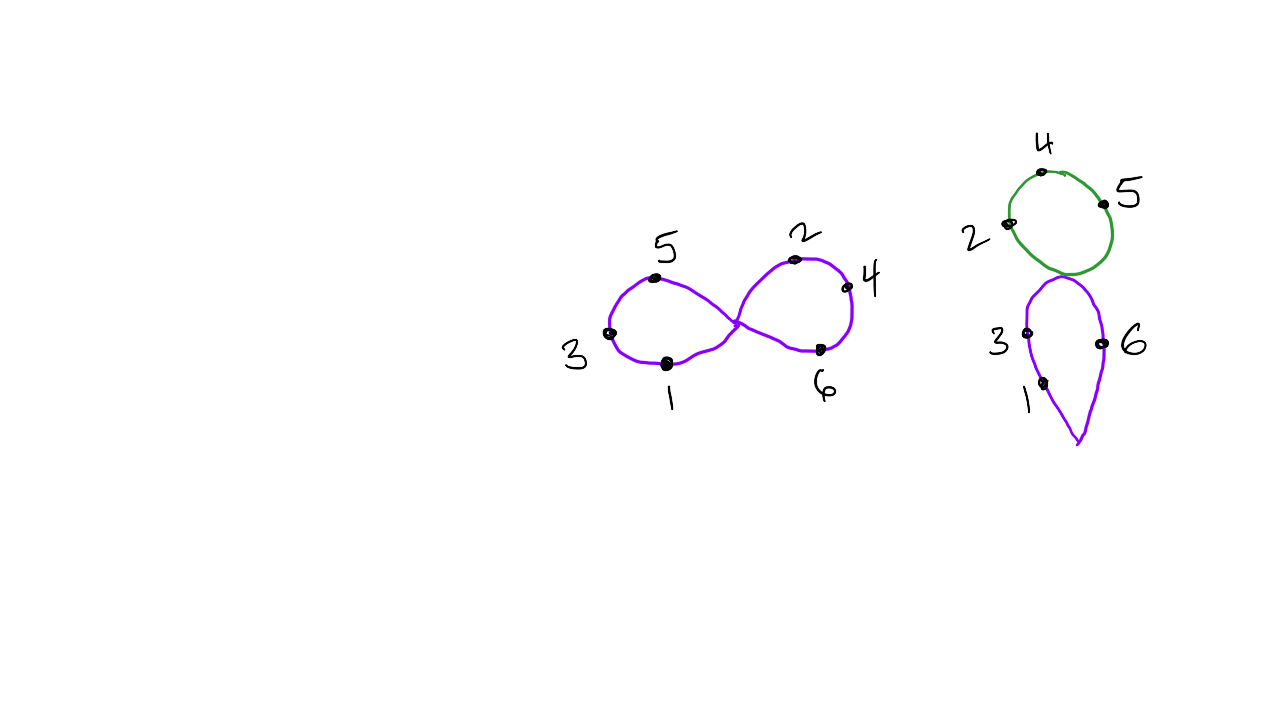}
	\caption{A generic point of $ \bV_{\{\{1,3,5\}, \{2,4,6\}\}}$ and one of $ \bV^{\{\{1\}, \{3\}, \{6\}, \{2,4,5\}\}}$.}
	\end{figure}

\subsection{A finer stratification of $ \overline F_n$} \label{se:bushy}
In Section \ref{se:linebundle}, we defined a stratification of $ \tM_{n+1} $ by bushy rooted trees.  We now extend this to $ \overline F_n $.  A \textbf{bushy rooted forest} is a collection of bushy rooted trees.

Given a point $C \in \overline F_n $, we will associate a bushy rooted forest as follows.

\begin{enumerate}
	\item First, $C$ lies in some $ \bV_\CS $. Using the isomorphism $ \bV_\CS \cong \tM_{S_1+1} \times \cdots \tM_{S_m+1} $, we associate $(C_1, \dots, C_m) $ where $ C_r \in \tM_{S_r+1} $.
	\item The point $ C_r \in \tM_{S_r+1} $ determines a $S_r$-labelled bushy rooted tree $ \tau_r $ (as in Section \ref{se:linebundle}).
	\item We collect the trees to form a forest $ \tau = \{ \tau_1, \dots, \tau_m \} $.
\end{enumerate}
Alternatively, if we regard $ C $ as a cactus flower curve $ C = C_1 \cup \dots \cup C_m $ as above, then $\tau $ is the forest of components of $ C_1 \sqcup \dots \sqcup C_m $, where the roots represent the components containing the distinguished points.

Thus we obtain a stratification of $ \overline F_n $ by $[n]$-labelled bushy rooted forests.  To each bushy rooted forest $ \tau $, we have a stratum of $ \overline F_n $ which is isomorphic to
$$
\prod_{\substack{r \in V(\tau)\\ \text{root}}} F_{E(r)} \times \prod_{\substack{v \in V(\tau)\\ \text{ non root}}} M_{E(v) + 1}
$$
where $ E(v) $ denotes the set of ascending edges containing $ v $. The 0-dimensional strata correspond to binary forests, that is bushy forests where each root has degree 1 and every internal vertex has degree 3.

\begin{eg}
Consider the middle curve in Figure \ref{fig:2}.  This is a point of $ \overline F_9 $.  The corresponding bushy rooted forest is
$$
		\begin{tikzpicture}[scale=0.6,every node/.style={scale=1}]]
		\draw (0,0) -- (0, 1);
		\draw (0,1) -- (-1, 3) node[above] {1};
		\draw (0,1) -- (0, 3) node[above] {4};
		\draw (0,1) -- (1, 3) node[above] {7};
		\draw (4,0) -- (3, 3) node[above] {3};
		\draw (4,0) -- (5,3) node[above] {8};
		\draw (9,0) -- (7,3) node[above] {9};
		\draw (9,0) -- (9,1) ;
		\draw (9,1) -- (11,3) node[above] {6};
		\draw (9,1) -- (9,2);
		\draw (9,2) -- (8.5, 3) node[above] {2};
		\draw (9,2) -- (9.5, 3) node[above] {5};
	\end{tikzpicture}
	$$
\end{eg}

\subsection{Open affine subsets of $ \CF_n $} \label{sec:Wtau}
Recall the open cover $ \bCU_\CS $ of $ \CF_n $.  These open sets are not generally affine, so now we will define an actual open affine cover consisting of smaller open sets.  These new open sets will be labelled by binary forests and centered on the corresponding 0-dimensional strata of $ \overline F_n $.

Let $ \tau $ be a binary forest.  Let $ \CS $ be the partition of $ [n] $ corresponding to the decomposition of the labels of $ \tau $ into trees (so $ i, j $ lie in the same part of $ \CS $ iff they lie on the same tree in $ \tau$).  By construction $ \bCU_\CS \subset \CU_\CS \times \prod_r \tCM_{S_r + 1} $.  The space $\CU_\CS $ is affine with coordinates $ \nu_{ij} $ or $ \nu_{ij}^{-1} $ for $ ij \in p([n]) $, while by Theorem \ref{th:QtM}, each space $ \tCM_{S_r+1} $ has $\BP^1$ valued coordinates $ \mu_{ijk} $ for $ ijk \in t(S_r) $.  We define
$$
\CW_\tau = \{(\nu, \mu) \in \bCU_\CS : \mu_{ijk} \ne \infty \text{ if the meet of $ i,k $ is above the meet of $ i, j$ in $\tau $} \}
$$
Because of the relations $ \mu_{ijk} \mu_{ikj} = 1 $ and $ \mu_{ijk} + \mu_{jik} = 1 $, we see that if the meet of $ i,k$ is above the meet of $ i,j$, then on $ \CW_\tau $
$$ \mu_{jik} \ne \infty, \ \mu_{jki} \ne 0, \ \mu_{jki} \ne 0, \ \mu_{kji} \ne 1, \ \mu_{kij} \ne 1, \ \mu_{ikj} \ne 0
$$

Thus, we conclude the following.
\begin{prop}
On $ \CW_\tau $, we have the following
\begin{gather*}
	\mu_{ijk} \ne 1, \infty, \text{ if the meet of $ i,k $ is above the meet of $ i, j$ } \\
	\mu_{ijk} \ne 0, \infty, \text{ if the meet of $ i,k $ equals the meet of $ i, j$ } \\
	\mu_{ijk} \ne 0, 1, \text{ if the meet of $ i,k $ is below the meet of $ i, j $}
\end{gather*}
In particular, $\CW_\tau $ is an affine scheme.
\end{prop}

We can generalize this as follows.
\begin{lem}
		Let $ i,j,k,l \in [n] $ be distinct.  Suppose that the meet of $ k,l $ lies weakly above the meet of $i,j $ in $ \tau $ (in particular, they must all lie in the same tree of $ \tau $).  Then the function $ \frac{z_k - z_l}{z_i - z_j} $ on $ \mathcal F_n $ extends to a regular map $ \CW_\tau \rightarrow \C $.  Moreover, if their meets are equal, then this function is never 0.
\end{lem}
\begin{proof}
	Let $ v' $ be the meet of $ k,l $ and let $ v $ be the meet of $ i,j$.  Consider the path $p $ from $ v' $ to the root; by hypothesis, this path must pass through $ v$.  Likewise, the paths from $i$ and $ j$ to the root must pass through $ v$.  Since $ v$ has only two ascending edges, one of these edges is common to $ p $. Without loss of generality, let us assume it is $ j $.  Then the meet of $ k $ and $ j $ and $l $ and $ j$ both lie above the meet of $ i $ and $ j$. So  $ \mu_{jik} $ and $ \mu_{jil} $ are both well-defined by the definition of $ \CW_\tau$, and
	$$ \frac{z_k - z_l}{z_i - z_j} = \frac{z_k - z_j}{z_i - z_j} - \frac{z_l - z_j}{z_i -z_j} = \mu_{jik} - \mu_{jil}$$
	Finally, if the meets are equal, then $ \frac{z_i - z_j}{z_k - z_l} $ is a also a regular function and is product with $ \frac{z_k - z_l}{z_i - z_j} $ is 1 on $ \CW_\tau $.
\end{proof}

Note that the open subset $ \CW_\tau $ contains the open locus $ \mathcal F_n $ for any $ \tau$.

\begin{lem}
	The open sets $ \CW_\tau $ cover $ \CF_n$.
\end{lem}
\begin{proof}
	Consider a point $ (\nu, \mu) \in \bCU_\CS $ for some $ \CS $.  By definition, this gives us points $ (\nu^r, \mu^r) \in \tM_{S_r+1} $ for $ r = 1, \dots, m $. Each such point determines a rooted tree $ \tau'_r $.  Because $ \mu_{ijk} = \frac{z_i - z_k}{z_i - z_j} $ we see that if the meet of $ i, k $ is not below the meet of $ i,j $ in $ \tau'_r $, then $ \mu_{ijk} \ne \infty $.  We can choose a binary forest $ \tau $ whose trees are labelled by $ \tau_1, \dots, \tau_m $ and such that $ \tau'_r $ is made from $ \tau_r $ by deleting edges.  Examining the definitions, we see that $ (\nu, \mu) \in \CW_\tau $.
\end{proof}

\section{Real structures}

\subsection{Generalities}
Let $ X $ be a scheme over $ \BR $.  Then $ X(\BC) $ carries a complex conjugation map $ \overline{\phantom{x}} : X(\BC) \rightarrow X(\BC) $ and we have $ X(\BR) = \{x \in X(\BC) : \bar x = x \} $.  More generally if $ A $ is any $ \BR$-algebra, then $ X(A \otimes_{\BR} \BC) $ carries a complex conjugation and $ X(A) = \{x \in X(A \otimes_\BR \BC) : \bar x = x \} $.

Now, suppose that we are given an involution $ \sigma : X \rightarrow X $, a morphism of schemes such that $ \sigma^2 = 1 $.  Then we can define a twisted real form $ X^\sigma$ such that for any $ \BR$-algebra $A$,
$$ X^\sigma(A) = \{ x\in X(A \otimes_\BR \BC) : \bar x = \sigma(x) \} $$

For each of our schemes we have an obvious real structure, which we will twist in this manner to obtain a twisted real form.  Recall that our schemes $ \overline \Cf_n, \CF_n $ are families over $ \BA^1 $.  The involution will act as $ \vareps \mapsto -\vareps $ on this $ \BA^1 $ and will be the identity on the $ \vareps = 0 $ fibre.

\subsection{The involutions and the real forms}
\subsubsection{The multiplicative group}
The most basic twisted real form we will consider concerns $ \Cx $.  Consider the automorphism $ \sigma$ of $ \BR^\times $ given by $ z \mapsto z^{-1} $.  Then the twisted real form is the unit 1 complex numbers, $ U(1) = \{z \in \Cx : \bar z = z^{-1} \} $, called the compact real form of this torus.
\subsubsection{Group scheme $ \BG$}
The twisted real form on $ \Cx $ naturally extends to the group scheme $ \BG$.  We define $\sigma : \BG \rightarrow \BG $ by $ \sigma(x, \vareps) = (x(1+x\vareps)^{-1}, -\vareps)$.

\begin{lem}
Suppose that $ (x, \vareps) \in \BG^\sigma(\BR) $.  Then:
\begin{enumerate}
\item $ \vareps \in i \BR $
\item If $ \vareps = 0 $, then $ x \in \BR$.
\item If $ \vareps = ir^{-1} \ne 0 $, then $ x $ lies on a circle of radius $ r$ centered at $ - i r$.
\item If $ \vareps \ne 0 $, then under the isomorphism $ \BG(\vareps) \cong \Cx$ given by $ x \mapsto 1 - \vareps x $, $\sigma $ becomes $ z \mapsto z^{-1}$ and $\BG(\vareps)^\sigma(\BR) $ is carried to the unit circle $ U(1) $.
\end{enumerate}
\end{lem}

\begin{proof}
	This follows from some simple calculations.  For (4), we use that $(1- \vareps x)^{-1} = 1 - \vareps x (1 - \vareps x)^{-1} $.
\end{proof}

\subsubsection{Flower space}
We define $ \sigma : \overline \Cf_n \rightarrow \overline \Cf_n $ by $ \sigma(\nu, \vareps) = (\nu - \vareps, -\vareps)$.  Note that $ \sigma $ is the identity on $ \overline \ft_n $, so $ \overline \ft_n^\sigma(\BR) = \overline \ft_n(\BR) $.

\begin{lem} The isomorphism $ \Cf_n \cong \BG^n / \BG $ is compatible with the involutions $ \sigma $.
Suppose that $ (\nu, \vareps) \in \overline \Cf_n^\sigma(\BR) $.  Then:
\begin{enumerate}
\item $ \vareps \in i \BR $.
\item $\nu_{ij} \in \BR\BP^1 + \vareps/2 $ for all $i,j$; equivalently $ \overline{\nu_{ij}} = \nu_{ji} = \nu_{ij} - \vareps$.
\item For $ \vareps \ne 0 $, the isomorphism $ \Cf_n(\vareps) \cong \overline T_n $ restricts to an isomorphism $ \Cf_n^\sigma(\vareps)(\BR) \cong U(1)^n / U(1) $, and is given by $ \alpha_{ij} = \frac{\overline{\nu_{ij}}}{\nu_{ij}}$.
\end{enumerate}
\end{lem}

\begin{proof}
	Again this follows by some simple computations.  For the first statement, we use that the isomorphism   $ \Cf_n \cong \BG^n / \BG $ is given by $ \nu_{ij} = \frac{1 - \vareps x_i}{x_i - x_j} $
	\end{proof}

\begin{rem}
	Recall that $ \Cf_n(\vareps) $ is a compactification of $ T_n = (\BC^\times)^n/ \Cx $, for $ \vareps \ne 0$, and a compactification $ \ft_n  = \BC^n / \C $, for $ \vareps = 0$.  The above result shows that this real locus does not see the compactification when $ \vareps \ne 0 $, but it does when $ \vareps = 0 $.
\end{rem}

\subsubsection{Deligne-Mumford space}
We define $ \overline M_{n+2} \rightarrow \overline M_{n+2} $ by $\sigma(C, \uz) = (C, z_{n+1}, z_1, \dots, z_n, z_0) $.

\begin{lem}
The inclusion $ M_{n+2} = (\Cx)^n \setminus \Delta / \Cx \subset \overline M_{n+2} $ is compatible with the involution $ \sigma $, where $ \sigma(z_1, \dots, z_n) = (z_1^{-1}, \dots, z_n^{-1})$.
\end{lem}
\begin{proof}
We have $ (z_1, \dots, z_n) \mapsto (\BP^1, z_0 = 0 , z_1, \dots, z_n, z_{n+1} = \infty) $. Applying $ \sigma $ to the right hand side gives
$$ (\BP^1, \infty, z_1, \dots, z_n, 0) = (\BP^1, 0, z_1^{-1}, \dots, z_n^{-1}, \infty) $$
where the equality uses the inverse map as an automorphism of $ \BP^1 $.
\end{proof}

Thus $\overline  M^\sigma_{n+2}(\BR) $ is the moduli space of $ (C, \uz) $ where $ C $ is defined over $ \BR $, $ \overline{z_0} = z_{n+1}$, and $ z_i \in C(\BR) $ for $ i = 1, \dots, n $; we have one pair of complex conjugate marked points and the rest of the marked points are real.  The different real forms of $ \overline M_{n+2} $ have been studied by Ceyhan \cite{C}.

For the open locus of $ \overline M^\sigma_{n+2}(\BR)$, we will regard $ C(\BR) = U(1) $, so that $ z_i \in U(1) $ for $ i =1, \dots, n $ and $ z_0, z_{n+1} \in \C \BP^1$, with $ \overline z_0^{-1} = z_{n+1} $.  Using the action of $ SU(2) $, we can arrange $ z_0 = 0, z_{n+1} = \infty $.  In this way, we see that the morphism $ \overline M_{n+2} \rightarrow \overline T_n $ restricts to a map $ \overline M_{n+2}^\sigma(\BR) \rightarrow U(1)^n/U(1)$.  For this reason, in \cite{IKR}, we call $ M_{n+2}^\sigma(\BR) $ the ``compact'' real form and $ M_{n+2}(\BR) $ the ``split'' real form.  (Note however that both spaces are real projective varieties and hence compact.)

Recall the $ \BP^1 $ valued coordinates $ \mu_{ijk} $ (for $ ijk \in t([n+1])$) from Theorem \ref{th:embedP1}.

\begin{prop}
	The involution $ \sigma : \overline M_{n+2} \rightarrow \overline M_{n+2} $ is given in these coordinates by
	$$ \sigma(\mu_{ijk}) = \mu_{ijk} \mu_{n+1 \,k j} \text{ for $ ijk \in t([n]) $  and } \sigma(\mu_{ij \, n+1}) = \mu_{j i \, n+1} \text{ for $ ij \in p([n])$}
	$$
\end{prop}

\begin{proof}
	As the function $ \mu_{ijk} $ are determined by their restrictions to $ M_{n+2} \subset \overline M_{n+2}$, it suffices to check these equations on this locus where they are obvious.
\end{proof}

\subsubsection{Cactus flower space} \label{se:cactusflowerReal}

We define $ \sigma : \CF_n \rightarrow \CF_n $ by defining it on each open set $ \bCU_\CS $ by
$$ \sigma(\nu_{ij}) = \nu_{ij} - \vareps \quad \sigma(\mu_{ijk}) = \mu_{ijk}(1 - \vareps \nu_{kj}^{-1}) $$
\begin{prop}
	The involution $\sigma $ is well-defined and glues together to an involution of $ \CF_n $.  It is compatible with the above defined involutions of $\overline \Cf_n $ and $ \overline M_{n+2} $.
\end{prop}
\begin{proof}
	To check that $ \sigma $ is well-defined, we must check that it is compatible with all the equations of each $ \bCU_\CS$.  This is a straightforward computation which we omit.
	
	The fact that the involutions glue is clear by the definition of $ \CF_n $.  Finally, their compatibility with the involutions of $ \overline \Cf_n $ and $ \overline M_{n+2} $ is clear.
\end{proof}

This leads to a twisted real form $ \CF_n^\sigma(\R) $.  By the above discussion, the $ \vareps$ coordinate gives a map $ \CF_n^\sigma(\R) \rightarrow i\R $.  For $ \vareps \ne 0 $, the fibre of this map is $ \CF_n^\sigma(\R)(\vareps) \cong \overline M_{n+2}^\sigma(\BR) $.

\begin{rem} \label{rem:scaling3}
	The $ \Cx $ action on $ \CF_n $ described in Remark \ref{rem:scaling2} restricts to a $ \BR^\times $ action on $ \CF_n^\sigma(\R)$.
\end{rem}

\section{Combinatorial spaces}

In the next two sections, our goal will be to define two combinatorial spaces $ D_n $ and $ P_n $ and then to define a commutative diagram
$$
	\begin{tikzcd}
	D_n \ar[d,"\Gamma"] \ar[r,"\theta"] & \overline F_n(\BR) \ar[d,"\gamma"] \\
	P_n \ar[r,"\Theta"] & \overline \ft_n(\BR)
\end{tikzcd}
$$
such that the horizontal arrows give rise to isomorphisms $ \widehat D_n \cong \overline F_n(\BR) $ and $ \widehat P_n \cong \overline \ft_n(\BR) $ for some explicit quotients $ D_n \rightarrow \widehat D_n $, $ P_n \rightarrow \widehat P_n $.

We will also define other quotients, $ \breve D_n $ and $ \breve P_n $ with homeomorphisms $ \breve D_n \cong \overline M_{n+2}^\sigma(\BR) $ (conjectural) and $\breve P_n \cong U(1)^n/U(1) $, which are the generic fibres of $ \CF_n(\BR) $ and $ \overline \Cf_n(\BR) $, respectively.

\subsection{The star} \label{se:star}
We begin by considering a non-convex polytope, closely related to the permutahedron.

Let
$$
X_n^e = \{ (x_1, \dots, x_n) \in \BR^n/\BR : 0 \le x_i - x_{i+1} \le 1 \text{ for } i = 1, \dots, n-1 \}
$$
which we call the fundamental parallelepiped.  More generally, for any $ w \in S_n $, let
$$
X_n^w = w X_n^e = \{ (x_1, \dots, x_n) \in \BR^n/\BR : 0 \le x_{w(i)} - x_{w(i+1)} \le 1 \text{ for } i = 1, \dots, n-1 \}
$$
and let $ X_n =  \bigcup_{w \in S_n} X_n^w $.  We call $X_n$, \textbf{the star}.

More generally, for any set $ S $ along with a total order $ w $ (a bijection $w : [n] \rightarrow S $ where $ n = |S| $) we define
$$X_S^w := \{ x \in \BR^S / \BR : 0 \le x_i - x_j \le 1 \text{ if $ i, j $ are consecutive in $ S $} \}
$$
and $ X_S := \bigcup_w X_S^w $.

The interior of $ X_n $, denoted $ X_n^\circ $, is the union of the subsets of $ X^w_n $ where $ x_{w(i)} - x_{w(i+1)} < 1 $.

Let $ x \in X^w_n $ lie on an outer face (i.e.\ not in $ X_n^\circ$). Then $ x $ determines a set partition $ \CS $ of $ [n]$, which is the finest partition such that if $ i, j $ are consecutive in the order $ w $ and $ x_i - x_j < 1 $, then $ i, j$ lie in the same part of $ \CS $.  Moreover, we can use $ x $ to define a point in $ \prod_{j=1}^m X_{S_j}  $ by restricting $ x $ to the parts of $ \CS = \{S_1, \dots, S_m \}$.

We define the equivalence relation $ \sim $ on $ X_n $ by setting $ x \sim x' $ if $ x, x' $ determine the same set partition $ \CS $ of $[n]$ and define the same point in $  \prod_{j=1}^m X_{S_j}  $.  Note that $ x, x' $ must (unless $ x = x'$) come from parallelepipeds for different orders on $ [n]$; each part $ S_j $ will be a consecutive block in both orders. We let $ \widehat X_n $ be the quotient of $ X_n $ by this equivalence relation.

The point $ \rho \in X_n^e $ defined by $ x_i - x_{i+1} = 1 $ for all $ i $, is called the \textbf{star point}.  Under the above equivalence relation, it is identified with all of its translates under the action of $ S_n $. (Note that it corresponds to the set partition $ [[n]] = \{\{1\}, \dots, \{n\}\} $.)

\begin{eg} \label{eg:star}
	Here is the star $X_3$ with the fundamental parallelepiped shaded in green, and the star point $ \rho$ labelled.  In the quotient $ \widehat X_3 $ all the points labelled by black dots are identified as well certain pairs of edges; one such pair is coloured red, thicker, and marked with an arrow.
	\begin{center}
		\begin{tikzpicture}
			\coordinate (a1) at (0:1);
			\coordinate (a2) at (30:{sqrt(3)});
			\coordinate (a3) at (60:1);
			\coordinate (a4) at (90:{sqrt(3)});
			\coordinate (a5) at (120:1);
			\coordinate (a6) at (150:{sqrt(3)});
			\coordinate (a7) at (180:1);
			\coordinate (a8) at (210:{sqrt(3)});
			\coordinate (a9) at (240:1);
			\coordinate (a10) at (270:{sqrt(3)});
			\coordinate (a11) at (300:1);
			\coordinate (a12) at (330:{sqrt(3)});
			\filldraw[fill=green!20!white, draw=green!20!white] (0,0) -- (a3) -- (a4) -- (a5) -- (0,0);
			\draw[red,thick,stealth-] (a1) -- (a2);
			\draw (a2) -- (a3);
			\draw (a3) -- (a4);
			\draw (a4) -- (a5);
			\draw (a5) -- (a6);
			\draw[red,thick, -stealth] (a6) -- (a7);
			\draw (a7) -- (a8);
			\draw (a8) -- (a9);
			\draw (a9) -- (a10);
			\draw (a10) -- (a11);
			\draw (a11) -- (a12);
			\draw (a12) -- (a1);
			\filldraw  (a4) circle [radius=2pt] node[above] {$\rho$};
			\filldraw  (a2) circle [radius=2pt] ;
			\filldraw  (a6) circle [radius=2pt] ;
			\filldraw  (a8) circle [radius=2pt] ;
			\filldraw  (a10) circle [radius=2pt] ;
			\filldraw  (a12) circle [radius=2pt] ;
			
		\end{tikzpicture}
	\end{center}
\end{eg}
In the appendix, we define the star for any root system and we prove that it is closely related to the permutahedron.

Let $ \rho = (n, n-1, \dots, 1) \in \BR^n / \BR $ and let $ P_n $ (the \textbf{permutahedron}) be the convex hull of the set $ \{ w \rho : w \in S_n \}$.  Let $ \widehat P_n $ be the quotient of $ P_n $ by the equivalence relation given by identifying all parallel faces. Theorem \ref{th:Bothmaps} specializes to the following result.

\begin{thm} \label{th:XnPn}
	There is a homeomorphism  $ X_n \cong P_n $ which induces a homeomorphisms $ \widehat X_n \cong \widehat P_n $.
\end{thm}

\begin{rem}
\label{rem:Pnpc}
	The quotient $ \widehat P_n $ was previously defined in \cite{BEER} and was shown to be nonpositively curved \cite[Thm~8.1]{BEER}.
\end{rem}

We will also be interested in an intermediate quotient.  Let $ \breve P_n $ be the equivalence relation given by identifying just closures of parallel facets.  There is also a similar quotient $ \breve X_n$, but we will not define this space, since it will not be used.

The polyhedron $ P_n $ and its two quotients $ \breve P_n $ and $\widehat P_n $ have natural cell structures. The $0$-cells of $P_n$ correspond to permutations $w\in S_n$. The $1$-cells of $P_n$ connect $0$-cells $w$ and $ww_{i\, i+1}$, for $w\in S_n$ and $w_{i\, i+1}$ the standard generators of $S_n$. The $2$-cells of $P_n$ are hexagons and squares corresponding to the cosets $w\langle w_{i\, i+1},w_{j\, j+1}\rangle$ for $j=i+1$ and $j>i+1$, respectively.  More generally, the $ k$-cells of $ P_n $ are labelled by cosets $ w \CW \subset S_n $ where $ \CW $ is a rank $ k $ standard parabolic subgroup and $ w \in S_n $.  Equivalently, the $k$-cells of $ P_n $ are indexed by ordered set partitions $ \CS $ of $ [n]$ with $ n- k $ parts.  The correspondence between these set partitions and cosets is given by
$$
\CS \mapsto \{ v \in S_n : v(a) < v(b) \text{ whenever $ a$ lies in an earlier part than $ b $ in $ \CS$} \}
$$

The complex $\widehat P_n$ has a single $0$-cell. Analyzing the relation $\sim$ in $X_n$, we obtain that in $\widehat P_n$ the (directed) $1$-cells $(w,ww_{k\, k+1})$ and $(w',w'w_{k'\, k'+1})$ of $P_n$ are identified, whenever $w(k)=w'(k')$ and $w(k+1)=w'(k'+1)$. Thus directed $1$-cells of $\widehat P_n$ are in correspondence with pairs $1\leq i\neq j \leq n$, where $w(k)=i,w(k+1)=j$.  More generally, two cells of $ P_n $ are identified in $ \widehat P_n $ if and only if they are indexed by two ordered set partitions whose underlying unordered partitions are equal. Thus $k$-cells of  $ \widehat P_n $ are indexed by unordered set partitions of $ [n]$ with $ n - k $ parts.

To describe the cells of $\breve P_n$, consider parallel facets of $P_n$ corresponding to $w\CW, ww_{1 n}\CW'$, where $\CW,\CW'$ have rank $n-2$, and $w_{1n}$ is the longest element of $S_n$. Suppose that $ \CW = S_p \times S_{n-p} $ and let $w_\CW$ be the longest element of~$\CW$. Then the translation identifying these facets sends $w$ to $ww_\CW w_{1 n}$. Note that $w_\CW w_{1 n}$ is a power $r^p$ of the long cycle $r$ in $S_n$. Thus in $\breve P_n$ we identified the $0$-cells $w,wr^p$ of $P_n$ for all $p$ and we identified the (directed) $1$-cells $(w,ww_{k\, k+1})$ and $(wr^p,wr^pw_{k-p\, k+1-p})$ (modulo $n$).  More generally, the cells of $ \breve P_n $ are indexed by cyclic set partitions (i.e. equivalence classes of ordered set partitions under the equivalence relation $ (S_1, \dots, S_m) \sim (S_2, \dots, S_m, S_1) $).



\subsection{Planar trees and forests}

A \textbf{planar tree}  is a rooted tree along with
an order on the set of ascending edges at each vertex.  A \textbf{planar forest} is a  sequence $ \tau = (\tau_1, \dots, \tau_m) $ of planar trees.

A planar forest $ \tau $ labelled by $ S $ determines the following total order on $S$, a bijection $ w_\tau : [n] \rightarrow S $ where $ n = \# S$. We visit the planar trees $ \tau_1, \dots, \tau_m $ in the obvious order. Given~$\tau_i$, we read the labels on its leaves in the order in which they are visited by a depth-first search starting at the root and respecting the order at each vertex. If $S=[n]$, then this total order of $[n]$ will be treated as a permutation of $[n]$.

For a planar forest $ \tau $, let $ V(\tau) $ denote the set of internal vertices. Let $ E(\tau) $ denote the set of edges of $ \tau $ that are \textbf{internal}, i.e.\ not containing a leaf. There is a natural bijection $ E(\tau) \rightarrow V(\tau) $ taking an edge $e \in E(\tau)$ to the vertex $v$ at which $e$ is descending. We then define $r_e(\tau)$ to be the result of \textbf{flipping} the order at $ v $. This means that we reverse the order of the ascending edges at $v$ and at all the other vertices $u$ such that the path from $u$ to the root passes through $v$.  Note that the effect of this flipping on the order is given by $ w_{r_e(\tau)} = w_\tau w_{ij} $ (as bijections $[n] \rightarrow S $), where the vertices above the edge $ e $ correspond to $ w_\tau(i), \dots, w_\tau(j) $.  Here $ w_{ij} \in S_n $ is the element of $ S_n $ which reverses $ [i,j] $ and leaves invariant the elements outside this interval.

\begin{eg} \label{eg:tree1}
	We will use the following running example throughout this section.  Here is a planar tree $ \tau $ with an edge $ e $, corresponding vertex $ v $, and the flipped tree $ r_e(\tau) $.
\begin{center}
	\begin{tikzpicture}[scale=0.6,every node/.style={scale=0.8}]
		\draw (0,0) -- node[left] {$e$}(-1, 1);
		\draw (-1,1) -- (0,2);
		\draw (0,2) -- (-1, 3) node[above] {2};
		\draw (0,2) -- (1, 3) node[above] {3};
		\draw (-1,1) -- (-3, 3) node[above] {1};
		\draw (0,0) -- (3,3) node[above] {4};
		\draw (0,0) --  (0,-1);
				\filldraw (-1,1) circle [radius=2pt] node[left] {$v$};
	\end{tikzpicture}
\quad \quad \quad
	\begin{tikzpicture}[scale=0.6,every node/.style={scale=0.7}]
	\draw (0,0) -- node[left] {$e$}(-1, 1);
	\draw (-1,1) -- (-2,2);
	\draw (-2,2) -- (-1, 3) node[above] {2};
	\draw (-2,2) -- (-3, 3) node[above] {3};
	\draw (-1,1) -- (1, 3) node[above] {1};
	\draw (0,0) -- (3,3) node[above] {4};
		\filldraw (-1,1) circle [radius=2pt] node[left] {$v$};
	\draw (0,0) -- (0,-1);
\end{tikzpicture}
\end{center}
\end{eg}



\subsection{The cube complex}

 A \textbf{cube complex} is a complex obtained by gluing cubes of side length~$2$ along their faces by isometries. For a detailed survey, we refer to \cite{Sageev}. Here we just recall some concepts. Given a vertex $v$ of a cube complex, its \textbf{link} is the small metric sphere around~$v$: its simplices correspond to the corners of cubes at $v$. If such a link is a simplicial complex, and all its cliques span simplices, then it is \textbf{flag}. Gromov proved that if all links are flag, then the path metric on the cube complex is \textbf{nonpositively curved} (also called \textbf{locally $\mathrm{CAT}(0)$}), which is a metric generalization of nonpositive curvature for Riemannian manifolds, see \cite[II]{BH}.

A map $\phi$ between cube complexes is \textbf{combinatorial}, if it maps the interior of each cube isometrically onto the interior of another cube. Such a map is a \textbf{local isometry} at a vertex $v$, if the induced map $\phi_v$ between the links at $v$ and at $\phi(v)$ is injective, and if for each simplex~$\triangle$ in the link of $\phi(v)$ with vertices $a_0,\ldots ,a_k$ in the image of $\phi_v$, we have that the simplex~$\triangle$ also lies in the image of $\phi_v$. If the link at  $\phi(v)$ is a simplicial complex, and the link at $v$ is flag, then it just suffices to verify this condition for $k=1$.

By \cite[Thm~B.2]{Leary}, if the target cube complex is nonpositively curved, a local isometry in this combinatorial sense is a local isometric embedding with respect to the path metrics on the cube complexes. By \cite[II.4.14]{BH} such a map is injective on the fundamental groups.

We define \textbf{the cube complex of planar forests $D_n$} as follows.  Let $PF_n(k)$ be the set of all planar forests labelled by $[n]$ with $k$ internal edges. This will be set of $k$-cubes, except that we will identify two such cubes if they are related by a sequence of flippings. In this way, we will end up with $ \# PF_n(k) / 2^k $ $k$-cubes.

For each planar forest $ \tau \in PF_n(k)$ and $ e \in E(\tau) $, we define $ d_e(\tau) \in PF_n(k-1) $ to be the result of collapsing the edge $e$.  However, if $ e$ is the trunk of a planar tree $ \tau_i $,  then we consider the planar forest $(\tau'_1, \dots, \tau'_p) $ such that identifying the roots of $\tau'_1, \dots, \tau'_p$, in that order, gives the tree obtained by collapsing $e$ in $\tau_i$. We then define $ d_e(\tau) = (\tau_1, \dots, \tau_{i-1}, \tau'_1, \dots, \tau'_p, \tau_{i+1}, \dots, \tau_m)$.

Note that we have a bijection $ E(d_e(\tau)) = E(\tau) \setminus \{e \} $. For distinct $e,f\in E(\tau)$, we have $r_f(d_e(\tau))=d_e(r_f(\tau))$.

 We write $ D_n $ for the geometric realization of this cube complex.  More precisely
 $$
 D_n = \bigcup_{\tau \in PF_n} [-1,1]^{E(\tau)} / \sim
 $$
 where $ \sim $ is the equivalence relation generated by
 $$
(\tau, (t,s)) \sim (r_e(\tau), (t, -s)) \qquad (\tau, (t, 1)) \sim (d_e(\tau), t)
 $$
Here $ t \in [-1,1]^{E(\tau) \setminus \{e\}} $, $ s \in [-1,1] $, and $ (t,s) $ denotes the result of inserting $ s $ into the coordinate labelled by $ e $.

Each individual cube $  [-1,1]^{E(\tau)} $ can be decomposed into $ 2^k $ (where $ k = \# E(\tau) $) \textbf{subcubes} (also called sub-$k$-cubes, wherever $k$ plays a role).  We call $ [0,1]^{E(\tau)} $ \textbf{the positive subcube}.  Note that each subcube of $[-1,1]^{E(\tau)}$  is the positive subcube for some unique $ \tau' $ obtained from $ \tau $ by a sequence of flippings.

Thinking about these sub-cubes, we can describe $D_n $ as
$$
D_n = \bigcup_{\tau \in PF_n} [0,1]^{E(\tau)} / \sim
$$
where $ \sim $ is the equivalence relation generated by
\begin{equation} \label{eq:Dnequiv}
(\tau, (t,0)) \sim (r_e(\tau), (t, 0)) \qquad (\tau, (t, 1)) \sim (d_e(\tau), t)
\end{equation}
Here $ t \in [0,1]^{E(\tau) \setminus \{e\}} $, and $ (t,0) $ denotes the result of inserting $ 0 $ into the coordinate labelled by $ e $.

The $0$-cubes of $D_n$ correspond to planar forests $\tau=(\tau_1, \ldots, \tau_n) $, where each $\tau_i$ is a single edge with leaf $v_i$. Thus a $0$-cube corresponds to the permutation $w=w_\tau\in S_n$, where $w(i)$ labels $v_i$. Furthermore, sub-$1$-cubes correspond to planar forests $\tau$, where exactly one $\tau_i$ is not a single edge, and has a single internal edge $e$, which is the trunk. Sub-$1$-cubes corresponding to $\tau$ and $r_e(\tau)$ form a $1$-cube containing the $0$-cubes corresponding to $w_\tau$ and~$w_{r_e(\tau)}$. We thus have a correspondence between the directed $1$-cubes and $\tau$ as above such that the directed $1$-cube corresponding to $\tau$ starts at the $0$-cube corresponding to $w_\tau$ and ends at the $0$-cube corresponding to~$w_{r_e(\tau)}$.

On the other hand, the sub-$(n-1)$-cubes (which are top-dimensional) correspond to planar binary trees.

\begin{eg} \label{eq:D3}
Here is a picture of the cube complex $D_3$ along with a zoom-in on one of the 2-cubes divided into 4 sub-2-cubes.  The left and right edges of the cube complex are identified in the manner shown by the arrows.  In the zoom-in, some of the forests labelling cubes are drawn.
\begin{center}
\begin{tikzpicture}[scale=1.3,baseline={(0,-1)}]
 \draw (0,0) -- (6,0);
 \draw[dashed] (0,1) -- (6,1);
  \draw (0,2) -- (6,2);
   \draw [-stealth, thick] (0,0) -- (0,2);
     \draw[dashed] (1,0) -- (1,2);
      \draw (2,0) -- (2,2);
     \draw[dashed] (3,0) -- (3,2);
     \draw (4,0) -- (4,2);
       \draw[dashed] (5,0) -- (5,2);
  \draw [stealth-, thick] (6,0) -- (6,2);
\end{tikzpicture} \hspace{0.4cm}
\begin{tikzpicture}[scale=0.4,every node/.style={scale=0.7}]
 \draw (-3,-3) -- (9,-3);
\draw[dashed] (-3,3) -- (9,3);
\draw (-3,9) -- (9,9);
 \draw (-3,-3) -- (-3,9);
\draw[dashed] (3,-3) -- (3,9);
\draw (9,-3) -- (9,9);
\draw (6,4.5) -- (6,5.5);
\draw (6, 5.5) -- (4.5,7) node[above] {1};
\draw (6, 5.5) -- (7, 6.3) ;
\draw (7, 6.3) -- (6, 7) node[above] {2};
\draw (7, 6.3) -- (7.5, 7) node[above] {3};
\draw (0,4.5) -- (0,5.5);
\draw (0, 5.5) -- (-1.5,7) node[above] {1};
\draw (0, 5.5) -- (1, 6.3) ;
\draw (1, 6.3) -- (0, 7) node[above] {3};
\draw (1, 6.3) -- (1.5, 7) node[above] {2};
\draw (10.3, 4.5) -- (10.3,5.5);
\draw (10.3, 5.5) -- (9.5,7) node[above] {1};
\draw (10.3, 5.5) -- (10.3, 7) node[above] {2} ;
\draw (10.3, 5.5) -- (11.1, 7) node[above] {3};
\draw (9.5, 9) -- (9.5, 10) node[above] {1};
\draw (10, 9) -- (10, 10) node[above] {2};
\draw (10.5, 9) -- (10.5, 10) node[above] {3};

\draw (5.3, 9.5) -- (5.3, 11-0.5) node[above] {1};
\draw (6.5, 9.5) -- (6.5, 10);
\draw (6.5, 10) -- (6.1, 11-0.5) node[above] {2};
\draw (6.5, 10) -- (6.9, 11-0.5) node[above] {3};

\draw (6, 4.5-6) -- (6,5.5-6);
\draw (6, 5.5-6) -- (7.5,7-6) node[above] {1};
\draw (6, 5.5-6) -- (5, 6.3-6) ;
\draw (5, 6.3-6) -- (6, 7-6) node[above] {2};
\draw (5, 6.3-6) -- (4.5, 7-6) node[above] {3};
\end{tikzpicture}
\end{center}
\end{eg}

\subsection{Map from the cube complex to the star} \label{se:DnXn}
We will now define a map $ \Gamma: D_n \rightarrow X_n $.
We will map the positive subcube $[0,1]^{E(\tau)} $ associated to a planar forest $ \tau $, into the parallelepiped $X^{w_\tau}_n $ associated to $ w_\tau $.


Let $ t \in [0,1]^{E(\tau)} $.  Let $ v \in V(\tau) $.  We define $ a_v = \prod_e t_e $ where the product is taken over all edges on the path between $ v $ and the root of the tree containing $ v $.  Note that to define a point $ x = \Gamma(t) \in X^{w_\tau}_n $, we need to specify $x_i - x_j \in [0,1] $ for every $ i,j \in [n] $ that are the images of consecutive numbers under $ w_\tau $.

If such $ i, j$  correspond to leaves in distinct trees, then we set $ x_i - x_j  = 1 $.  If $ i, j $  correspond to leaves in the same tree, then we set $ x_i - x_j = a_v $, where $ v $ is the meet of $ i, j$.

\begin{eg} \label{eg:tree2}
	We continue with the tree $ \tau $ from Example \ref{eg:tree1}.
\begin{center}
	  \begin{tikzpicture}[scale=0.6,every node/.style={scale=0.8}]
	  	\draw (0,0) -- node[left] {$t_2$}(-1, 1);
	  	\draw (-1,1) -- node[right] {$t_3$}(0,2);
	  	\draw (0,2) -- (-1, 3) node[above] {2};
	  	\draw (0,2) -- (1, 3) node[above] {3};
	  	\draw (-1,1) -- (-3, 3) node[above] {1};
	  	\draw (0,0) -- (3,3) node[above] {4};
	  	\draw (0,0) -- node[left] {$t_1$} (0,-1);
	\end{tikzpicture} \quad \quad
	  \begin{tikzpicture}[scale=0.6,every node/.style={scale=0.8}]
	\draw (0,0) -- (-1, 1);
	\draw (-1,1) -- (0,2);
	\draw (0,2) -- (-1, 3) node[above] {2};
	\draw (0,2) -- (1, 3) node[above] {3};
	\draw (-1,1) -- (-3, 3) node[above] {1};
	\draw (0,0) -- (3,3) node[above] {4};
	\draw (0,0)  -- (0,-1);
	\filldraw (0,0) circle [radius=2pt] node[right] {$t_1$};
		\filldraw (-1,1) circle [radius=2pt] node[right] {$t_1 t_2$};
			\filldraw (0,2) circle [radius=2pt] node[right] {$t_1 t_2 t_3$};
\end{tikzpicture}
\end{center}
On the left is a planar tree with internal edges carrying $ t_1, t_2, t_3 \in [0,1] $.  On the right, the tree is decorated with the values of the vertices. This results in a point $ x= \Gamma(t) \in \BR^4/\BR $ defined by
$$
x_1 - x_2 = t_1 t_2 \quad x_2 - x_3 = t_1 t_2 t_3 \quad x_3 - x_4 = t_1
$$
\end{eg}

\begin{prop} \label{prop:CubeStar}
	This gives a well-defined map $ \Gamma : D_n \rightarrow X_n$.
\end{prop}

\begin{proof}
We must check that $ \Gamma $ is well-defined.  There are two things to check, corresponding to the identifications in (\ref{eq:Dnequiv}).  First we need to check that $ \Gamma $ is well-defined on the overlap of subcubes.  Next we need to check that $\Gamma $ is well-defined on gluing faces of cubes.

For overlap of subcubes, consider a planar forest $ \tau $ and an edge $ e $.  The two planar forests $ \tau $ and $ r_e(\tau) $ define two total orders on $[n]$.  Let $ (\tau , (t,0)) $ be a point in $ D_n $, where $ t \in [0,1]^{E(\tau) \setminus \{e\}}$.  In $D_n$, it is the same as the point $ (r_e(\tau), (t,0))$.

Then we have two potentially different elements $ x, x' \in X_n $ as the images of $(\tau, (t,0)) $ and $ (r_e(\tau), (t,0))$. Assume, without loss of generality, that the order defined by $ \tau $ is the standard order.

  Let $ i $ be the smallest label on a leaf in $ \tau $ above $ e $ and let $ j $ be the largest label on a leaf above $ e$.

For $ k = i, \dots, j-1, $ the leaves labelled by $ k, k+1 $ both lie above $ e $ and so their meet $v $ does as well.  So $ a_v = 0 $, since it is the product of different edge values including $ t_e = 0 $. Hence $ x_k - x_{k+1} = 0 $.  Thus $ x_i = x_{i+1} = \cdots = x_j $.

Now in $ r_e(\tau) $, the order is $ 1, \dots, i-1, j , \dots, i, j+1, \dots , n $.  So by a similar argument $ x'_j = \dots = x'_i $.

Now if $ k+1 < i $ or $ k > j $, then $ k, k+1 $ have the same meet in both forests, and so $ x_k - x_{k+1} = x'_k - x'_{k+1} $.

Finally, $ i-1, i $ in $ \tau $ and $ i-1, j $ in $ r_e(\tau) $ have the same meet, and so $ x_{i-1} - x_{i} = x'_{i-1} - x'_{j} $ but since $ x'_j = x'_i $, we see that $ x_{i-1} - x_{i}= x'_{i-1} - x'_{i}$.  Similarly $x_{j} - x_{j+1} = x'_{j} - x'_{j+1} $.  So all differences are equal and we conclude that $ x = x' $.

To check gluing on faces, we consider a  planar forest $ \tau $ and an edge $ e $.  Let $ x =\Gamma(\tau, (t, 1))$ and $ x' = \Gamma(d_e(\tau), t)$. The orders defined by $ \tau $ and $ d_e(\tau) $ are the same.  Once again, assume that this is the standard order on $ [n] $.

If $ e $ contains two internal vertices $v, v '$, then at $ t_e = 1 $, we see that the values $a_v $ and $ a_{v'} $ are equal.  Since these two vertices are identified in $ d_e(\tau)$, for any $ i$, we see that $ x_i - x_{i+1} = x'_i - x'_{i+1}$.  On the other hand, if $ e $ is a trunk, i.e.\ it contains a vertex $ v $ and a root, then if the leaves labelled by $ i, i+1 $ are split into different trees by the collapse of $ e $, then $ x_i - x_{i+1} = a_v = t_e = 1$, while $ x'_i - x'_{i+1} = 1 $ by definition.
\end{proof}

\begin{example}
\label{exa:edge} Consider a $1$-cube of $D_n$ consisting of sub-$1$-cubes corresponding to planar forests $\tau$ and $\tau'=r_e(\tau)$, where $ e $ is the unique internal edge of $ \tau$.  Assume for simplicity that the total order corresponding to $ \tau $ is the standard order on $[n]$.  Let $ i, j $ be the minimal, maximal labels of leaves above $ e $.

Then $\Gamma$ sends its $0$-cubes to the star points of $X_n$ corresponding to $w_\tau = \id,w_{\tau'}=w_{ij}$ for some $i,j$. Furthermore, $\Gamma$ sends the midpoint of that $1$-cube to the point $x\in X_n$ where \begin{gather*}
	x_1-x_2 = \cdots = x_{i-1} - x_i = 1,\\
	x_i - x_{i+1} =\cdots =x_{j-1} -x_j = 0, \\
	x_j-x_{j+1}= \cdots = x_{n-1} - x_n = 1.
\end{gather*}
This is the point $ \omega_\Pi^D $ defined in Section~\ref{faceX} (where $ \Pi $ is the standard set of simple roots and $ D = \{1, \dots, i - 1, j, \dots, n-1\} $).

In $P_n$, this point is mapped to the centre of the face corresponding to $\mathcal \langle w_{i\,i+1},\ldots, w_{j-1\,j}\rangle$ (this point is denoted $ \rho_\Pi - \rho_\Pi^D $ in \ref{faceP}). Consequently, the entire $1$-cube is sent to the main diagonal of that face.

\end{example}

 \subsection{The quotients $ \widehat D_n $ and $\breve{D}_n$}

We first define the quotient $\widehat D_n$ of $ D_n $.

Given a planar forest $ \tau = (\tau_1, \dots, \tau_m) $, we identify the cubes $ [0,1]^{E(\tau)} $ and $ [0,1]^{E(\tau')} $ where $ \tau' $ is obtained from $\tau $ by permuting the planar trees by any element of $ S_m $. We let $ \widehat D_n $ be the quotient.  Note that a cube in~$D_n$ might be equivalent to itself only in the obvious way, i.e.\ the result of a sequence of flippings can never be the same as the result of permuting the planar trees. Consequently, $ \widehat D_n$~is still a cube complex.

In $\widehat D_n$, after quotienting by this equivalence relation (and the earlier one given by flipping orders), the cubes will be indexed by unordered planar forests $\widehat \tau=\{\tau_1, \dots, \tau_m\}$ modulo flipping at each edge.  We write $ \widehat{PF}_n $ for the set of $[n]$-labelled unordered planar forests (modulo the equivalence relation of flipping).

\begin{eg}
	The cube complex $ \widehat D_3$ has three $2$-cubes, six $1$-cubes, and one $0$-cube.  The three edges ($1$-cubes) on the top row are identified with those on the bottom row in a flipped fashion and all the vertices ($0$-cubes) are identified.
\begin{center}
\begin{tikzpicture}[scale=1.3,baseline={(0,0)}]
\begin{scope}[thick,decoration={
    markings,
    mark=at position 0.4 with {\arrow{>}}}]
 \draw[postaction={decorate}] (0,0) -- (2,0);
   \draw[postaction={decorate}] (2,2) -- (0,2);
   \end{scope}
\begin{scope}[thick,decoration={
    markings,
    mark=at position 0.4 with {\arrow{>>}}}]
 \draw[postaction={decorate}] (2,0) -- (4,0);
   \draw[postaction={decorate}] (4,2) -- (2,2);
   \end{scope}
   \begin{scope}[thick,decoration={
    markings,
    mark=at position 0.4 with {\arrow{>>>}}}]
 \draw[postaction={decorate}] (4,0) -- (6,0);
   \draw[postaction={decorate}] (6,2) -- (4,2);
   \end{scope}
   \draw [-stealth, very thick] (0,0) -- (0,2);
      \draw (2,0) -- (2,2);
     \draw (4,0) -- (4,2);
  \draw [stealth-, very thick] (6,0) -- (6,2);
 \filldraw (0,0) circle [radius=2pt];
  \filldraw (2,0) circle [radius=2pt];
   \filldraw (4,0) circle [radius=2pt];
    \filldraw (6,0) circle [radius=2pt];
     \filldraw (0,2) circle [radius=2pt];
      \filldraw (2,2) circle [radius=2pt];
       \filldraw (4,2) circle [radius=2pt];
              \filldraw (6,2) circle [radius=2pt];
\end{tikzpicture}
\end{center}
\end{eg}

The cube complex $\widehat D_n$ has only one $0$-cube. Furthermore, its sub-$1$-cubes correspond to~$\widehat \tau$ with only one $\tau_i$ not a single edge, and having one internal edge $e$, the trunk.  Such $\widehat \tau$ are in bijection with sequences $A=(a_1,\ldots, a_p)$, over $2\leq p\leq n$, where distinct $a_1,\ldots,a_p\in [n]$ label the leaves of $\tau_i$ in that order.  We have a correspondence between the directed $1$-cubes and $\widehat \tau$ (or $A$) as above such that the directed $1$-cube corresponding to $\widehat \tau$ starts with the sub-$1$-cube corresponding to $\widehat \tau$ (and ends with the sub-$1$-cube corresponding to~$r_e(\tau)$).

\vspace{0.2cm}
\begin{center}
\begin{tikzpicture}[scale=0.5,every node/.style={scale=0.8}]
\draw (0,0) -- (0,2);
\node at (1,1) {$\dots$};
\draw (3,0) -- (3, 1);
\draw (3,1) -- (2, 2) node[above] {$a_1$};
\draw (3,1) -- (4, 2) node[above] {$a_p$};
\node at (3, 2.5) {$\dots$};
\node at (5, 1) {$\dots$};
\draw (6,0) -- (6,2);
\end{tikzpicture}
\end{center}
\vspace{0.2cm}

We now define the quotient $ \breve D_n $ of $ D_n $.
	Given a planar forest $ \tau = (\tau_1, \dots, \tau_m) $, we identify the cubes $[0,1]^{E(\tau)} $ and $ [0,1]^{E(\tau')} $ where $ \tau' $ is obtained from $\tau $ by cyclically permuting the planar trees by a power of the long cycle $ r = (12\cdots m)$ of $ S_m $.  After quotienting by this equivalence relation, the cubes will be indexed by cyclically ordered lists $\breve \tau$ of planar trees (we call this a \textbf{cyclic forest}).  We denote by $ \breve{D}_n$ the quotient cube complex.
	
	We have obvious combinatorial maps $ \phi_n\colon D_n\to \breve D_n, $ and $\breve \phi_n \colon \breve D_n\to \widehat{D}_n$.

	\begin{lem}
		\label{le:npc}
		 The cube complexes $\widehat D_n,\breve D_n,$ and $D_n$ are nonpositively curved.
	\end{lem}
	
	\begin{proof}
		We concentrate on the cube complex $\widehat D_n.$ By a result of Gromov \cite{Gr}, see \cite[II.5.20]{BH}, we need to verify that the vertex link of $\widehat D_n$ is a flag simplicial complex.  In other words, we need to verify that
		\begin{itemize}
			\item
			the two sub-$1$-cubes of each sub-$2$-cube are distinct and determine the sub-$2$-cube uniquely, and
			\item
			each set of $k$ sub-$1$-cubes pairwise contained in sub-$2$-cubes is contained in a unique sub-$k$-cube.
		\end{itemize}
		
		For the first bullet point, note that each sub-$2$-cube corresponds to a set of planar trees $\widehat \tau$ with two internal edges. If these two edges lie in distinct planar trees of $\widehat \tau$ (and hence they are trunks), then its sub-$1$-cubes correspond to sequences $A,A'$ with disjoint sets of entries. If these two edges lie in a single tree of $\widehat\tau$ (and hence form an edge-path of length two starting at the root), then one of $A,A'$ is a proper interval inside the other. In particular $A\neq A'$.
		
		Conversely, if sequences $A,A'$ have disjoint sets of entries, then $\widehat \tau $ must have exactly two trees $\tau_i,\tau_j$ that are not single edges, and each of them has a single interior vertex with the order of the ascending edges determined  by $A,A'$. If $A'\subsetneq A$, then  $\widehat \tau $ must have exactly one tree~$\tau_i$ that is not single edge, with two interior vertices $v,v'$, where the ascending edges at~$v'$ are all ending with leaves and ordered according to $A'$, and the ascending edges at~$v$ ending with leaves are ordered according to $A\setminus A'$, with an edge $vv'$ inserted in the position corresponding to the interval $A'$. In particular, $A$ and $A'$ determine $\widehat \tau$ uniquely.
		
		For the second bullet point, suppose that the sequences $A_1, \ldots, A_k$
		correspond to sub-$1$-cubes pairwise contained in sub-$2$-cubes. We construct a set of planar trees $\widehat \tau$ corresponding to a sub-$k$-cube containing all our sub-$1$-cubes as follows. The edges of
		$\widehat \tau$ are in bijection with the union $\mathcal A$ of the set $\{A_1,\ldots, A_k\}$ and the set $[n]$, whose elements are treated as length~$1$ sequences. We direct all these edges, and identify the starting vertex of an edge $A''\in \mathcal A$ with the ending vertex of an edge $A\in \mathcal A$ whenever $A \subsetneq A''$ and there is no $A'\in \mathcal A$ with $A\subsetneq A'\subsetneq A''$.
		 Given $A''$, we order such edges $A$ according to their order in $A''$.
		This produces a required set of planar trees $\widehat \tau$ with roots the ending vertices of the maximal elements of $\mathcal A$.
		
		As in the first bullet point, it is easy to see that such $\widehat \tau$ is unique.

 The proofs for $\breve D_n$ and $D_n$ are analogous. Alternatively, we could appeal to Remark~\ref{rem:DJS} and \cite[Lem~3.4.1]{DJS}.
\end{proof}


	\begin{lem}
		\label{le:loc_iso}
		The maps $\phi_n,\breve \phi_n$ are local isometric embeddings.
	\end{lem}
	
	In particular, by \cite[II.4.14]{BH}, the homomorphisms induced between their fundamental groups are injective.
	
	\begin{proof}
		We first focus on the map $\breve \phi_n$. Let $\breve \tau_0$ be a cyclic forest corresponding to a $0$-cube of~$\breve D_n$ and let $\breve \tau,\breve \tau'$ correspond to sub-$k$-cubes containing that $0$-cube.
		This means that the cyclic order of the leaves of $\breve \tau$ and $\breve \tau'$ is the same as that of $\breve \tau_0$.
		
		For local injectivity, suppose that the sub-$k$-cubes corresponding to $\breve \tau,\breve \tau'$ map under $\breve \phi_n$ to the same sub-$k$-cube of $\widehat D_n$. Then $\breve \tau'$ consists of the same set of planar trees as $\breve \tau$, possibly ordered differently. However, the cyclic order of the leaves of $\breve \tau, \breve \tau'$ is the same, implying $\breve \tau= \breve \tau'$.
		
		 For the condition on $\Delta$, since $\breve D_n,\widehat D_n$ are nonpositively curved, we can assume $k=1$. Suppose that $\breve \tau,\breve \tau'$ correspond to sub-$1$-cubes mapped under $\breve \phi_n$ to sub-$1$-cubes corresponding to sequences $A,A'$, contained
		in the same sub-$2$-cube of $\widehat D_n$, corresponding to a set of planar trees $\widehat \rho$. Since $A,A'$ are intervals in the cyclic order of the leaves of $\breve \tau_0$, we can cyclically order the planar trees of $\widehat \rho$ into a cyclic forest $\breve \rho$ corresponding to a sub-$2$-cube containing the original sub-$1$-cubes. Thus $\phi_n$ is a local isometry (in the combinatorial sense). By \cite[Thm~B.2]{Leary}, this proves that $\breve \phi_n$ is a local isometric embedding.
		
		Analogously, replacing the discussion of the cyclic orders by the total orders, we obtain that the composition $\breve \phi_n\circ \phi_n$ is a local isometry, and so $\phi_n$ is a local isometry.
	\end{proof}
Recall the map $ \Gamma : D_n \rightarrow X_n $ constructed in Section \ref{se:DnXn}.  Via the homeomorphism $ X_n \cong P_n $ from Theorem \ref{th:XnPn}, we can consider $\Gamma$ as a map $ D_n \rightarrow P_n $.

\begin{lem} \label{lem:CubeStarDescend}
	The map $\Gamma$ descends to maps $ \widehat D_n \rightarrow \widehat P_n $ and $ \breve D_n \rightarrow \breve P_n $.
\end{lem}
\begin{proof}
	We will verify that $ \Gamma $ descends to a map $ \widehat D_n \rightarrow \widehat X_n$.  This transfers to the above statements, because of the compability of the equivalence relations between $ P_n $ and $X_n $, see Lemma \ref{le:A4}.
	
	Consider a planar forest $ \tau $ and let $ \tau' $ be the result of permuting the trees. Let $ t \in [0,1]^{E(\tau)} $ be edge values and let $ t' \in [0,1]^{E(\tau')} $ the corresponding edge values.  We must show that $ \Gamma(\tau, t) = \Gamma(\tau',t')$.
	
	To see this, we let $ \CS $ be the set partition of $ [n] $ with each part corresponding to the set of leaves of a tree of $ \tau$.  We note that this is the same partition as for $ \tau' $.  It is also the same partition as the one defined by $ x = \Gamma(\tau, t) $ and $ x' = \Gamma(\tau',t')$, since $x_i - x_j = 1 $ if $ i, j $ belong to different trees.  On the other hand, if $ i, j $ belong to the same tree of $\tau$, then they belong to the same tree of $ \tau' $ and $ x_i - x_j = x'_i - x'_j $.  Thus examining the definition, we see that $ \Gamma(\tau, t) $ and $ \Gamma(\tau', t') $ are identified by the equivalence relation.
	
	The same argument works for the descending to $ \breve D_n \rightarrow \breve P_n $, using the fact that cubes of $ \breve D_n $ are indexed by cyclic forests and cells of $ \breve P_n $ are indexed by cyclic set partitions (see the last paragraph of section \ref{se:star}).
\end{proof}

\begin{rem}
\label{rem:DJS}
In \cite{DJS}, Davis-Januszkiewicz-Scott defined the blowup $ \Sigma_\#$ of a Coxeter cell complex $ \Sigma $, depending on a choice of an ``admissible set'' $ \mathcal R $.

The complex $D_n$ is exactly the $\mathcal R$-blow-up $\Sigma_\#$ of $\Sigma$ (see \cite[\S 3.2]{DJS}), where $\Sigma$ is the permutahedron (with the action of the symmetric group $S_n$), and $\mathcal R$ is the minimal blow-up set. Furthermore, the universal covering space of $D_n$, which is a $\mathrm{CAT}(0)$ cube complex, was discussed as $\mathcal M_n$ in \cite[\S 7]{Genevois}.

Similarly, if we consider the Coxeter cell complex $\Sigma$ of the affine symmetric group $AS_n$, built of permutahedra, then the quotient of $\Sigma_\#$ by the action of the group $\BZ^n/\BZ$ from Section~\ref{sec:affine} coincides with the complex $\breve D_n$.
\end{rem}

\section{Isomorphisms between the combinatorial spaces and the real loci} \label{se:CombIso}
	\subsection{Map from the star to the flower space} \label{se:diffeof}
	For the remainder of the paper, fix an increasing diffeomorphism $ f : [-1,1] \rightarrow [-\infty, \infty] $ such that $ f(0) = 0$.  For example we can choose $ f(t) = \tan t\pi / 2 $.
	
	We define $ \Theta : X_n \rightarrow \overline \ft_n(\BR) $ as follows.  We define $ \Theta $ initially on $ X^e_n $, by $ \Theta(x) = \delta $ where
	$$ \delta_{i j} = \sum_{k = i}^{j-1} f(x_{k+1} - x_k), \text{ if $ i < j $ } \qquad	 \delta_{ij} = - \delta_{ji}, \text{ if $ i > j $}$$

	In the appendix (Theorem \ref{th:Bothmaps}), we proved the following result.
	\begin{thm}
		$\Theta $ extends uniquely to an $ S_n$-equivariant map $ X_n \rightarrow \overline \ft_n(\BR) $ and this gives an isomorphism $ \widehat X_n \cong \overline \ft_n(\BR) $.
	\end{thm}

\subsection{Map from the cube complex to the cactus flower space}
We begin by recalling charts on $ \tM_{n+1}$ associated to planar binary trees, originally due to de Concini-Procesi \cite{dCP} and described in Section~2.3 of \cite{Ryb1}.

Let $ \tau $ be a planar binary tree.  Recall in Section \ref{sec:Wtau}, we described an open subset $ \CW_\tau \subset \CF_n $ (which depends only on the underlying tree, and not the planar structure).  Since $ \tau $ is a tree, we have $ \CW_\tau \subset \tCM_{n+1} \subset \CF_n $.  We write $ W_\tau = \CW_\tau \cap \overline F_n $, and we have $ F_n \subset W_\tau \subset \tM_{n+1} $.

Though $ W_\tau $ does not depend on the planar structure on $ \tau $, we will now define coordinates on $ W_\tau $ which do depend on this planar structure.

Let $ e \in E(\tau) $ be an internal edge, descending at $ v $ and ascending at $ v' $.  Choose $ i, j $ to be the unique consecutive pair of leaves labels whose meet is $ v $. If $ e $ is not the trunk, choose $ k, l $ to be the unique consecutive pair of leaves whose meet is $ v' $.  Define $ b_e : W_\tau \rightarrow \C^{E(\tau)} $ by
$$
b_e = \frac{z_i - z_j}{z_k-z_l} \text{ if $e $ is not the trunk } \quad b_e = z_i - z_j \text{ if $ e $ is the trunk}
$$

From \cite[Theorem 3.1(1)]{dCP}, we see that this defines an isomorphism between $ W_\tau $ and an open subset $ W'_\tau \subset \C^{E(\tau)} $.  The inverse of this isomorphism will be denoted $ H_\tau $.

\begin{lem} \label{le:ratio}
	Let $ ij \in p([n]) $ and $ kl \in p([n]) $.  Suppose that the meet of $ i,j $ is above the meet of $ k,l $ in $ \tau $.  Then $ \frac{z_i - z_j}{z_k - z_l} $ evaluated on $ H_\tau(b) $ is a rational function in $ \{b_e \} $ whose denominator is a positive polynomial with constant term equal to 1.  Therefore, $ W'_\tau$ is defined by the non-vanishing of such polynomials.
\end{lem}
\begin{proof}
	Assume without loss of generality that the order defined by $ \tau $ is the standard order and that $ i < j $ and $ k < l $.  Then we have
	\begin{equation} \label{eq:bigfrac}
	\frac{z_i - z_j}{z_k - z_l} = \frac{ z_i - z_{i+1} + \dots + z_{j-1} - z_j}{z_k - z_{k+1} + \dots + z_{l-1} - z_l}
	\end{equation}
	For each $r$, we have $ z_r - z_{r+1} = \prod_{e} b_{e} =: a_v$, where $ v $ is the meet of $ r $ and $ r+1 $ and the product is taken over the edges on the path from the root to $ v $.
	
	Now fix $ v $ to be the meet of $ k $ and $ l $.  Then $ v $ is the meet of $ p $ and $p+1 $ for some $ k \le p < l $.  Also for every $r \ne p$ such that $ k \le r < l $ or $ i \le r < j $, the meet of $ r, r+1 $ lies above $ v $.  Thus, if consider (\ref{eq:bigfrac}), we see that $a_v $ appears as a term in the denominator and divides every other term in both the numerator and denominator.  So after dividing by $ a_v $, we reach a rational function of the desired form.
\end{proof}

\begin{lem} \label{le:ratiovanishes}
Fix some edge $ e $ in $ \tau $ and let $ H_\tau(b) = (C, \uz) $ as above.  Suppose that $ b_e = 0 $.  Then $ \frac{z_i - z_j}{z_k - z_l} = 0 $ whenever the meet of $ i, j$ is above $ e $ and the meet of $ k,l $ is below $ e $.
\end{lem}

\begin{proof}
It suffices to check this on the open subset of $ W_\tau $ where all other coordinates $ b_f $ are non-zero.  On this locus, for a consecutive pairs $i,j $ of leaf labels, we have $ z_i - z_j = 0 $ if and only if the meet of $ i, j$ is above $ e $.  This implies that the above ratio vanishes.
\end{proof}

Associated to $ \tau $, we have the subcube $ C(\tau) = [0,1]^{E(\tau)}$.  We consider an open subset of this cube defined by requiring that the value of the trunk $e_0$ is not 1, $C(\tau)^\circ = [0,1) \times [0,1]^{E(\tau) \setminus \{e_0\}} $.  We define a diffeomorphism
$$
B : C(\tau)^\circ = [0,1) \times [0,1]^{E(\tau) \setminus \{e_0\}} \xrightarrow{\sim} [0,\infty) \times [0,1]^{E(\tau) \setminus \{e_0\}} $$
by $B((t_e))= (b_e) $ where $ b_{e_0} = f(t_{e_0}) $ and if $ e \ne e_0 $, then
$$
b_e = \begin{cases} \frac{ f(t_e \prod_{e'} t_{e'} )}{f(\prod_{e'} t_{e'} )} \text{ if $ t_{e'} \ne 0 $ for all $ e' < e $} \\
	t_e \quad \text{ if $ t_{e'} = 0 $ for some $ e' < e $}
\end{cases}
$$
where the inequality $ e' < e $ means that $ e' $ lies on the path between $ e $ and the root and the products are taken over this set of edges.

\begin{eg} \label{eg:tree3}
	Consider our tree from Example \ref{eg:tree2} with the following edge values.
\begin{center}
	\begin{tikzpicture}[scale=0.6,every node/.style={scale=0.8}]
		\draw (0,0) -- node[left] {$t_2$}(-1, 1);
		\draw (-1,1) -- node[right] {$t_3$}(0,2);
		\draw (0,2) -- (-1, 3) node[above] {2};
		\draw (0,2) -- (1, 3) node[above] {3};
		\draw (-1,1) -- (-3, 3) node[above] {1};
		\draw (0,0) -- (3,3) node[above] {4};
		\draw (0,0) -- node[left] {$t_1$} (0,-1);
	\end{tikzpicture}
\end{center}
The resulting point $ (C, \uz) \in \tM_{n+1} $ is given by
$$
z_3 - z_4 = f(t_1) \qquad
\frac{z_2 - z_3}{z_1 - z_2} = \frac{f(t_1 t_2 t_3)}{f(t_1t_2)} \qquad \frac{z_1 - z_2}{z_3 - z_4} = \frac{f(t_1 t_2)}{f(t_1)}
$$
assuming that $ t_1, t_2 $ are non-zero.

Now suppose that $ t_1 = 0 $.  Then $ (C, \uz) $ lies in the zero section $ \overline M_5 \subset \tM_5 $ and we have
$$
\frac{z_2 - z_3}{z_1 - z_2} = t_3 \qquad \frac{z_1 - z_2}{z_3 - z_4} = t_2
$$
\end{eg}

\begin{lem}
	The map $ B $ is a diffeomorphism.
	\end{lem}
\begin{proof}
	This follows by l'H\^opital's rule.
\end{proof}

From Lemma \ref{le:ratio}, the polynomials whose non-vanishing defines $ W'_\tau $ are all positive polynomials with constant term 1, hence they cannot vanish on non-negative real numbers.  Thus, $ [0,\infty) \times [0,1]^{E(\tau) \setminus \{e_0\}} $ is contained in $ W'_\tau $.  Hence, we define $ \theta : C(\tau)^\circ \rightarrow \tM_{n+1} $ to be the composition $$
 C_\tau^\circ = [0,1) \times [0,1]^{E(\tau) \setminus \{e_0\}} \xrightarrow{B} [0,\infty) \times [0,1]^{E(\tau) \setminus \{e_0\}} \xrightarrow{H_\tau} W_\tau \subset \tM_{n+1}(\BR)$$

Let $ D_n^\circ $ be the subcomplex of the cube complex given by cubes indexed by trees (not forests) and where the value of the trunk is not 1.  As before, let $ X_n^\circ $ be the interior of the star.  Examining the definition of $ \Gamma : D_n \rightarrow X_n $, we see that $ \Gamma(D_n^\circ) \subset X_n^\circ $.

\begin{lem} \label{th:commutesHB}
	The diagram
$$
\begin{tikzcd}
D_n^\circ \supset C(\tau)^\circ \arrow{r}{\theta} \arrow{d}{\Gamma} & \tM_{n+1}(\BR) \ar[d,"\gamma"] \\
X_n^\circ  \arrow{r}{\sim}[swap]{\Theta} & \ft_n(\BR)	
\end{tikzcd}
$$	commutes.
\end{lem}

 \begin{proof}

Let $ (t_e) \in C_\tau^\circ$.  By continuity, we can assume that no $ t_e $ vanishes.  Let $ i, j $ be a pair of consecutive elements for $ [n]$ and let $ v $ be their meet.  It suffices to check that the coordinates $ \delta_{i j} $ are equal for the two possible images of $(t_e)$ in $ \ft_n $.

Following $ \theta $ to the right, and then the morphism $ \tM_{n+1} \rightarrow \ft_n$, we reach $ \delta_{i j} = f(\prod_e t_e) $, where the product is taken over all edges on the path from $ v $ to the root.  On the other hand, following $\Gamma$, we reach $ x_i - x_j = a_v = \prod_e t_e $, where the product is over the same edges.  Then when we apply $ \Theta$, we end up with $ \delta_{i j} = f(x_i - x_j)=  f(\prod_e t_e) $ as desired.
\end{proof}

\begin{lem} \label{le:Dncircmap}
	The above maps $ \theta: C(\tau)^\circ \rightarrow \tM_{n+1}(\BR) $ glue together to a map $ \theta : D_n^\circ \rightarrow \tM_{n+1}(\BR)$.
\end{lem}
\begin{proof}
	First we check that the maps glue. We must check this glueing under flipping and deletion of edges.

	 We begin with flipping.  Let $ \tau $ be a planar binary tree and let $ e $ be an edge.  Let $ t \in C_\tau^\circ$ with $ t_e = 0 $.  We wish to show that $ H_\tau(B(t)) = H_{r_e(\tau)}(B(t))$.
	
	 Let $ (C, \uz) $ denote the image of $ H_\tau(B((t)) $ and $ (C', \uz') $ the image of $ H_{r_e(\tau)}(B(t))$.  As usual, assume that order defined by $ \tau $ is the standard order on $ [n]$.  By Lemma \ref{th:commutesHB}, we see that $ z_i - z_j = z'_i - z'_j $ for all $ i,j $.  Now, we need to check if all the ratios associated to the edges agree.  For the trunk, this equality is clear by Lemma \ref{th:commutesHB}.

Now, consider a non-trunk edge $ f $ of $ \tau $ descending at $ v $ and ascending at $ v' $.  Let $i,j$ be the consecutive pair of leaves meeting at $ v $ and $ k,l $ be the consecutive pair of leaves meeting at $ v'$.  We have some possible cases.

First, suppose that $ f $ is above $ e$.   In this case, because of the order reversal, $ j,i $ will be consecutive in $ r_e(\tau) $ and will still meet at $ v$ and similarly for $l,k$.  Thus we see that
$$
\frac{z_i - z_j}{z_k - z_l} = t_f = \frac{z'_j - z'_i}{z'_l - z'_k} = \frac{z'_i - z'_j}{z'_k - z'_l}
$$
as desired.

Suppose that $ f = e $.   As above, $ j,i $ will be consecutive in $ r_e(\tau) $ and will still meet at $ v$.  On the other hand, $ k,l $ will no longer be consecutive, but will still meet at $ v' $ in $r_e(\tau) $.  Thus, applying Lemma \ref{le:ratiovanishes}, we conclude that
$$
\frac{z_i - z_j}{z_k - z_l} = 0 = \frac{z'_i - z'_j}{z'_k - z'_l}
$$
as desired.

Finally, suppose that $ f $ is below $ e $. At most one of $ i, j$ lies above $ e $ and at most one of $ k,l $ lies above $ e$. Assume that $ i $ lies above $ e $ and the rest of the leaves do not.  Then in $ r_e(\tau) $, $ i', j $ are consecutive and meet at $ v $, where $ i' $ is another leaf above $ e $.  Then we have
$$
\frac{z_i - z_j}{z_k - z_l} = t_f = \frac{z'_{i'} - z'_j}{z'_k - z'_l} = \frac{z'_i - z'_j}{z'_k - z'_l} - \frac{z'_{i'} - z'_i}{z'_k - z'_l} = \frac{z'_i - z'_j}{z'_k - z'_l}
$$
where in the last equality we apply Lemma \ref{le:ratiovanishes}.

Next, we check glueing with respect to deletion of an edge.  For this purpose let $ \tau, \tau' $ be two binary planar trees carrying non-trunk edges $ e, e' $ such that $ d_e(\tau) = d_{e'}(\tau') $.  We wish to show that $H_\tau(B(t)) = H_{r_e(\tau)}(B(t)) $ where $ t_e = 1 $. In this case, the orders defined by the these two trees agree, $ w_\tau = w_{\tau'}$.

Let $ v, u $ be the vertices in $ \tau $ such that $ e $ is descending at $ v $ and ascending at $ u $.  In the tree $ \tau' $, the vertex $u $ still exists while the  $ v $ is replaced with a different vertex $ v' $.  Moreover, suppose that $ i, j $ are consecutive and meet at $ v $ and $ k,l $ are consecutive and meet at $ u $ in $ \tau$.  Then in $ \tau' $ we see that $ k,l $ meet at $ v' $ and $ i,j $ meet at $ u $.  Then we find
$$
\frac{z_i - z_j}{z_k - z_l} = 1 \quad \frac{z'_k - z'_l}{z'_i - z'_j} = 1
$$
and so we see that these ratios are equal.  All the rest of the pairs of consecutive leaves meet at the same vertices and hence all the other ratios are equal.

\begin{center}
	\begin{tikzpicture}[scale=0.6,every node/.style={scale=0.8}]
				\draw (0,0) node[right] {$u$} -- (0,-1);
		\draw (0,0) -- node[left] {$e$} (-1, 1) node[left] {$v$};
		\draw (-1,1) -- (0,2);
		\draw (-1,1) -- (-2,2);
		\draw (-2,2) -- (-1.5,3) node[above] {$i$};
		\draw (0,2) -- (-0.8,3) node[above] {$j$};
		\draw (0,2) -- (0.8,3) node[above] {$k$};
		\draw (0,0) -- (2,2);
		\draw (2,2) -- (1.5,3) node[above] {$l$};
				\node at (0,3) {$\dots$};
						\node at (-2,3) {$\dots$};
								\node at (2,3) {$\dots$};
	\end{tikzpicture} \quad \quad
	\begin{tikzpicture}[scale=0.6,every node/.style={scale=0.8}]
	\draw (0,0) node[left] {$u$} -- (0,-1);
	\draw (0,0) -- node[right] {$e'$} (1, 1) node[right] {$v'$};
	\draw (1,1) -- (0,2);
	\draw (0,0) -- (-2,2);
	\draw (-2,2) -- (-1.5,3) node[above] {$i$};
	\draw (0,2) -- (-0.8,3) node[above] {$j$};
	\draw (0,2) -- (0.8,3) node[above] {$k$};
	\draw (1,1) -- (2,2);
	\draw (2,2) -- (1.5,3) node[above] {$l$};
	\node at (0,3) {$\dots$};
	\node at (-2,3) {$\dots$};
	\node at (2,3) {$\dots$};
\end{tikzpicture}
\end{center}
\end{proof}

\begin{lem} \label{le:Dnmap}
The map $ \theta: D_n^\circ \rightarrow \tM_{n+1}(\BR)$ extends to a map $ \theta: \widehat D_n \rightarrow \overline F_n(\BR) $.
\end{lem}
\begin{proof}
	 Let $ \tau = (\tau_1, \dots, \tau_m) $ be a planar forest with leaves labelled $ S_1, \dots, S_m $.  Then as above we have maps $C(\tau_j)^\circ \rightarrow \tM_{S_j +1}(\BR) $, which we combine together to give a maps
$$C(\tau_1)^\circ \times \cdots \times C(\tau_m)^\circ \rightarrow \tM_{S_1 +1}(\BR) \times \cdots \times \tM_{S_m+1}(\BR) \cong \bV_\CS(\BR) \subset \overline F_n(\BR)$$
Here we use the isomorphism $ \tM_{S_1 +1} \times \cdots \times \tM_{S_m+1} \cong \bV_\CS$ from Proposition \ref{pr:strataFn}(1).

Let $ C(\tau)^\circ$ be the subset of the cube for $ \tau $ where all the trunks are not given the value 1.  We have $C(\tau)^\circ = C(\tau_1)^\circ \times \cdots \times C(\tau_m)^\circ$.  As every point of $ D_n $ lies in some $ C(\tau)^\circ $ for some forest $ \tau $, we have defined $ \theta: D_n \rightarrow \overline F_n(\BR) $.

Note that the stratum $ \bV_\CS $  of $ \overline F_n $ depends only on $ \CS $ as an unordered set partition.  Thus, the image of the cube of $ \tau $ is the same as the image of the cube of any forest made by permuting the trees in $ \tau $ and this descends to a map  $ \theta: \widehat D_n \rightarrow \overline F_n(\BR) $.
\end{proof}

At this point, it will be useful to consider the \textbf{cubical subdivision} of $ \widehat D_n$.  The cubes of this complex will be our original ``sub cubes'' together with all of their faces.  We will call all these \textbf{little cubes} and use \textbf{big cubes} for our original cubes.  To index these little cubes, we introduce the following combinatorics.

A \textbf{planar tree with 0s} is a labelled planar rooted tree as before, along with a decoration of some of the internal edges by $0$, up to the following equivalence; two such trees are considered equivalent if they are related by flipping at edges decorated with $0$s.   A \textbf{planar forest with 0s} is a collection of planar trees with 0s. Given such a planar forest~$\tau$ with 0s, we write $ E(\tau) $ for the set of internal edges not decorated by $0$s.   We write $ ZT_n $ (resp. $ZF_n $) for the set of planar trees (resp. forests) with 0s labelled by $[n]$.  (Note that we are considering ``unordered'' planar forests here.)

There is an obvious map $ ZF_n \rightarrow \widehat{PF}_n $ given by forgetting which edges are decorated by 0s.

The following observation is clear.
\begin{lem}
	The little cubes of $ \widehat D_n $ are indexed by $ ZF_n $; the cube indexed by $ \tau \in ZF_n $ is $ [0,1]^{E(\tau)}$.  The embedding of little cubes into big cubes is given by the map $ ZF_n \rightarrow \widehat{PF}_n$.
\end{lem}

\begin{eg} Here is the cubical subdivision of the one of the maximal cubes in $ \widehat D_3$, with some of the little cubes labelled by  elements of $ ZF_3$.
\begin{center}
\begin{tikzpicture}[scale=0.4,every node/.style={scale=0.7}]
	\draw (-3,-3) -- (9,-3);
	\draw (-3,3) -- (9,3);
	\draw (-3,9) -- (9,9);
	\draw (-3,-3) -- (-3,9);
	\draw (3,-3) -- (3,9);
	\draw (9,-3) -- (9,9);
	
	\draw (6,4.5) -- (6,5.5);
	\draw (6, 5.5) -- (4.5,7) node[above] {1};
	\draw (6, 5.5) -- (7, 6.3) ;
	\draw (7, 6.3) -- (6, 7) node[above] {2};
	\draw (7, 6.3) -- (7.5, 7) node[above] {3};
	
	\begin{scope}[thin,xshift=-30,yshift=40,scale=0.7,every node/.style={scale=0.5}]
	\draw (6,4.5) -- (6,5.5);
	\draw (6, 5.5) -- (4.5,7) node[above] {1};
	\draw (6, 5.5) -- node {0} (7, 6.3) ;
	\draw (7, 6.3) -- (6, 7) node[above] {2};
	\draw (7, 6.3) -- (7.5, 7) node[above] {3};
	\end{scope}

	\begin{scope}[thin,xshift=60,yshift=-40,scale=0.7,every node/.style={scale=0.5}]
	\draw (6,4.5) -- node {0} (6,5.5);
	\draw (6, 5.5) -- (4.5,7) node[above] {1};
	\draw (6, 5.5) --  (7, 6.3) ;
	\draw (7, 6.3) -- (6, 7) node[above] {2};
	\draw (7, 6.3) -- (7.5, 7) node[above] {3};
\end{scope}
	
		\begin{scope}[thin,xshift=-30,yshift=-40,scale=0.7,every node/.style={scale=0.5}]
		\draw (6,4.5) -- node {0} (6,5.5);
		\draw (6, 5.5) -- (4.5,7) node[above] {1};
		\draw (6, 5.5) -- node {0} (7, 6.3) ;
		\draw (7, 6.3) -- (6, 7) node[above] {2};
		\draw (7, 6.3) -- (7.5, 7) node[above] {3};
	\end{scope}
	
	\draw (0,4.5) -- (0,5.5);
	\draw (0, 5.5) -- (-1.5,7) node[above] {1};
	\draw (0, 5.5) -- (1, 6.3) ;
	\draw (1, 6.3) -- (0, 7) node[above] {3};
	\draw (1, 6.3) -- (1.5, 7) node[above] {2};

	\draw (10.3, 4.5) -- (10.3,5.5);
	\draw (10.3, 5.5) -- (9.5,7) node[above] {1};
	\draw (10.3, 5.5) -- (10.3, 7) node[above] {2} ;
	\draw (10.3, 5.5) -- (11.1, 7) node[above] {3};
	
	\draw (9.5, 9) -- (9.5, 10) node[above] {1};
	\draw (10, 9) -- (10, 10) node[above] {2};
	\draw (10.5, 9) -- (10.5, 10) node[above] {3};
	
	\draw (5.3, 9.5) -- (5.3, 11-0.5) node[above] {1};
	\draw (6.5, 9.5) -- (6.5, 10);
	\draw (6.5, 10) -- (6.1, 11-0.5) node[above] {2};
	\draw (6.5, 10) -- (6.9, 11-0.5) node[above] {3};
	
	\draw (6, 4.5-6) -- (6,5.5-6);
	\draw (6, 5.5-6) -- (7.5,7-6) node[above] {1};
	\draw (6, 5.5-6) -- (5, 6.3-6) ;
	\draw (5, 6.3-6) -- (6, 7-6) node[above] {2};
	\draw (5, 6.3-6) -- (4.5, 7-6) node[above] {3};
\end{tikzpicture}
\end{center}
\end{eg}

By Lemma \ref{le:Dnmap}, we have a well defined map $(0,1)^{E(\tau)} \rightarrow \overline F_n(\BR) $ from the interior of each little cube.

\begin{lem}
Let $ \tau \in ZF_n$.  The map $ \theta: (0,1)^{E(\tau)} \rightarrow \overline F_n(\BR) $ is injective.  Its image is the set of $ (C, \uz) $ defined by the following conditions:
\begin{itemize}
	\item The set partition defined by distributing the marked points among the components of $ C $ is the same as the set partition defined by distributing $[n]$ among the trees in $ \tau $ (i.e. $(C,\uz) $ lies in the appropriate $ \bV_\CS $).
	\item For each $i,j,k,l \in [n] $ such that $i,j $ and $ k,l $ are consecutive for the order defined by $ \tau $, let $ v $ be the meet of $ i, j $ and $ v' $ be the meet of $k,l$.  Assume that $ v $ is weakly above $ v'$.  We have
	$$
	\frac{z_i - z_j}{z_k - z_l} = \begin{cases} 1 \text{ if $v = v'$} \\
		b \in (0,1) \text{ if the path between $v $ and $ v'$ contains no edges decorated with 0} \\
		0 \text{ if the path between $ v $ and $ v'$ contains an edge decorated with 0}
		\end{cases}
	$$
\end{itemize}
\end{lem}

\begin{proof}
 Since each little cube is contained in a big cube, and the map on each big cube is the restriction of a coordinate chart, we see that the map on each little cube is injective.

The conditions on the image come from examining the definitions.
\end{proof}

\begin{lem} \label{le:maptoF}
	For each point $ \uz \in F_n(\BR) = \BR^n \setminus \Delta / \BR $, there exists a unique $ \tau \in ZT_n $ such that no edges of $ \tau $ are decorated with 0s, and $ \uz $ is in the image of $ (0,1)^{E(\tau)}$.
\end{lem}
\begin{proof}
	The tree $ \tau $ is the unique tree such that $ z_i - z_j \le z_k - z_l $ whenever $ i,j $ and $k,l $ are consecutive, and the meet of $ i, j$ is weakly above $ k,l$.

	Beginning with $ \uz $ we will inductively define the tree $ \tau $.  Assume without loss of generality, that we have $ z_1 < \dots < z_n $.
	
	Let $ \ell = \min(z_2 - z_1, \dots, z_n - z_{n-1}) $ be the minimal distance between neighbouring point, let $ A = \{ i : z_{i+1} - z_i = \ell \} $ and let $ m = n - \#A $.  Then there exist unique $ 1 \le k_1 < \dots < k_m = n $ such that
\begin{gather*} z_2 - z_1 = \dots = z_{k_1} - z_{k_1 - 1} = \ell  \\ z_{k_1 +2} - z_{k_1 + 1} = \cdots = z_{k_2} - z_{k_2 - 1} = \ell  \\ \dots \end{gather*}
In other words, we partition $ z_1, \dots, z_n $ into consecutive groups where the minimal distance is attained.

Define a new sequence $ \uz' \in F_m(\BR) $, by
$$
z'_1 = z_1,\ z'_2 = z_{k_1 + 1} - z_{k_1} + z_1, \ \dots, \ z'_m = z_{k_{m-1} + 1} - z_{k_{m-1}} + \dots + z_{k_1 + 1} - z_{k_1} + z_1
$$
By induction (as $A $ nonempty, $m < n$), we have a $[m] $-labelled tree $ \tau' $ associated to $ \uz' $.  Define the tree $ \tau $ by replacing the leaf labelled $p$ in $ \tau'$ with an internal vertex and attaching leaves labelled $ k_{p-1} + 1, \dots, k_p $ to this vertex.

By construction, it is easy to see that $\uz $ is in the image of $(0,1)^{E(\tau)} $.
\end{proof}

Recall now that we have a copy of $ \overline M_{n+1} $ inside $ \tM_{n+1} $ as the zero section (the preimage of $ \delta = 0 $ under $ \gamma : \overline F_n \rightarrow \overline \Cf_n $).  So we have a copy of $ M_{n+1}(\BR) = \BR^n \setminus \Delta / \BR^\times \ltimes \BR $ inside $ \overline F_n(\BR) $.
	
\begin{lem} \label{le:maptoM}
		For each point $ \uz \in M_{n+1}(\BR) $, there exists a unique $ \tau \in ZT_n $ whose trunk is decorated with a 0, such that $ \uz $ is in the image of $ (0,1)^{E(\tau)}$.
\end{lem}
\begin{proof}
	The proof is almost identical to the previous one.  We just note that having the trunk decorated with a 0 means that the tree is only well-defined up to overall reversal.  This corresponds to the fact that the order on the points $ z_1, \dots, z_n $ is only well-defined up to reversal, as $ \BR^\times \ltimes \BR $ contains multiplication by $ -1 $.
\end{proof}

Now, we extend to all of $ \overline F_n(\BR) $.  For this purpose, we will need to relate the combinatorics of planar forests to the combinatorics of bushy rooted forests, which index the strata of $ \overline F_n $.

Recall from Section \ref{se:bushy}, that a bushy rooted forest is a rooted forest, except that we allow the roots to be contained in more than one edge.  To each point $ C = C_1 \cup \dots \cup C_m \in \overline F_n $, we assign a forest $ \tau $ which is the component graph of the $ C_j $.  Conversely, to each bushy rooted forest $ \tau $, we have a stratum of $ \overline F_n $ which is isomorphic to
$$
\prod_{\substack{r \in V(\tau)\\ \text{root}}} F_{E(r)} \times \prod_{\substack{v \in V(\tau)\\ \text{ non root}}} M_{E(v) + 1}
$$
where $ E(v) $ denotes the set of ascending edges containing $ v $.

If we look at the real points of each stratum, we note that $F_n(\BR) = \BR^n \setminus \Delta / \BR $ has $ n!$ connected components, corresponding to orderings of the points, while $ M_{n+1}(\BR) = \BR^n \setminus \Delta / B(\BR) $, where $ \BR =\BR^\times \ltimes \BR $, has $ n! /2 $ connected components, corresponding to orders of the points modulo reversal.  To take these components into account, we define a \textbf{planar bushy rooted forest} to be a bushy rooted forest along with an order of the ascending edges at each vertex, except that two such forests are considered equivalent if they are related by reversing the order at a non-root vertex.  We write $ BF_n $ for the set of $[n]$-labelled planar bushy rooted forests.

To summarize, we have three sets of labelled planar forests:
\begin{enumerate}
	\item Planar forests, $\widehat{PF}_n $, which index the big cubes of the cube complex $ \widehat D_n $.
	\item Planar forests with with 0s, $ZF_n$, which index the small cubes of the cubical subdivision of $ \widehat D_n $.
	\item Planar bushy forests, $ BF_n$, which index the connected components of the strata of $ \overline F_n(\BR) $.
\end{enumerate}
(Throughout this section, all of our forests are unordered.)

We already have a map $ ZF_n \rightarrow \widehat{PF}_n $, defined above, and now we define a map $ ZF_n \rightarrow BF_n $ by collapsing all the edges not labelled $ 0 $ (and then erasing all the 0s).

\begin{eg}
	Here is an element of $ ZT_4$ and its image in $ BF_4$.
	$$
		\begin{tikzpicture}[scale=0.5,every node/.style={scale=0.8}]]
			\draw (0,0) -- node[left] {$0$}(-1, 1);
			\draw (-1,1) -- (0,2);
			\draw (0,2) -- (-1, 3) node[above] {2};
			\draw (0,2) -- (1, 3) node[above] {3};
			\draw (-1,1) -- (-3, 3) node[above] {1};
			\draw (0,0) -- (3,3) node[above] {4};
			\draw (0,0) --  (0,-1);
		\draw[|->] (4.5,2) to (6,2);
	\begin{scope}[xshift=290]
		\draw (0,0) -- (-1, 1);
		\draw (-1,1) -- (-1, 3) node[above] {2};
		\draw (-1,1) -- (1, 3) node[above] {3};
		\draw (-1,1) -- (-3, 3) node[above] {1};
		\draw (0,0) -- (3,3) node[above] {4};
		\end{scope}
	\end{tikzpicture}
	$$
\end{eg}

\begin{eg}
	Recall the ($n-1$)-cubes (the biggest possible ones) of $ \widehat D_n $ are labelled by binary trees (the planar structure goes away because of the flipping).  So, the centres of the big ($n-1$)-cubes of $ \widehat D_n $ are labelled by binary trees where every edge is labelled by 0.  Under the above map, they go to the binary trees in $ BF_n $.  These label the point strata of $ \overline F_n(\BR) $.
\end{eg}

\begin{lem} \label{le:surject}
	For each point $ \uz \in \overline F_n(\BR) $, there exists a unique $ \tau \in ZF_n $, such that $ \uz $ is in the image of $ (0,1)^{E(\tau)}$.
\end{lem}
\begin{proof}
	Fix $ (C, \uz) \in \overline F_n(\BR) $.  Then $ (C,\uz)$ lies in a stratum of $ \overline F_n $ labelled by a bushy forest $ \tau_1$ and as explained above this stratum is isomorphic to
	$$
	\prod_{r \in V(\tau_1) \text{ root}} F_{E(r)} \times \prod_{v \in V(\tau_1) \text{ non root}} M_{E(v) + 1}
	$$
	The order of the points on each component gives a planar structure to $ \tau_1 $ and thus we can get an element $ \tau_2 \in BF_n $.
	
	Applying Lemmas \ref{le:maptoF} and \ref{le:maptoM} to each factor above, we deduce that there exists a unique lift $ \tau := \tau_3 \in ZF_n $ of $ \tau_2 $ such that $ (C, \uz) $ is in the image of $ (0,1)^{E(\tau)} $.
\end{proof}

\begin{thm} \label{th:homeohat}
	 $ \theta: \widehat D_n \rightarrow \overline F_n(\BR) $ is a homeomorphism.  It is compatible with the homeomorphism $ \Theta: \widehat X_n \cong \overline \ft_n(\BR) $.
\end{thm}

\begin{proof}
	Because $\theta $ is defined using charts of $ \tM_{n+1} $, it is injective on each little cube.  Lemma \ref{le:surject} implies that each point is in the image of precisely one cube.  Thus, we conclude that $  \theta $ is a bijection.  It is continuous because setting the value of the trunk to be 1, $ t_{e_0} = 1 $, is compatible with diffeomorphism $ B$ and the decomposition $C(\tau)^\circ = C(\tau_1)^\circ \times \cdots \times C(\tau_m)^\circ$.
	
	Finally, $\theta$ is a continuous bijection from a Hausdorff space to a compact space, and hence it is a homeomorphism.
	
	The compatibility of $ \theta $ and $ \Theta$ follows from Lemma \ref{th:commutesHB}.
\end{proof}

\begin{rem} \label{rem:dual}
	Let $ H_n $ be the subcomplex of our cube complex indexed by trees in $ ZT_n $ where the trunk is decorated by $0$ (this is a hyperplane of our original cube complex $ D_n $).  Our homeomorphism $ \widehat D_n \cong \overline F_n(\BR) $ restricts to a homeomorphism $ H_n \cong \overline M_{n+1}(\BR) $.  Our proof above shows that under this homeomorphism, the cube complex and geometric stratification are dual complexes.  This fact has been remarked before in the literature, but we were not able to find a precise proof.
\end{rem}

\subsection{Combinatorial models for deformations}

Recall that we have the deformations $ \CF_n(\BR) $ of $ \overline F_n(\BR) $ whose general fibre is $ \overline M_{n+2}^\sigma(\BR) $ (see \ref{se:cactusflowerReal}) and the deformation $ \overline \Cf_n(\BR) $ of $ \overline \ft_n(\BR) $ whose general fibre is $ U(1)^n/U(1) $.  We now describe combinatorial models for these spaces.

First, recall the quotient $ \breve P_n $ of the permutahedron by opposite facets.
\begin{prop} \label{prop:homeobreve1}
	There is a homeomorphism $ \breve P_n \cong U(1)^n/ U(1)$.
\end{prop}
\begin{proof}
Consider the action of $ A:=\BZ^n/\BZ $ on $ \BR^n/\BR$ by $ a \cdot x = x + n a $ (so we are considering translation by the subgroup generated by the vector $(n,0,\ldots, 0)=(n-1,-1,
\ldots, -1)$ and the vectors obtained from it by permuting the coordinates).

By \cite[Thm~A]{Munro}, the orbit of the origin $ 0 \in \BR^n/\BR$ under $A$ coincides with the orbit $\mathcal O$ of $0$ under the action of the affine Coxeter group of type $A_{n-1}$ acting on  $\BR^n / \BR$. Let $P\subset\BR^n/\BR$ be the locus of points $x$ such that the distance from $x$ to $\mathcal O$ is attained at $0$. It is a well-known fact in Coxeter groups that $P = P_n$, the permutahedron with vertices $ w \rho $, where $ w \in S_n $ and $ \rho =(n, n-1, \dots, 1) $.

By construction, $P$ is a fundamental domain for the action of $A$ on~$\BR^n/\BR$. Furthermore, for $x\in \partial P$, the point $x$ is at the same distance from~$0$ and a translate~$na$ for some nontrivial $a\in A$. Then $x - na\in \partial P$. In particular, if $x$ belongs to the interior of a facet $F$ of~$P$, then such $a$ is unique and common to all $x\in F$, and so $F - na$ is also a facet of $P$. It is also easy to check that for $x\in \partial P$ not in the interior of a facet of the permutahedron, and for~$a$ as before, there are always facets $F, F - nb\subset \partial P$ with $x\in F, x - nb\in F -nb$, where $b\in A$. Thus, identifying the opposite facets of $P_n$ as in the definition of $\breve P_n$, we obtain the quotient $(\BR^n/\BR) / A = U(1)^n/U(1)$.
\end{proof}

Now, we consider $ \overline M_{n+2}^\sigma(\BR) $.  Recall that the strata of $ \overline M_{n+2} $ are indexed by $[n]$-labelled rooted trees.  In $ \overline M_{n+2}^\sigma(\BR) $, the marked points $ z_0, z_{n+1} $ always lie on the same component, so for a stratum which intersects this real locus, in the corresponding tree, the leaf labelled $ n+1 $ is always connected to the vertex which is connected to the root.  So we can delete this vertex and the leaf labelled $ n+1 $ (producing a $[n]$-labelled rooted forest) without losing any information.

Now, as in the previous section, we split the strata of $ \overline M_{n+2}^\sigma(\BR) $ into connected components.  This leads to an order at each vertex, and a cyclic order on the set of trees, exactly the data of an $[n]$-labelled cyclic forest.  The codimension of such a stratum component is given by the number of internal edges.  (These strata and their components were also studied by Ceyhan \cite{C}.)

These same $[n]$-labelled cyclic forests indexed the cubes of $ \breve D_n $, where the number of internal edges give the dimension of the cube.  This motivates the following conjecture (compare with Remark \ref{rem:dual}).

\begin{conj}\label{conj:homeobreve2}
	There is a homeomorphism   $ \breve D_n \cong \overline M_{n+2}^\sigma(\BR) $ such that the cube complex and the geometric stratification are dual complexes.
\end{conj}

We will now formulate a conjecture concerning a combinatorial description of the total space of the deformation $ \CF_n(\BR) $.  We have an obvious quotient map $ \breve \phi_n :  \breve D_n \rightarrow \widehat D_n $.  Define a two-sided mapping cylinder for this map by
$$
\widehat \CD_n := \breve D_n \times \BR \sqcup \widehat D_n / (x, 0) \sim \breve \phi_n(x)
$$
Similarly, we have a quotient map $ q: \breve P_n \rightarrow \widehat P_n $ and we can define
$$
\widehat \CP_n := \breve P_n \times \BR \sqcup \widehat P_n / (x, 0) \sim q(x)
$$

\begin{conj} \label{conj:deform}
	There are isomorphisms $ \widehat \CD_n \cong \CF_n(\BR) $ and $ \widehat \CP_n \cong \overline \Cf_n(\BR) $ compatible with the projections to $ \BR $, extending the isomorphisms $ \widehat D_n \cong \overline F_n(\BR) $ and $ \widehat X_n \cong \overline \ft_n(\BR) $ and making the following diagram commute
	$$
		 		\begin{tikzcd}
		\widehat \CD_n \arrow[r] \arrow[d] & \CF_n(\BR) \arrow[d] \\
		\widehat \CP_n \arrow[r] & \overline \Cf_n(\BR) 	
	\end{tikzcd}
$$
\end{conj}

\subsection{Deformation retraction}
It would follow from Conjecture \ref{conj:deform} that $\overline F_n(\BR) $ is a deformation retraction of $ \CF^\sigma_n(\BR) $.  However, we will now prove this fact independent of the conjecture.  We begin with the following lemma.  We thank Yibo Ji for suggesting this lemma and its proof, which was inspired by a result of Slodowy \cite[Section 4.3]{Slo}.

\begin{lem} \label{le:retract}
	Let $ X $ be a CW complex, equipped with a proper map $ p : X \rightarrow \BR $. Assume that $X_0 = p^{-1}(0) $ is a subcomplex of $ X $.  Assume also that we have an action of $ \BR^\times $ on $ X $ such that $ p $ is equivariant (with the usual action of $ \BR^\times $ on $ \BR $).
	
	Then $ X_0 $ is a deformation retract of $ X $.
\end{lem}

\begin{proof}
	Since $ X_0 $ is a subcomplex of $ X $, there exists a precompact open neighbourhood $ U \subset X $ of $ X_0 $ that deformation retracts onto $X_0$ \cite[Prop~A.5]{Hat}.
	
	Since $ p $ is proper, we see that $ p(X \setminus U) $ is closed in $ \BR $.  As this closed set does not contain~$ 0 $, we see that there exists $ a > 0 $ such that $ [-a,a] $ is disjoint from $ p(X \setminus U) $ and hence that $ U \supset X_{\le a} := p^{-1}([-a,a])$.
	
	We claim that $ X_{\le a}$ is homotopy equivalent to $ X $.  To prove this, we define $ H : X \times [0,1] \rightarrow X $ by
	$$
	H(x,t) = \begin{cases} x \text{ if $|p(x)| \le a $} \\
		t \cdot x \text{ if $ \frac{a}{|p(x)|} \le t \le 1 $} \\
		\frac{a}{|p(x)|} \cdot x \text{ if $ 0 \le t \le \frac{a}{|p(x)|}$}
	\end{cases}
$$
where $ t \cdot x $ denotes the $ \BR^\times $ action. This provides a deformation retraction of $ X $ onto $ X_{\le a} $.

Consider the chain of inclusions $ X_0 \subset X_{\le a} \subset U \subset X $.  We have induced maps on homotopy groups
$$
\pi_i(X_0) \rightarrow \pi_i(X_{\le a})  \rightarrow  \pi_i(U) \rightarrow \pi_i(X)
$$
Since the composite maps  $ \pi_i(X_0) \rightarrow \pi_i(X_{\le a})  \rightarrow  \pi_i(U)$ and $ \pi_i(X_{\le a})  \rightarrow  \pi_i(U) \rightarrow \pi_i(X) $ are isomorphisms, we conclude that every map in this sequence is an isomorphism.

Thus the inclusion of $ X_0 $ into $ X $ is a weak homotopy equivalence and hence there is a deformation retraction of $ X $ onto $ X_0 $ \cite[Thm~4.5]{Hat}.
\end{proof}

Now we apply the lemma to our situation to deduce our desired result.

\begin{thm} \label{th:deformretract}
	$\overline F_n(\BR) $ is a deformation retract of $ \CF_n^\sigma(\BR) $.
\end{thm}

\begin{proof}
	We must check all the hypotheses of the Lemma \ref{le:retract}.  We have a map $ f : \CF_n^\sigma(\BR) \rightarrow i \BR $ provided by the $ \vareps $ coordinate.  This map is proper by Corollary \ref{cor:proper} (that statement refers to properness in the sense of algebraic geometry, but the same argument implies properness on the level of real points).
	
	Next, since $ \CF^\sigma_n(\BR) $ is an algebraic variety and $ \overline F_n(\BR) $ is a subvariety, the pair is triangularizable (see for example \cite{Sato}).
	
	Finally, we have an $ \BR^\times $ action on $ \CF^\sigma_n(\BR) $ provided by Remark \ref{rem:scaling3}, which is compatible with the standard action on $ i\BR$.
	
	Thus all the hypotheses of Lemma \ref{le:retract} hold, so we deduce the desired deformation retraction.
\end{proof}

	Rather than the twisted real form, we can also consider the standard real form $ \CF_n(\BR) $.  The $ \Cx $ action on $ \CF_n $ again restricts to a $ \BR^\times $ action on $ \CF_n(\BR) $ and the above proof goes through in exactly the same manner to give the following.
	
\begin{thm}
		$\overline F_n(\BR) $ is a deformation retract of $ \CF_n(\BR) $.
\end{thm}

\section{Affine and virtual cactus and symmetric groups}

We begin by defining the relevant groups.

	\subsection{Affine symmetric group}
\label{sec:affine}
	
	The \textbf{affine symmetric group} $ AS_n $ is the group of all permutations $ f : \BZ \rightarrow \BZ $ such that $ f(a +n) = f(a) + n $ and $ \sum_{i=1}^n f(i) = \binom{n+1}{2} $.
	
	Given $ f \in AS_n $, we define a permutation $ \bar f $ of $ \BZ / n = \{1, \dots, n\} $  by setting $ \bar f(\bar k) = \overline{f(k)} $ (here $ \bar k $ denotes the image of $ k $ in $ \BZ/n $).  This defines a group homomorphism $ AS_n \rightarrow S_n $.
	
	Define the $ \sl_n $ root lattice $$ \BZ^n_0 = \{ (k_1, \dots, k_n) \in \BZ^n : k_1 + \dots + k_n = 0 \} $$
	We have an injective group homomorphism $ \BZ^n_0 \rightarrow AS_n $ defined by $ \uk \mapsto f_\uk $ where
	$$
	f_\uk( i + nm) = i + n k_i + nm \ \text{ where } i \in \{1, \dots, n\}
	$$
	It is easy to see that the image of $ \BZ^n_0 $ is the kernel of $ AS_n \rightarrow S_n $.
	
	Moreover, we have a splitting of $ AS_n \rightarrow S_n $ by defining $ S_n \rightarrow AS_n $ given by $ \sigma \mapsto f_\sigma $ where
	$$
	f_\sigma( i + nm) = \sigma(i) + nm  \ \text{ where } i \in \{1, \dots, n\}
	$$
	This shows that $ AS_n $ is a semi-direct product $ AS_n = S_n \ltimes \BZ^n_0 $.

	By a slight abuse of notation, we will write $ \sigma_i := f_{\sigma_i} $ where $ \sigma_1, \dots, \sigma_{n-1} $ are the usual generators of $ S_n$.  So we have
	$$
	\sigma_k(i + nm) = \begin{cases} i+1 + nm \text{ if $ i \cong k \mod n $} \\
		i-1 + nm \text{ if $ i \cong k+1 \mod n $}  \\
		i + nm \text{ otherwise}
		\end{cases}
	$$
	So it is natural to extend the definition to $ \sigma_0 $ by
	$$
		\sigma_0(i + nm) = \begin{cases} i+1 + nm \text{ if $ i \cong 0 \mod n $} \\
		i-1 + nm \text{ if $ i \cong 1 \mod n $}  \\
		i + nm \text{ otherwise}
	\end{cases}
	$$
	
The following result is well-known.
	
\begin{lem} \label{lem:AS1}
		$ AS_n $ is a Coxeter group with generators $ \sigma_0, \dots, \sigma_{n-1} $ and relations given by the affine type $ A_n $ Dynkin diagram.	\end{lem}

There is an action of $ \BZ/n $ on $AS_n $ by $ (r \cdot f)(a) = f(a-1) + 1 $, where $ r \in \BZ/n $ is the generator, $ f \in AS_n $, and $ a \in \BZ $.  We let $ \widetilde{AS}_n := AS_n \rtimes \BZ/n $ be the semi-direct product, which we call the \textbf{extended affine symmetric group}.

We will also use $ r = (12 \dots n) $ denote the long cycle in $ S_n$, defined by $ r(a) = a+1$ for $ a < n $ and $ r(n) = 1$.

\begin{lem} \label{lem:AS2}
	We have $ \widetilde{AS}_n \cong S_n \ltimes \BZ^n/\BZ $.
\end{lem}
\begin{proof}
	Define a homomorphism $ \widetilde{AS}_n = AS_n \rtimes \BZ/n  \rightarrow S_n $ by $ (f, r^k) \mapsto \bar f r^k $.  Let $ K $ denote the kernel.  Note that $ \sigma \mapsto f_\sigma $ also splits this map, so we have $ \widetilde{AS}_n \cong S_n \ltimes K $.
	
	We claim that there is an isomorphism $\BZ^n/\BZ \cong K $.  To define this, we extend the definition of $ f_{\uk} $ for $ \uk \in \BZ^n $ by
	$$
	f_\uk( i + nm) = i + n k_i - \sum k_i + nm \ \text{ where } i \in \{1, \dots, n\}
	$$
	and then we map $ \uk \mapsto (f_\uk, r^{\sum k_i}) \in K $.
\end{proof}
	
	\subsection{Intervals}
		Consider the set $ \BZ / n = \{1, \dots, n\} $ with its cyclic order $ 1 < \dots < n < 1 $.  Given an ordered pair $  1 \le i,j \le n $ with $ i \ne j $, we consider the interval $ [i,j] = \{i < i+1 < \dots < j\} $ in this cyclic order. Each interval carries a total order.  We write $ [k,l] \subset [i,j] $ (and say that it is a subinterval) if there is a containment which preserves the orders.
		
		\begin{eg} \label{eq:n3intervals} Consider the case $ n = 3$.  Note that $[3,1] =\{ 3 < 1 \} $ is not considered a subinterval of $ [1,3] = \{1< 2 <3 \} $.  On the other hand $ [1,2] = \{ 1 < 2 \} $ is a subinterval of $ [3,2] = \{ 3 < 1 < 2 \} $.
		\end{eg}
	
	 Given $ i,j$, we define $ w_{ij} \in S_n $ to be the permutation which reverses the elements of $ [i,j] $ and leaves invariant the elements outside $ [i,j] $.
	
	  Under $ \BZ \rightarrow \BZ/n $, the preimage of $ [i,j] $ is a union of intervals in $ \BZ $, each ordered according to the usual order on $ \BZ $.  We define $ \hat w_{ij} \in AS_n $ to be the permutation which reverses each of these intervals and leaves invariant all elements outside these intervals.  As the notation suggests, $ \hat w_{ij} $ is a lift of $ w_{ij} $ with respect to $ AS_n \rightarrow S_n $.
	
	 \begin{eg}
	 	In $ S_2$ and $ S_3 $, we have an equality $ w_{ij} =w_{ji} $ for all $i,j$.  But in general this is not true.  For example $ w_{41} \in S_4$ is the transposition $(14) $, while $w_{14} \in S_4 $ is the longest element, which is the product of two transpositions $ (14)(23)$.
	 	\end{eg}
	
	An interval $ [i,j] $ is called \textbf{standard} if $ i < j $ in the usual order on $ [1,n]$.  If $ [i,j] $ is a standard interval, then $ \hat w_{ij} $ is the image of $ w_{ij} $ under the embedding $ S_n \rightarrow AS_n$, but otherwise it is not.
	
	\subsection{Affine cactus group} \label{se:ac}  Recall the definition of the cactus group from Section~\ref{se:introreal}.
	We now define an affine version of the cactus group.
	
\begin{defn}	
\label{def:affine}
	 The \textbf{affine cactus group} $ AC_n $ is the group with generators $ s_{ij} $ for $  1 \le i \ne j \le n $ and relations
	 \begin{enumerate}
	 	\item $ s_{ij}^2 = 1 $
	 	\item $ s_{ij} s_{kl} = s_{kl} s_{ij} $ if $ [i,j] \cap [k,l] = \emptyset $
	 	\item $ s_{ij} s_{kl} = s_{ w_{ij}(l)  w_{ij}(k)} s_{ij} $ if $ [k,l] \subset [i,j] $
	 \end{enumerate}
\end{defn}
Note that as compared with the usual cactus group, here we do not impose the condition $ i < j $ and so we allow non-standard intervals.

We have a group homomorphism $ AC_n \rightarrow AS_n $ taking $ s_{ij} $ to $\hat w_{ij}$ and thus by composition, there is a group homomorphism $ AC_n \rightarrow S_n $ taking $ s_{ij} $ to $ w_{ij}$.

If we just consider those generators corresponding to standard intervals, then we have the generators and relations of the usual cactus group $C_n $ and thus we have a homomorphism $ \psi_n: C_n \rightarrow AC_n $.

\begin{eg}
	The group $ AC_2 $ has two generators $ s_{12} $ and $ s_{21} $, with the only relations that they square to the identity.  Thus it is the infinite dihedral group.  The homomorphism $ AC_2 \rightarrow S_2 $ takes each generator to the non-trivial element of $ S_2 $.  The kernel is just the free group on $ s_{12} s_{21} $.  In this case, the map $ AC_2 \rightarrow AS_2 $ is an isomorphism.
\end{eg}

\begin{eg}
	The group $AC_3$ has six generators, coming from the intervals $[1,2], [2,3], [3,1] $ of size 2 and the intervals $ [1,3], [2,1], [3,2] $ of size 3.  Each generator is an involution and they satisfy
	$$
	s_{13} s_{12} = s_{23} s_{13} \quad s_{21}s_{23} = s_{31} s_{21} \quad s_{32} s_{31} = s_{12} s_{32}
	$$
	There is no relation between the generators $s_{31} $ and $ s_{13}$ (see Example \ref{eq:n3intervals}).
\end{eg}

We may think of $ AC_n $ as the cactus group associated to the affine type A Dynkin diagram.

As for $ AS_n $, we define an action of $ \BZ/n $ on $ AC_n $ by $ r \cdot s_{ij} = s_{i+1\, j+1} $ (where addition is considered modulo $n $).  As before, we define the \textbf{extended affine cactus group} $\widetilde{AC}_n$ to be the semi-direct product $ AC_n \rtimes \BZ/n $.  The homomorphism $ AC_n \rightarrow AS_n $ extends to a homomorphism $ \widetilde{AC}_n \rightarrow \widetilde{AS}_n$.

\subsection{Virtual symmetric group}
We define the \textbf{virtual symmetric group} $vS_n$ to be the free product of two copies of the symmetric group $S_n$ modulo the relation
\begin{equation} \label{eq:vSndef}
w \sigma_i w^{-1} = \sigma_{w(i)}
\end{equation}
for all $ 1 \le i\leq n-1 $ and $ w \in S_n $ such that $ w(i+1) = w(i) + 1$. Here $ \sigma_i = (i \, i+1)$
are generators of the first copy of $ S_n $ and $ w $ is from the second copy.

\begin{rem}
	In the literature (see for example \cite{PeterLee}), the \textbf{flat virtual braid group} is defined by imposing (\ref{eq:vSndef}) only for $ w  $ which are 3-cycles and 2-cycles of consecutive elements.  It is not difficult to prove that the more general relations that we consider follow from these special cases.  Thus the virtual symmetric group is the same as the flat virtual braid group.
\end{rem}

\subsection{Virtual cactus group}
Suppose that $1 \le  i< j \le n$.  We say that $ w \in S_n $ is a \textbf{translation} on $[i,j] $, if $ w(i+k) = w(i)+k$ for $ k = 1,\dots, j-i$.  Note that given $ w(i) $, the subset of such $ w $ is naturally in bijection with $ S_{n-(j-i+1)} $.

We define the \textbf{virtual cactus group} $vC_n $ to be the quotient of the free product $ C_n * S_n $ by the relations
$$
w s_{ij} w^{-1} = s_{w(i) w(j)}
$$
for all $ 1 \le i < j \le n $ and $ w \in S_n $, such that $ w $ is a translation on $[i,j]$.

Note that there is a group homorphism $ vC_n \rightarrow vS_n $ extending the usual map $ C_n \rightarrow S_n $ and the identity on $ S_n $.

\subsection{A diagram of groups}

In $ vS_n $, we have $ r \sigma_i = \sigma_{i+1} r $ for $1\leq i \leq n-1 $.

In $ vC_n$, we have $ r s_{ij} = s_{i+1 \, j+1} r$  for $1\leq i <j \leq n-1$. Consequently, for $1\leq i <j \leq n-p$, where $p\geq 1$, we have $ r^p s_{ij} = s_{i+p \, j+p} r^p$.

We define a group homomorphism $ AS_n \rightarrow vS_n $ by
$$
\sigma_i \mapsto  \sigma_i, \text{ if $ i \ne 0 $, and } \sigma_0 \mapsto r^{-1} \sigma_1 r
$$

We define a group homomorphism $ AC_n \rightarrow vC_n $ on generators as follows
 $$
s_{ij} \mapsto \begin{cases} s_{ij} \text{ if $ i < j $} \\
	r^{i-1} s_{ij\ominus(i-1)} r^{1-i} \text{ if $ j < i $},
\end{cases}
$$
where we define the index $ij\oplus k$ (resp.\  $ij\ominus k$) as $i+k \, j+k$, (resp.\ $i-k \, j-k$), where the addition is modulo $n$, and where we write $n$ instead of $0$. In particular, the index $ij\ominus(i-1)$ is $1 \, j-(i-1)+n$. In fact, for $i<j$ the formula $s_{ij}\mapsto r^{i-1} s_{ij\ominus(i-1)} r^{1-i}$ holds as well.

These extend to group homomorphisms $ \widetilde{AS}_n \rightarrow vS_n$ and $ \breve \psi_n : \widetilde{AC}_n \rightarrow vC_n $ taking $ r $ (the generator of $ \BZ/n$) to $ r $ (the long cycle). Furthermore, the homomorphisms $ {AS}_n \rightarrow S_n, {AC}_n \rightarrow S_n$ extend to homomorphisms $ \widetilde{AS}_n \rightarrow S_n, \widetilde{AC}_n \rightarrow S_n$ by mapping $r$ to the long cycle.

\begin{thm}
\label{thm:groups}
	These are group homomorphisms and fit into the commutative diagram
	\begin{equation} \label{eq:groups}
	 		\begin{tikzcd}
C_n \arrow[r, "\psi_n"] \arrow[d] & \widetilde{AC}_n \arrow[r,"\breve \psi_n"] \arrow[d] & vC_n \arrow[d] \\
S_n \arrow[r] & \widetilde{AS}_n \arrow[r] & vS_n 	
 	 \end{tikzcd}
  \end{equation}
Moreover, these homomorphisms are compatible with the projections to $ S_n $.
\end{thm}

\begin{proof}
The only difficult part is to check that $ \widetilde{AC}_n \rightarrow vC_n $ is well defined.  For this, we must check the relations of the affine cactus group.  Relations (1) and (2) are easy to verify. We verify relation (3), which is $s_{ij} s_{kl} = s_{l'k'} s_{ij}$ for $l'=w_{ij}(l),k'=w_{ij}(k)$. The image of $ s_{ij} s_{kl}$ is
\begin{align*}
		&r^{i -1} s_{ij\ominus (i-1)} r^{1-i} r^{k-1}s_{kl\ominus (k-1)}r^{1-k}\\		
		=\ & r^{i-1} s_{ij\ominus (i-1)} s_{kl\ominus (i-1)} r^{1-i} \\
		=\  &r^{i-1}s_{l'k'\ominus (i-1)}s_{ij\ominus (i-1)} r^{1-i}\\
		=\ &r^{l'-1}s_{l'k'\ominus (l'-1)}r^{1-l'}r^{i-1}s_{ij\ominus (i-1)} r^{1-i},
\end{align*}
which is exactly the image of $s_{l'k'} s_{ij}$.
%
%
	
\end{proof}

\subsection{Pure virtual groups}
We define the pure virtual cactus and symmetric groups, $ PvC_n$, $PvS_n $ to be the kernels of the homomorphisms $ vC_n \rightarrow S_n, vS_n \rightarrow S_n $.

For each pair $i,j $, we define an element $ \sigma_{ij} \in PvS_n $ by
 $ \sigma_{ij} = w \sigma_k w_{k\, k+1} w^{-1} $,
where $ w \in S_n$, $ 1 \le k < n $ and $ w(k) = i, w(k+1) = j $.

\begin{lem}
\label{lem:pres_vS}
	\begin{enumerate}
		\item $ \sigma_{ij} $ is independent of the choice of $ w, k $ above.
		\item These elements satisfy the relations
		$$\sigma_{ij} \sigma_{ji} = 1 \quad  \sigma_{ij} \sigma_{lm} = \sigma_{ml} \sigma_{ij} \quad \sigma_{ij}\sigma_{il} \sigma_{jl} = \sigma_{jl} \sigma_{il} \sigma_{ij}
		$$
		for all distinct $ i, j, l,m$.
		\item $ PvS_n $ is generated by $ \sigma_{ij} $ subject to the above relations.
		\item $ vS_n = S_n \ltimes PvS_n $ where $ S_n $ permutes the generators $ \sigma_{ij} $ in the obvious way.
	\end{enumerate}
\end{lem}

\begin{rem}
	This Lemma shows that $ PvS_n $ is isomorphic to the triangle group $ Tr_n $ defined in \cite{BEER}.
\end{rem}

\begin{proof}
For (1), suppose that we have $w' \in S_n $ with $ w'(k') = i, w'(k'+1) = j$. Then $w^{-1}w'$ sends $k',k'+1$ to $k,k+1$. Consequently, $w^{-1}w'$ conjugates $w_{k'\, k'+1}$ to $w_{k\, k+1}$ and $\sigma_{k'}$ to $\sigma_k$, and thus $\sigma_kw_{k\, k+1} w^{-1}w' =w^{-1}w' \sigma_{k'}w_{k'\, k'+1}$, as desired.

Parts (2) and (3) follow by the Reidemeister--Schreier procedure, see e.g.\ \cite[Prop~II.4.1]{LS} (where $T=S_n$). Namely, we consider a set $X$ with elements $x_{ij}$ corresponding to $\sigma_{ij}$. For every relation in $vS_n$, say $\sigma_2\sigma_1\sigma_2\sigma_1\sigma_2\sigma_1$, we consider all the expressions $w\sigma_1\sigma_2\sigma_1\sigma_2\sigma_1\sigma_2w^{-1}$ for $w\in S_n$. Each such word is the image under $x_{ij}\to w \sigma_k w_{k\, k+1} w^{-1}$, after reductions in $S_n$, of a word in the alphabet $X$. The first three terms in our example will be
$$\big(w\sigma_1w_{12}w^{-1}\big)\big((ww_{12})\sigma_2w_{23}(ww_{12})^{-1}\big)\big((ww_{12}w_{23})\sigma_1w_{12}(ww_{12}w_{23})^{-1}\big),$$
which are the image of $x_{ij}x_{il}x_{jl}$ for $w(1)=i, w(2)=j,w(3)=l$. The group $ PvS_n $ is presented by these relations over $X$.
The three types of relations in (2) come from the relations between~$\sigma_k$. The relation $w \sigma_i w^{-1} = \sigma_{w(i)}$ is already taken into account by the identifications between different $w \sigma_k w_{k\, k+1} w^{-1}$.

For part (4), let $\sigma_{ij}= w  \sigma_k w_{k\, k+1}  w^{-1} $ and let $u\in S_n$. Then $uw(k)=u(i),uw(k+1)=u(j)$. Consequently, $u\sigma_{ij}u^{-1}=(uw)\sigma_k w_{ij} (uw)^{-1}=\sigma_{u(i)u(j)}$, as desired.
\end{proof}

An \textbf{ordered subset} of $ [n] $ is a sequence $ A= (a_1,\dots, a_k) $ of distinct elements of $ [n]$.  The \textbf{reverse} of $ A $ is the ordered subset $ A^r = (a_k, \dots, a_1) $. Finally, $AB$ denotes the concatenation of the sequences $A$ and $B$.

For each ordered subset $ A = (a_1, \dots, a_k) $ of $ [n]$, we define $ s_A \in PvC_n$ by $ s_A = w s_{ij} w_{ij} w^{-1} $, where $w \in S_n $, $ 1 \le i < j \le n $ and $ w(i) = a_1, w(i+1) = a_2, \dots, w(j) = a_k $.

\begin{lem}
\label{le:present}
	\begin{enumerate}
		\item $ s_A $ is independent of the choice of $ w, i, j $ above.
		\item These elements satisfy the relations
		$$s_A s_{A^r} = 1, \quad s_As_B=s_Bs_A, \quad s_{A^r} s_{C A B} = s_{C A^rB} s_A
		$$
		for any disjoint ordered subsets $ A, B, C$.  		\item $ PvC_n $ is generated by the elements $ s_A $ subject to the above relations.
		\item $ vC_n = S_n \ltimes PvC_n $ where $ u\in S_n $ maps each generator $ s_A $  to $s_{u(A)}$, where $u\big((a_1,\ldots, a_k)\big)=\big(u(a_1),\ldots, u(a_k)\big)$.
	\end{enumerate}
\end{lem}
\begin{proof}
For (1), suppose that we have $w' \in S_n $ with $ w'(i') = a_1, w'(i'+1) = a_2, \dots, w'(j') = a_k $. Then $w^{-1}w'$ is a translation on $[i'j']$ and sends it to $[ij]$. Consequently, $w^{-1}w'$ conjugates $w_{i'j'}$ to $w_{ij}$ and $s_{i'j'}$ to $s_{ij}$, and thus  $s_{ij} w_{ij} w^{-1}w' =w^{-1}w' s_{i'j'}w_{i'j'} $, as desired.

Parts (2) and (3) follow by the Reidemeister--Schreier procedure from the three types of relations in the definition of the usual cactus group 
as in the proof of Lemma~\ref{lem:pres_vS}.

 For part (4), let $ s_A = w s_{ij} w_{ij} w^{-1} $ and let $u\in S_n$. Then $(uw(i),\ldots uw(j))=u(A)$. Consequently, $us_Au^{-1}=(uw)s_{ij} w_{ij} (uw)^{-1}=s_{u(A)}$, as desired.
\end{proof}

\section{Fundamental groups}

\subsection{Equivariant fundamental groups}
	Let $ G $ be a finite group acting on a path-connected, locally simply-connected space $ X $.  Let $ x \in X $ be a basepoint.

\begin{defn}
	The $G$-\textbf{equivariant fundamental group} $ \pi_1^G(X, x) $ is defined as follows.
	$$
	\pi_1^G(X, x) = \{ (g,p) : g \in G, \text{ $ p $ is a homotopy class of paths from $ x $ to $ gx $} \}
	$$
	The multiplication in $ \pi_1^G(X, x) $ is defined as follows.  We define $$ (g_1, p_1) \cdot (g_2, p_2) =  (g_1 g_2,p_1*g_1(p_2)) $$ where $ * $ denotes concatenation of paths.
\end{defn}
The map $ (g,p) \mapsto g $ defines a group homomorphism $ \pi_1^G(X) \rightarrow G $ and there is a short exact sequence of groups
$$
1 \rightarrow \pi_1(X) \rightarrow \pi_1^G(X) \rightarrow G \rightarrow 1
$$

Our combinatorial spaces $ P_n, D_n, \dots $ all carry $ S_n $ actions. Indeed, the group $S_n $ acts on the set of planar forests by permuting the labels.  Thus, it acts on the cube complexes
 $D_n,\breve D_n,\widehat D_n$.  We also have evident actions of $ S_n $ on the permutahedron $ P_n $ and its quotients $ \breve P_n, \widehat P_n$.
Thus we will consider their $S_n $-equivariant fundamental groups. From Proposition \ref{prop:CubeStar} and Lemma \ref{lem:CubeStarDescend}, we obtain the following  commutative diagram.

\begin{equation}\label{eq:spaces}
	\begin{tikzcd}
				\pi_1^{S_n}(D_n) \arrow[r] \arrow[d] & \pi_1^{S_n}(\breve D_n) \arrow[r] \arrow[d] & \pi_1^{S_n}(\widehat D_n) \arrow[d] \\
		\pi_1^{S_n}(P_n) \arrow[r] & \pi_1^{S_n}(\breve P_n) \arrow[r] & \pi_1^{S_n}(\widehat P_n)  \\
	\end{tikzcd}
\end{equation}

\begin{rem} Note that all these spaces are nonpositively curved. Indeed, for $D_n,\breve D_n,$ and~$\widehat D_n$ this is Lemma~\ref{le:npc}. For $\widehat P_n$, this is Remark~\ref{rem:Pnpc}, and for $\breve P_n$, this follows from Proposition~\ref{prop:homeobreve1}. Thus all these spaces are aspherical \cite[II.4.1(2)]{BH}, and so their higher homotopy groups vanish. \end{rem}

\begin{thm} \label{th:pi1comb}
	There are isomorphisms
	\begin{gather*} C_n \cong \pi_1^{S_n}(D_n) \quad \widetilde{AC}_n \cong \pi_1^{S_n}(\breve D_n) \quad vC_n \cong \pi_1^{S_n}(\widehat D_n) \\
	S_n \cong \pi_1^{S_n}(P_n) \quad \widetilde{AS}_n \cong \pi_1^{S_n}(\breve P_n) \quad vS_n \cong \pi_1^{S_n}(\widehat P_n) 	
	\end{gather*}
such that the two commutative diagrams (\ref{eq:groups}) and (\ref{eq:spaces}) match.
\end{thm}


Before proceeding to the proof of this result, note that it has the following consequence.

\begin{cor}
	The group homomorphisms $ C_n \rightarrow \widetilde {AC}_n $ and $ \widetilde{AC}_n \rightarrow vC_n $ are injective.
\end{cor}

\begin{proof}
By Theorem~\ref{th:pi1comb}, we need to show that the maps
$$\pi_1^{S_n}(D_n) \to \pi_1^{S_n}(\breve D_n) \to \pi_1^{S_n}(\widehat D_n)
$$
are injective. It suffices to show that the maps between the fundamental groups of these complexes are injective. This follows from Lemma~\ref{le:loc_iso}.
\end{proof}

\subsection{Fundamental groups of the combinatorial spaces}
\begin{lem}
\label{le:pi1_of_hatD} $\pi_1^{S_n}(\widehat D_n)=vC_n$
\end{lem}
\begin{proof} Let $\Gamma^1(PvC_n)$ be the Cayley graph of $PvC_n$ with respect to the generators $\{s_A\}$ from Lemma~\ref{le:present}. Note that $PvC_n \backslash \Gamma^1(PvC_n)$ is isomorphic with the $1$-skeleton of $\widehat D_n$ under the map sending the orbit of the directed $1$-cubes of the form $(g,gs_A)$ to the directed $1$-cube corresponding to $A$. Furthermore, the relators of length $4$ from the presentation in Lemma~\ref{le:present}(2) are sent, bijectively, to the boundary paths of $2$-cubes. Consequently, we have $\pi_1(\widehat D_n)=PvC_n$. By Lemma~\ref{le:present}(4), we have that $S_n$  permutes the generators $s_A$ of $PvC_n$ exactly as it acts on the $1$-cubes of $\widehat D_n$, by interchanging the corresponding~$A$, and the lemma follows.
\end{proof}

\begin{lem}
\label{le:pi1_of_hatP} $\pi_1^{S_n}(\widehat P_n)=vS_n$
\end{lem}
\begin{proof} Let $\Gamma^1(PvS_n)$ be the Cayley graph of $PvS_n$ with respect to the generators $\{\sigma_{ij}\}$ from Lemma~\ref{lem:pres_vS}. Note that $PvC_n \backslash \Gamma^1(PvC_n)$ is isomorphic with the $1$-skeleton of $\widehat P_n$ under the map sending the orbit of the directed $1$-cells of the form $(g,g\sigma_{ij})$ to the directed $1$-cell corresponding to the pair $ij$.
Furthermore, the relators from the presentation in Lemma~\ref{lem:pres_vS}(2) are sent, bijectively, to the boundary paths of $2$-cells. Consequently, we have $\pi_1(\widehat P_n)=PvS_n$ (this was also proved in \cite[Thm~8.1]{BEER}). By Lemma~\ref{lem:pres_vS}(4), we have that $u\in S_n$ acts on the generators $\sigma_{ij}$ of  $PvS_n$ by replacing $ij$ with $u(i)u(j)$ exactly as it does on the $1$-cells of $\widehat P_n$, since $uw(k)=u(i),uw(k+1)=u(j)$. \end{proof}


Recall the homomorphism $\breve \psi_n : \widetilde{AC}_n\to vC_n$ from diagram (\ref{eq:groups}).

\begin{lem}
\label{le:pi1_of_tildeD} We have $\pi_1^{S_n}(\breve D_n)=\widetilde {AC}_n$, and $\breve \phi_n \colon \breve D_n\to \widehat D_n$ induces $\breve \psi_n$ between their equivariant fundamental groups.
\end{lem}
\begin{proof}
Let $\Gamma^2(AC_n)$ be the Cayley $2$-complex of the affine cactus group with the presentation from Definition~\ref{def:affine}. More precisely,
\begin{itemize}
\item
the set of the $0$-cubes of $\Gamma^2(AC_n)$ is $AC_n$,
\item
we join by single $1$-cube, labelled by $s_{ij}$, the $0$-cubes $g$ and $gs_{ij}$, and
\item
we span $2$-cubes on the closed edge-paths of length four labelled by the words of length four in points (2) and (3) of the presentation.
\end{itemize}
Then $\Gamma^2(AC_n)$ is simply connected. The affine cactus group $AC_n$ acts on $\Gamma^2(AC_n)$ by left multiplication. Identifying the $0$-cubes of $\Gamma^2(AC_n)$ with
the cosets $g\BZ/n$ in $\widetilde {AC}_n$ by assigning to each $g\in AC_n$ the coset
$g\BZ/n$, this action extends to an action of $\widetilde {AC}_n$ (by left multiplication of the cosets).

Let $P\widetilde {AC}_n$ be the kernel of the homomorphism $\widetilde{AC}_n\to S_n$.
Note that $P\widetilde {AC}_n$ acts freely on the $0$-cubes of $\Gamma^2(AC_n)$, since $P\widetilde {AC}_n\cap \BZ/n=\emptyset.$ Similarly, since (for $n\geq 3$, leaving the case $n=2$ for the end) the images of all the generators of ${AC}_n$ in $S_n$  lie outside the subgroup generated by the long cycle, the action of $P\widetilde {AC}_n$ is free on the $1$-cubes of $\Gamma^2(AC_n)$. It remains to justify that $P\widetilde {AC}_n$ acts freely on the $2$-cubes of $\Gamma^2(AC_n)$.

Indeed, assume that a nontrivial element of $P\widetilde {AC}_n$ maps a $0$-cube to an opposite $0$-cube inside the same $2$-cube, connected by an edge-path of two $1$-cubes of types corresponding to the generators $s_{ij},s_{kl}$. For $[ij]$ disjoint or properly containing $[kl]$, the composition $w_{ij}w_{kl}\in S_n$ lies in $\BZ/n$ only if (up to a cyclic permutation of indices) $[ij]=[1n]$ and $[kl]=[1(n-1)]$ or $[kl]=[2n]$. However, the elements of $\widetilde {AC}_n$ permute cyclically the types of $1$-cubes in $\Gamma^2(AC_n)$  corresponding to the generators $s_{1(n-1)},s_{2n},s_{31}$ etc. Fixing the above $2$-cube would mean interchanging the first two of these types, which is a contradiction.

If $n=2$, we have only $1$-cubes, and the action of nontrivial $g\in P\widetilde {AC}_n$ is free on them since otherwise $g$ would coincide with a conjugate of $s_{12}$ or $s_{21}$ that do not belong to $P\widetilde{AC}_n$.

Thus $P\widetilde{AC}_n=\pi_1(\breve D^2_n)$, where we define $\breve D^2_n$ as $P\widetilde {AC}_n\backslash \Gamma^2(AC_n)$. Note that $\breve D^2_n$ is equipped with the action of $S_n=P\widetilde {AC}_n\backslash \widetilde {AC}_n$ by left multiplication.

We will now show how to identify $\breve D^2_n$ as the $2$-skeleton of $\breve D_n$. The $0$-cubes of $\breve D^2_n$ are $P\widetilde {AC}_n$-orbits of the $0$-cubes in $\Gamma^2(AC_n)$, and thus double cosets
$P\widetilde {AC}_ng\BZ/n$. Denoting by $\bar g$ the image of $g$ in $S_n$, these double cosets correspond to the cosets $\bar g\overline \BZ/n$, where $\overline  \BZ/n$ denotes the cyclic subgroup generated by the long cycle in $S_n$.

Note that the $0$-cubes of $\breve D_n$ corresponded to cyclic forests $\tau$ with trees having only one edge. Thus $\tau$ corresponded to the cyclic orders of their leaves in $[n]$, hence with the cosets $\bar g\overline \BZ/n$. The $1$-cubes of $\breve D_n$ connected $\tau,\tau'$ differing by reversing the cyclic order in a (cyclic) interval $[i,j]$. Thus $\tau,\tau'$ corresponded to
$\bar g\overline \BZ/n, \bar gw_{ij}\overline \BZ/n$, which are also connected by a $1$-cube in $\breve D^2_n$ that is the image of $1$-cubes $g\BZ/n, gs_{ij}\BZ/n$ in $\Gamma^2(AC_n)$. The $2$-cubes of $\breve D_n$ and $\breve D^2_n$ are identified analogously.

Consequently, we have $\pi_1(\breve D_n)=P\widetilde {AC}_n$. Since the actions of $S_n$ on the $0$-cubes of both $\breve D_n$ and $\breve D^2_n$ are by the left multiplication of the cosets $\bar g\overline \BZ/n$, we also have $\pi_1^{S_n}(\breve D_n)=\widetilde {AC}_n$.

We denote by $\Gamma(AC_n)$ the universal covering space of $\breve D_n$, which contains $\Gamma^2(AC_n)$ as its 2-skeleton.  Let $\breve{\id}$ be the base $0$-cube of $\Gamma(AC_n)$, that is, the identity element of $AC_n$ (or the trivial coset of $\BZ/n$). Note that $\breve{\id}$ projects to the $0$-cube of $\breve D_n$ with the trivial cyclic order $(1,2,\ldots, n)$. Let $\Gamma(PvC_n)$ be the universal covering space of $\widehat D_n$, which contains $\Gamma^1(PvC_n)$. Finally, let $\widetilde \phi_n\colon \Gamma(AC_n)\to \Gamma(PvC_n)$ be the map covering $\breve \phi_n$ with $\widehat{\id}=\widetilde\phi_n(\breve{\id})$ the base $0$-cube of $\Gamma(PvC_n)$, that is, the identity element of $PvC_n$.

We will now show that with our identifications the homomorphism $(\breve \phi_n)_*$ induced by $\breve \phi_n$ equals $\breve\psi_n$. It suffices to verify that $(\breve \phi_n)_*=\breve \psi_n$ on the generators $s_{1j},r$ of $\widetilde {AC}_n$, where $1<j\leq n$. First note that $(\breve \phi_n)_*$ respects the homomorphisms of $\widetilde {AC}_n$ and $vC_n$ into $S_n$, since $\breve \phi_n$ is $S_n$-equivariant.
Thus, since $r$ fixes $\breve {\id}$, we have that $(\breve \phi_n)_*(r)$ is the unique element in the stabiliser $S_n$ of $\widehat{\id}$ mapping to the long cycle under $vC_n\to S_n$, which is the long cycle $\breve \psi_n(r)$, as desired.

Second, note that $(\breve \phi_n)_*(s_{1j}), \breve \psi_n(s_{1j})\in vC_n$ both have image $w_{1j}\in S_n$ under $vC_n\to S_n$. Observe that the $1$-cube in $\Gamma^2(AC_n)$ starting at $\breve{\id}$ and labelled by $s_{1j}$ is send by $\widetilde \phi_n$ to the  directed $1$-cube of $\Gamma^1(PvC_n)$ starting at $\widehat{\id}$ and labelled by $s_A$, where $A=(1,2,\ldots, j)$. Since $\breve \psi_n(s_{1j})=s_{1j}=s_Aw_{1j}$, and $w_{1j}\in S_n$, it follows that $(\breve \phi_n)_*(s_{1j})= \breve\psi_n(s_{1j})$.
\end{proof}

The proof of the following is analogous to the proof of Lemma~\ref{le:pi1_of_tildeD} and we omit it.

 \begin{lem}
\label{le:pi1_of_tildeP} We have $\pi_1^{S_n}(\breve P_n)=\widetilde {AS}_n$, and $\breve P_n\to \widehat P_n$ induces the map $\widetilde {AS}_n\to vS_n$ from diagram (\ref{eq:groups}).
\end{lem}

Recall the homomorphism $\psi_n : C_n\to \widetilde{AC}_n$ from diagram (\ref{eq:groups}).

\begin{lem}
\label{le:pi1_of_D} We have $\pi_1^{S_n}(D_n)=C_n$, and $\phi_n \colon D_n\to \breve D_n$ induces $\psi_n$ between their equivariant fundamental groups.
\end{lem}
\begin{proof} Let $\Gamma^2(C_n)$ be the Cayley $2$-complex of $C_n$ with respect to the standard presentation. 
Let $PC_n$ be the kernel of the homomorphism $C_n\to S_n$. Note that $PC_n$ acts freely on $\Gamma^2(C_n)$ and hence $PC_n=\pi_1(D^2_n)$, where we define
$D^2_n=PC_n\backslash \Gamma^2(C_n)$. Note that $D^2_n$ is equipped with the action of $S_n=PC_n\backslash C_n$ by left multiplication.

We will now show how to identify $D^2_n$ as the $2$-skeleton of $D_n$. The $0$-cubes of $D^2_n$ are $PC_n$-orbits of the $0$-cubes in $\Gamma^2(C_n)$, and thus correspond to the elements
of $S_n$. Each element of $S_n$ is of form $w_\tau$ for a unique planar forest $\tau$ with trees having only one edge, i.e.\ a $0$-cube of $D_n$. The $1$-cubes of~$D_n$ connect $\tau,\tau'$ differing by reversing the order in an interval $[i,j]$. Then $w_\tau, w_\tau'$ correspond to the cosets $PC_ng,PC_ngs_{ij}$. Thus the corresponding $0$-cubes in~$D^2_n$ are also connected by a $1$-cube. The $2$-cubes of $D_n$ and $D^2_n$ are identified analogously. Since the action of $S_n$ on the $0$-cubes of $D_n$ is also by the left multiplication of $w_\tau$, we have  $\pi_1^{S_n}(D_n)=C_n$.

The $S_n$-equivariant local isometry $\phi_n\colon D^2_n\to \breve D^2_n$ lifts to a map from $\Gamma^2(C_n)$ to $\Gamma^2(AC_n)$ sending the identity $0$-cube to the identity $0$-cube and each incident $1$-cube labelled by $s_{ij}$ to the $1$-cube labelled by $s_{ij}$. Consequently, we have $(\phi_n)_*=\psi_n$.

\end{proof}

 \begin{lem}
\label{le:vertical} The maps $\widehat D_n\to \widehat  X_n\cong \widehat P_n$ and $ \breve D_n\to \breve P_n$ induce the maps ${vC}_n\to vS_n$ and $ \widetilde{AC}_n\to \widetilde{AS}_n$ from diagram (\ref{eq:groups}).
\end{lem}
\begin{proof} For ${vC}_n\to vS_n$, by the $S_n$-equivariance, we just need to prove that the homomorphism induced by $\widehat D_n\to \widehat  X_n$ is correct on $s_{1k}$. For $A=(1,\ldots, k)$, we have $s_A=s_{1k}w_{1k}$.
The directed $1$-cube of $\widehat D_n$ labelled by $s_A$ (we treat the $1$-skeleton of $\widehat D_n$ as the quotient of the Cayley graph $\Gamma^1(PvC_n)$), starting at the identity element, corresponds to the set of planar trees $\hat \tau$
with one tree not a single edge, and with leaves corresponding to $1,\ldots, k$. By Example~\ref{exa:edge}, the map $\widehat D_n\to \widehat  X_n$ sends this $1$-cube to the main diagonal of the image in $\widehat P$ of the cell in $P$ corresponding to the trivial coset of $\mathcal W=\langle w_{12}, \ldots, w_{k-1\,k}\rangle$. This diagonal is homotopic in~$P$, relative the endpoints, to an edge-path with vertices $w^1=\id, w^2,
\dots, w^m$, where $w^m$ is the longest word $w_\mathcal W$ in $\mathcal W$ and $m-1=\frac{k(k-1)}{2}$. In $\widehat P$ (with $1$-skeleton treated as the quotient of the Cayley graph $\Gamma^1(PvS_n)$), the $l$th of these $1$-cells is labelled by $s_{ij}$, where $i=w^l(k),j=w^l(k+1)$ and $k$ is defined by $w^{l+1}=w^lw_{k\, k+1}$.
Since $s_{ij}=w^l\sigma_kw_{k\, k+1}(w^l)^{-1}$ in $vS_n$, we see that the product of all consecutive $s_{ij}$, after cancellations, has the form $\sigma_{\mathcal W}w_{\mathcal W}$, which are the longest elements of ${\mathcal W}$ in the two copies of the symmetric group. Consequently, $s_{1k}$ is mapped to $\sigma_{\mathcal W}$, as desired.

For  $\widetilde{AC}_n\to \widetilde{AS}_n$, by the $r$-equivariance, we just need to prove that the homomorphism induced by $\breve D_n\to \breve  P_n$ is correct on $s_{1k}$. This follows from Example~\ref{exa:edge}, as before.
\end{proof}

\subsection{Fundamental groups of real points}
The group $S_n $ acts on the schemes $ \CF_n, \overline \Cf_n $ by permuting the labels on the $ \nu, \mu $ coordinates in the obvious way (equivalently, it permutes the labels on the marked curves).  This action restricts on actions on the real points of these schemes as well as on fibres of the $ \vareps $ map.

From Theorems \ref{th:homeohat} and \ref{th:pi1comb}, we immediately deduce the following result.

\begin{thm} \label{th:pi1geom1}
	There are isomorphisms $\pi_1^{S_n}(\overline F_n(\BR)) \cong vC_n $ and $ \pi_1^{S_n}(\overline \ft_n(\BR)) \cong vS_n $ making the following diagram commute.
	$$
	\begin{tikzcd}
		 \pi_1^{S_n}(\overline F_n(\BR)) \arrow[d] \arrow[r] & vC_n \arrow[d] \\
	\pi_1^{S_n}(\overline \ft_n(\BR)) \arrow[r] & vS_n
\end{tikzcd}
$$
\end{thm}

\begin{thm} \label{th:pi1geom2}
	There are isomorphisms $ \pi_1^{S_n}(\overline M_{n+2}^\sigma(\BR)) \cong \widetilde{AC}_n $ and $ \pi_1^{S_n}(U(1)^n/U(1)) \cong \widetilde{AS}_n $.
\end{thm}

\begin{proof}
	The first isomorphism follows from the work of Ceyhan \cite[Theorem 8.3]{C} who gave a presentation of the fundamental groupoid of $ \overline M_{n+2}^\sigma(\BR)$.  The isomorphism $ \pi_1^{S_n}(U(1)^n/U(1)) \cong \widetilde{AS}_n $ is immediate from Lemma \ref{lem:AS2} (or from Proposition \ref{prop:homeobreve1} and Theorem \ref{th:pi1comb}).
\end{proof}

Now, recall that by Theorem \ref{th:deformretract} we have a deformation retraction of $ \CF_n(\BR) $ onto $ \overline F_n(\BR) $.  Thus the inclusion of $ \overline F_n(\BR) \hookrightarrow \CF_n(\BR) $ gives an isomorphism $ \pi_1^{S_n}(\overline F_n(\BR)) \cong \pi_1^{S_n}(\CF_n(\BR)) $.  Also, we use the identification of $ \overline M_{n+2}^\sigma(\BR) $ with the $ \vareps = i $ fibre of $ \CF_n(\BR) $ to give a map $ 	\pi_1^{S_n}(\overline M_{n+2}^\sigma(\BR)) \rightarrow \pi_1^{S_n}(\CF_n(\BR))$.

\begin{thm} \label{th:pi1ACvC}
	The following diagram commutes
	$$
	\begin{tikzcd}
	\pi_1^{S_n}(\overline M_{n+2}^\sigma(\BR)) \arrow[r] \arrow[d] & \pi_1^{S_n}(\CF_n(\BR)) \cong \pi_1^{S_n}(\overline F_n(\BR)) \arrow[d] \\
	\widetilde{AC}_n \arrow[r] & vC_n
	\end{tikzcd}
$$
\end{thm}
Before proceeding to the proof, we note that this would follow from Conjecture \ref{conj:deform} and Theorem \ref{th:pi1comb}, but we will give a proof avoiding this conjecture.

\begin{proof}
As in the proof of Lemma \ref{le:pi1_of_tildeD}, we just need to check this for the generators $ s_{1k}, r $ of $ \widetilde{AC}_n $.

Fix a basepoint $ p \in \overline M_{n+2}^\sigma(\BR) $ which corresponds to a configuration of $ n $ evenly spaced points on $ U(1) $.  In terms of the $ \alpha $ coordinates, this means that $ \alpha_{j \, j+1}(p) = e^{i 2\pi/n} $ and that $ \nu_{j j+1} = \frac{i}{2} - \frac{1}{2} \cot \frac{\pi}{n} $ for $ j = 1, \dots, n -1 $.

The generator $ s_{1k} $ is represented by a path which begins at $ p $ and ends at $w_{1k}(p) $.  This path passes transversely through the codimension 1 stratum given by two component curves where the points $ z_1, \dots, z_k $ are on one component and the points $ z_0, z_{k+1}, \dots, z_n, z_{n+1} $ are on the other component (recall that the points $ z_1, \dots, z_n $ are real, while $ z_0, z_{n+1} $ are complex conjugate).

Recall that we have a diffeomorphism $ f : [0,1] \rightarrow [0, \infty] $. Define maps $ a, b :  [0, \frac{1}{2}] \times [0,1] \rightarrow \BC $ by
\begin{gather*} a(t,s) := \frac{is}{2} + \frac{s}{2} \cot \frac{\pi}{n} - f(2t)  \\
	b(t,s) := \frac{is}{2} + \frac{s}{2}((1-2t) \cot \frac{\pi}{n} + 2t \cot \frac{\pi}{n-k+1})
\end{gather*}
and then define $ H : [0, \frac{1}{2}] \times [0,1] \rightarrow \CF_n(\BR) $ by
$$
\nu_{j\, j+1}(H(t,s)) = \begin{cases} a(t,s) \text{ if $ j = 1, \dots, k-1$ } \\
	b(t,s)  \text{ if $ j = k, \dots, n-1 $}
	\end{cases} \quad \vareps(H(t,s)) = is
$$
As long as $ t \ne \frac{1}{2}$, this this point lives in $ \bCU_{[[n]]} = \overline \Cf_n^\circ $ and so only these $ \nu $ coordinates are needed.  But when $ t = \frac{1}{2}$, then $ \nu_{j \, j+1} = \infty $ for $ j = 1, \dots k-1 $ and we move into the open set $ \bCU_\CS $ defined by $ \CS = \{\{1, \dots, k\}, \{k+1\}, \dots, \{n\}\} $.  On this open set we have additional coordinate $ \mu_{abc} $ for $ abc \in t([k]) $.  A simple computation show that for any $ t, s $ and any $ j = 1, \dots, k-2$, we have $ \mu_{j\, j+1\, j+2} = \frac{ 2 a(2t,s) - is}{a(2t,s)} $ and thus at the limit $ t = \frac{1}{2} $ we get $ \mu_{j\, j+1\, j+2} = 2 $.  So this gives a well-defined point $ H(\frac{1}{2}, s) $.

Now, we extend $ H $ to $  [ \frac{1}{2}, 1] \times [0,1]$ by setting $ H(t,s) = w_{1 k}(H(1-t, s)) $.  This makes sense because $ H(\frac{1}{2}, s) $ is invariant under the action of $ w_{1,k} $ (to see this, note that $  \mu_{j\, j+1\, j+2} = 2 $ implies that $ \mu_{j+2\, j+1\, j} = 2 $).

Now $ H(t,1) : [0,1] \rightarrow  \overline M_{n+2}^\sigma(\BR)  $ gives the generator $ s_{1k} \in \pi_1^{S_n}(\overline M_{n+2}^\sigma(\BR) ) $ (this is because at $ H(\frac{1}{2}, 1) $ we pass through the desired codimension 1 stratum).  On the other hand, if we compare $ H(t,0) :[0,1] \rightarrow \overline F_n(\BR) $ with the map from the appropriate 1-cube of $ D_n $ defined in above Theorem \ref{th:homeohat}, we see that this loop gives the generator $ s_{1k} \in \pi_1^{S_n}(\overline F_n(\BR)) $.

Thus, we conclude that the two $ s_{1k} $ generators are homotopic within $ \CF_n^\sigma(\BR)$, which proves the desired result.

\end{proof}

\begin{rem}
	A more conceptual explanation of the compatibility of the cactus group generators is as follows.  Consider $ \CB = \{[n]\} $, the set partition with one part.  Then we have the stratum $\bCV^\CB \cong  \overline M_{n+1} \times \BA^1 $ by Proposition \ref{pr:strataCFn} (equivalently this is the ``zero section'' of the deformation of the line bundle $ \tCM_{n+1}$).
	
	The involution $ \sigma $ acts trivially on the $ \overline M_{n+1} $ factor and so we get $ \bCV^\CB(\BR) \cong \overline M_{n+1}(\BR) \times i\BR \subset \CF_n^\sigma(\BR)$.  This gives embeddings of $ \overline M_{n+1}(\BR)$ into both $ \overline M_{n+2}^\sigma(\BR) $ and $ \overline F_{n+1}(\BR) $ and hence we get a commutative diagram of fundamental groups
	$$
	\begin{tikzcd}
		C_n \cong \pi_1^{S_n}(\overline M_{n+1}(\BR)) \arrow[d,equal] \arrow[r] & \widetilde{AC}_n \cong \pi_1^{S_n}(\overline M_{n+2}^\sigma(\BR)) \arrow[d] \\
		C_n \cong \pi_1^{S_n}(\overline M_{n+1}(\BR)) \arrow[r] & vC_n \cong \pi_1^{S_n}(\overline F_{n+1}(\BR))
	\end{tikzcd}
	$$
	This analysis also applies to the equivariant fundamental group of the standard real form $ \pi_1^{S_n}(\overline M_{n+2}(\BR)) $ which we will investigate in a future paper.
\end{rem}
	\begin{appendix}
	
	\section{The permutahedron, the star, and the real points of the compactification of the Cartan}
	
	\subsection{Introduction}
	Let $\fg$ be a \textbf{semisimple} Lie algebra over $\mathbb C$ and $\fh$ a \textbf{Cartan subalgebra}.  The \textbf{compactification} $\bh$ of $\fh$ is a complex projective variety that contains the complex vector space $\fh$ as an open dense subvariety.  The variety $\bh$ turns out to be defined over $\mathbb Z$.  Its set of real points, equipped with the classical topology, will be denoted by $\bhr$.
	
	Let $\Phi$ be the set of \textbf{roots} associated with $(\fg, \fh)$.  The real vector space $\text{Span}_{\mathbb R} (\alpha: \alpha \in \Phi)$ will be denoted by $\mathfrak h^{\ast}_{\mathbb R}$.  If $\Phi_+$ is a choice of \textbf{positive roots} (positive system), then the \textbf{Weyl vector} $\rho$ with respect to $\Phi_+$ is defined to be the element $$\frac 12 \sum \limits_{\alpha \in \Phi_+} \alpha$$ of $\mathfrak h^{\ast}_{\mathbb R}$.  Write $W$ for the \textbf{Weyl group} associated with $\Phi$.  The convex hull in $\mathfrak h^{\ast}_{\mathbb R}$ of the set $$\{w \cdot \rho: w \in W\}$$ is called the \textbf{permutahedron} associated to $\Phi$ and will be denoted by $P$.
	
	The permutahedron $P$ is a convex polytope.  Hence it makes sense to speak of its faces.  Two faces $F_1$ and $F_2$ of $P$ are said to be \textbf{parallel} if there exists a vector $v \in \mathfrak h^{\ast}_{\mathbb R}$ such that $$F_1 + v = F_2.$$  In this case, we say that vectors $v_1 \in F_1$ and $v_2 \in F_2$ are \textbf{related} if $$v_1 + v = v_2,$$ and we write $$v_1 \sim v_2.$$  It is clear that $\sim$ is an equivalence relation on $P$.  We equip the set $\widehat P := P/\sim$ of equivalence classes with the quotient topology.
	
		Each root gives a linear functional on the Cartan subalgebra and together they provide an embedding $ \fh \rightarrow \BC^{\Phi} $.  We can embed $\BC \subset \BP^1 $ and then we define $  \bh $ to be the closure of the image of $ \fh $ in the product $ (\BP^1)^{\Phi}$.  This is a special case of the \textbf{matroid Schubert variety} construction of Ardila-Boocher \cite{AB}.
	
	The goal of this appendix is to construct a homeomorphism between $\widehat P$ and $\bhr$.  For this purpose, we now introduce an intermediate space, which we call the star.
	
	For a choice $\Pi$ of \textbf{simple roots} (simple system), we write $X_{\Pi}$ for the parallelepiped in $\mathfrak h^{\ast}_{\mathbb R}$ generated by the \textbf{fundamental weights} corresponding to $\Pi$.  The subset $$X := \bigcup \limits_{\substack{\Pi ~ \text{is a} \\ \text{simple system}}} X_{\Pi}$$ of $\mathfrak h^{\ast}_{\mathbb R}$ is called the \textbf{star}.
	
	\subsection{A map from the star to the permutahedron}
	
	\subsubsection{Faces of the permutahedron and their centres} \label{faceP}
	
	Let $\Pi$ be a simple system and $\Delta$ be a subset of $\Pi$.  Then $\Pi$ determines a positive system $\Phi_{\Pi}$.  The Weyl vector defined by $\Phi_{\Pi}$ will be denoted by $\rho_{\Pi}$.  The intersection of $P$ with the affine subspace $$\rho_{\Pi} + \text{Span}_{\mathbb R} (\Pi - \Delta)$$ of $\mathfrak h^{\ast}_{\mathbb R}$ will be denoted by $F^{\Delta}_{\Pi}$.  It is clear from the definition that $F^{\Delta}_{\Pi}$ is a face of $P$ that contains the vertex $\rho_{\Pi}$.
	
	Define $$\rho^{\Delta}_{\Pi} := \frac 12 \sum \limits_{\alpha \in \Phi_{\Pi - \Delta}} \alpha.$$  Namely, $\rho^{\Delta}_{\Pi}$ is half of the sum of those roots which are $\mathbb Z_{\ge 0}$-linear combination of elements of $\Pi - \Delta$.
	
	\begin{lem} \label{center}
		The centre of $F^{\Delta}_{\Pi}$ is $\rho_{\Pi} - \rho^{\Delta}_{\Pi}$.
	\end{lem}
	
	\begin{proof}
		Note that the set $\Pi - \Delta$ is a simple system of a root subsystem $\Psi$ of $\Phi$, and that $\rho^{\Delta}_{\Pi}$ is a vertex of the permutahedron $P_{\Psi}$ associated with $\Psi$.  Hence the vertices of $P_{\Psi}$ are elements of the set $$\{w \cdot \rho^{\Delta}_{\Pi}: w \in \langle s_{\alpha}: \alpha \in \Pi - \Delta \rangle\},$$ where $s_{\alpha}$ is the reflection associated with $\alpha$.  It follows that the vertices of the translation $P_{\Psi} + \rho_{\Pi} - \rho^{\Delta}_{\Pi}$ of $P_{\Psi}$ are of the form $$w \cdot \rho^{\Delta}_{\Pi} + \rho_{\Pi} - \rho^{\Delta}_{\Pi}, ~ \text{where} ~ w \in \langle s_{\alpha}: \alpha \in \Pi - \Delta \rangle.$$
		
		For any $w \in W$, define $$N(w) := \{\alpha \in \Phi_{\Pi}: w^{-1} \cdot \alpha \notin \Phi_{\Pi}\}.$$  Recall that $$\rho_{\Pi} - w \cdot \rho_{\Pi} = \sum \limits_{\alpha \in N(w)} \alpha.$$  Since, for $w \in \langle s_{\alpha}: \alpha \in \Pi - \Delta \rangle$, we have $N(w) \subseteq \Phi_{\Pi - \Delta}$, it follows that $$w \cdot \rho^{\Delta}_{\Pi} + \rho_{\Pi} - \rho^{\Delta}_{\Pi} = w \cdot \rho_{\Pi}.$$  Hence, $w \cdot \rho^{\Delta}_{\Pi} + \rho_{\Pi} - \rho^{\Delta}_{\Pi}$ is a vertex of $P$ and is contained in $\rho_{\Pi} + \text{Span}_{\mathbb R} (\Pi - \Delta)$.  Since $P$ is convex and $P_{\Psi} + \rho_{\Pi} - \rho^{\Delta}_{\Pi}$ is the convex hull of $\{w \cdot \rho^{\Delta}_{\Pi} + \rho_{\Pi} - \rho^{\Delta}_{\Pi}: w \in \langle s_{\alpha}: \alpha \in \Pi - \Delta \rangle\}$, we see that $$P_{\Psi} + \rho_{\Pi} - \rho^{\Delta}_{\Pi} \subseteq F^{\Delta}_{\Pi}.$$
		
		Observe that $\dim F^{\Delta}_{\Pi} = \text{Card} (\Pi - \Delta) = \dim P_{\Psi} $.  Hence we must have $$P_{\Psi} + \rho_{\Pi} - \rho^{\Delta}_{\Pi} = F^{\Delta}_{\Pi}.$$
		
		Since the centre of $P_{\Psi}$ is $0$, we conclude that the centre of $F^{\Delta}_{\Pi}$ is $\rho_{\Pi} - \rho^{\Delta}_{\Pi}$.
	\end{proof}
	
	\subsubsection{Faces of the parallelepiped} \label{faceX}
	
	Let $\Pi$ and $\Delta$ be as in Section \ref{faceP}.  Define $$X^{\Delta}_{\Pi} := \{x \in X_{\Pi}: \langle \alpha^{\vee}, x \rangle = 1 ~ \forall \alpha \in \Delta\}.$$  This is a face of the parallelepiped $X_{\Pi}$.  Write $\text{Fund}(\Pi)$ for the set of fundamental weights corresponding to $\Pi$.  For any set $D$ such that $\Delta \subseteq D \subseteq \Pi$, define $$\omega^D_{\Pi} := \sum \limits_{\substack{\omega \in \text{Fund}(\Pi), \\ \omega ~ \text{corresponds to} \\ \text{an element of} ~ D}} \omega.$$  It is obvious that $\omega^D_{\Pi}$ is a vertex of $X^{\Delta}_{\Pi}$ and all vertices of $X^{\Delta}_{\Pi}$ are of this form.
	
	\subsubsection{Mapping $X^{\Delta}_{\Pi}$ to $P$}
	
	Retain the notation from Section \ref{faceX}.  Intuitively, we would like to define a homeomorphism from $X$ to $P$ sending $X_{\Pi}$ to the ``corner'' of $P$ at $\rho_{\Pi}$, namely, the intersection of $P$ with the \textbf{fundamental Weyl chamber} $\mathcal C_{\Pi}$ determined by $\Pi$.  Naturally, we want the map to send the ``star point'' $\omega_{\Pi} := \omega^{\Pi}_{\Pi}$ of $X$ to the vertex $\rho_{\Pi}$ of $P$.  It is also natural to expect that the map sends the vertex $\omega^D_{\Pi}$ of $X_{\Pi}$ to the centre of the face $F^D_{\Pi}$ of $P$, namely the vector $\rho_{\Pi} - \rho^D_{\Pi}$ in view of Lemma \ref{center}.
	
	A point of $X^{\Delta}_{\Pi}$ is of the form
	\begin{equation} \label{pt}
		(\sum \limits_{\substack{\omega \in \text{Fund}(\Pi) \\ \omega ~ \text{does not correspond} \\ \text{to an element of} ~ \Delta}} t_{\omega} \omega) + \omega^{\Delta}_{\Pi},
	\end{equation}
	where each $t_{\omega}$ is in the interval $[0,1]$.  We define a map $$\Xi^{\Delta}_{\Pi}: X^{\Delta}_{\Pi} \rightarrow P$$ which sends such a point to $$\sum \limits_{\Delta \subseteq D \subseteq \Pi} (\prod \limits_{\substack{\omega \in \text{Fund}(\Pi) \\ \omega ~ \text{does not correspond} \\ \text{to an element of} ~ D}} (1 - t_{\omega}) \prod \limits_{\substack{\omega \in \text{Fund}(\Pi) \\ \omega ~ \text{corresponds to} \\ \text{an element of} ~ D - \Delta}} t_{\omega})(\rho_{\Pi} - \rho^D_{\Pi}).$$
	
	\begin{rem}
		\begin{enumerate}
			\item The map $\Xi^{\Delta}_{\Pi}$ is designed in such a way that $\Xi^{\Delta}_{\Pi}(\omega^D_{\Pi}) = \rho_{\Pi} - \rho^D_{\Pi}$ for any $ \Delta \subseteq D \subseteq \Pi$.
			\item The diagram
			\begin{equation} \label{diag}
				\begin{tikzcd}
					X^{\Delta}_{\Pi} \arrow[d, hook] \arrow[rd, "\Xi^{\Delta}_{\Pi}"] & \\
					X_{\Pi} = X^{\emptyset}_{\Pi} \arrow[r, "\Xi^{\emptyset}_{\Pi}"'] & P
				\end{tikzcd}
			\end{equation}
			is commutative.  Hence we have a well-defined map $$\Xi_{\Pi}: X_{\Pi} \rightarrow P$$ whose restriction to $X^{\Delta}_{\Pi}$ is $\Xi^{\Delta}_{\Pi}$.
		\end{enumerate}
	\end{rem}
	
	\subsubsection{Gluing the maps $\Xi_{\Pi}$}
	
	\begin{lem} \label{glue}
		Let $\Pi$ and $\Pi'$ be simple systems and $x$ a point of $X_{\Pi} \cap X_{\Pi'}$.  Then we have $$\Xi_{\Pi}(x) = \Xi_{\Pi'}(x).$$
	\end{lem}
	
	\begin{proof}
		By assumption, there is a common face of $X_{\Pi}$ and $X_{\Pi'}$ that contains $x$.  In other words, there exists a set $K$ such that $K \subseteq \text{Fund}(\Pi)$, $K \subseteq \text{Fund}(\Pi')$ and $x = \sum \limits_{\omega \in K} t_{\omega} \omega$.  When viewed as point of $X_{\Pi}$ and expressed in the form of (\ref{pt}), $x$ has the property that $t_{\omega} = 0$ for all $\omega \in \text{Fund}(\Pi) - K$.  Hence, by definition, we have
		\begin{align} \label{xi}
			\Xi_{\Pi}(x) & = \sum \limits_{D \subseteq \Pi} (\prod \limits_{\substack{\omega \in \text{Fund}(\Pi) \\ \omega ~ \text{does not correspond} \\ \text{to an element of} ~ D}} (1 - t_{\omega}) \prod \limits_{\substack{\omega \in \text{Fund}(\Pi) \\ \omega ~ \text{corresponds to} \\ \text{an element of} ~ D}} t_{\omega})(\rho_{\Pi} - \rho^D_{\Pi}) \nonumber \\
			& = \sum \limits_{\substack{D \subseteq \Pi \\ \text{no element of} ~ D ~ \text{corresponds} \\ \text{to an element of} ~ \text{Fund}(\Pi) - K}} (\prod \limits_{\substack{\omega \in \text{Fund}(\Pi) \\ \omega ~ \text{does not correspond} \\ \text{to an element of} ~ D}} (1 - t_{\omega}) \prod \limits_{\substack{\omega \in \text{Fund}(\Pi) \\ \omega ~ \text{corresponds to} \\ \text{an element of} ~ D}} t_{\omega})(\rho_{\Pi} - \rho^D_{\Pi}).
		\end{align}
		Analogous statements hold if we replace $\Pi$ with $\Pi'$.
		
		Recall that $\rho_{\Pi} = \sum \limits_{\omega \in \text{Fund}(\Pi)} \omega$.  So, for any $D \subseteq \Pi$, we have $\rho_{\Pi} - \rho^D_{\Pi} = \sum \limits_{\substack{\omega \in \text{Fund}(\Pi) \\ \omega ~ \text{corresponds to} \\ \text{an element of} ~ D}} \omega$.  In particular, if no element of $D$ corresponds to an element of $\text{Fund}(\Pi) - K$, then $\rho_{\Pi} - \rho^D_{\Pi}$ is a sum of elements of the set $K$.  This, together with the last line of (\ref{xi}), implies that $\Xi_{\Pi}(x) = \Xi_{\Pi'}(x)$.
	\end{proof}
	
	It follows from Lemma \ref{glue} that there exists a map $$\Xi: X \rightarrow P$$ whose restriction to $X_{\Pi}$ is $\Xi_{\Pi}$ for all simple systems $\Pi$.
	
	\subsection{The map $\Xi$ Is a homeomorphism}
	
	\subsubsection{Injectivity of $\Xi$}
	
	Observe that, for any simple system $\Pi$, the map $\Xi_{\Pi}$ is an injection.  Since $X$ is the union of the $X_{\Pi}$'s and $P$ is the union of the $P \cap \mathcal C_{\Pi}$'s, injectivity follows.
	
	\subsubsection{Surjectivity of $\Xi$}
	
	It suffices to show that, for any simple system $\Pi$, the image of $X_{\Pi}$ under $\Xi_{\Pi}$ contains $P \cap \mathcal C_{\Pi}$.  By \cite{DJS}, $P \cap \mathcal C_{\Pi}$ is combinatorially isomorphic to a cube of dimension $\text{Card}(\Pi)$.  It follows that $P \cap \mathcal C_{\Pi}$ is the convex polytope with vertices $$\rho_{\Pi} - \rho^D_{\Pi}, ~ D \subseteq \Pi.$$  The map $\Xi_{\Pi}$ is designed so that it is a bijection from $X_{\Pi}$ to the convex polytope with vertices $\rho_{\Pi} - \rho^D_{\Pi}, ~ D \subseteq \Pi$.  Hence $$\Xi_{\Pi}(X_{\Pi}) = P \cap \mathcal C_{\Pi},$$ as desired.
	
	\subsection{Translating the relation $\sim$ from the permutahedron to the star}
	
	Let $\Pi$ and $\Pi'$ be simple systems and $\Delta$ (resp. $\Delta'$) a subset of $\Pi$ (resp. $\Pi'$) such that $\Pi - \Delta = \Pi' - \Delta'$.  For $x \in X^{\Delta}_{\Pi}$ and $x' \in X^{\Delta'}_{\Pi'}$, we say that they are \textbf{related}, and write $x \sim x'$, if $$\langle \alpha^{\vee}, x \rangle = \langle \alpha^{\vee}, x' \rangle ~ \forall \alpha \in \Pi - \Delta = \Pi' - \Delta'.$$
	
	We show that this relation is well-defined.  Namely, we must show that if $\Pi''$ is a simple system and $\Delta''$ is a subset of $\Pi''$ such that $X^{\Delta}_{\Pi} = X^{\Delta''}_{\Pi''}$, then $x$, viewed as a point of $X^{\Delta''}_{\Pi''}$, is also related to $x'$.  The assumption $X^{\Delta}_{\Pi} = X^{\Delta''}_{\Pi''}$, together with well-definedness of $\Xi$ and  commutativity of the diagram (\ref{diag}), implies that $$\Xi^{\Delta}_{\Pi} (X^{\Delta}_{\Pi}) = \Xi^{\Delta''}_{\Pi''} (X^{\Delta''}_{\Pi''}).$$  Observe that the star point $\omega_{\Pi}$ (resp. $\omega_{\Pi''}$) of $X^{\Delta}_{\Pi}$ (resp. $X^{\Delta''}_{\Pi''}$) is mapped to $\rho_{\Pi}$ (resp. $\rho_{\Pi''}$) under $\Xi^{\Delta}_{\Pi}$ (resp. $\Xi^{\Delta''}_{\Pi''}$).  It follows that $\rho_{\Pi} = \rho_{\Pi''}$ and, hence, that $\Pi = \Pi''$.  But then the assumption $X^{\Delta}_{\Pi} = X^{\Delta''}_{\Pi''}$ forces $\Delta = \Delta''$.  This proves what we need.
	
	Now we show that $ \Xi $ intertwines the equivalence relations on $ X $ and $ P $.
	\begin{lem} \label{le:A4}
		For any $x, x' \in X$, we have $$x \sim x' ~ \text{if and only if} ~ \Xi(x) \sim \Xi(x').$$
	\end{lem}
	
	\begin{proof}
		\underline{Only if}) Suppose that $x \sim x'$.  There exist simple systems $\Pi, \Pi'$, subsets $\Delta \subseteq \Pi, \Delta' \subseteq \Pi'$ such that
		\begin{align*}
			& \Pi - \Delta = \Pi' - \Delta' = \{\alpha_1, \cdots, \alpha_k\}, \\
			& \Pi = \{\alpha_1, \cdots, \alpha_k, \alpha_{k+1}, \cdots, \alpha_r\}, ~ \Pi' = \{\alpha_1, \cdots, \alpha_k, \alpha'_{k+1}, \cdots, \alpha'_r\}, \\
			& \text{Fund}(\Pi) = \{\omega_1, \cdots, \omega_r\}, ~ \text{Fund}(\Pi') = \{\omega'_1, \cdots, \omega'_r\}, ~ \text{and} \\
			& x = \omega^{\Delta}_{\Pi} + \sum \limits_{i=1}^k t_i \omega_i, ~ x' = \omega^{\Delta'}_{\Pi'} + \sum \limits_{i=1}^k t_i \omega'_i
		\end{align*}
		for some $t_1, \cdots, t_k \in [0,1]$.
		
		We compute
		\begin{align} \label{translation}
			\Xi(x) & = \sum \limits_{I \subseteq \{k+1, \cdots, r\}} (\prod \limits_{i \in \{k+1, \cdots, r\} - I} (1 - t_i) \prod \limits_{i \in I} t_i) (\rho_{\Pi} - \rho^{\Delta \cup \{\alpha_i: i \in I\}}_{\Pi}) \\
			& = \sum \limits_{I \subseteq \{k+1, \cdots, r\}} (\prod \limits_{i \in \{k+1, \cdots, r\} - I} (1 - t_i) \prod \limits_{i \in I} t_i) (\rho_{\Pi'} - \rho^{\Delta' \cup \{\alpha_i: i \in I\}}_{\Pi'} + \rho_{\Pi} - \rho_{\Pi'}) \nonumber \\
			& = \Xi(x') + \sum \limits_{I \subseteq \{k+1, \cdots, r\}} (\prod \limits_{i \in \{k+1, \cdots, r\} - I} (1 - t_i) \prod \limits_{i \in I} t_i) (\rho_{\Pi} - \rho_{\Pi'}) \nonumber \\
			& = \Xi(x') + (\rho_{\Pi} - \rho_{\Pi'}), \nonumber
		\end{align}
		where the second equality follows from the definition of $\rho^{\Delta \cup \{\alpha_i: i \in I\}}_{\Pi}$ and $\rho^{\Delta' \cup \{\alpha_i: i \in I\}}_{\Pi'}$, and the assumption that $\Pi - \Delta = \Pi' - \Delta'$; and the last equality follows from the fact that $$\sum \limits_{I \subseteq \{k+1, \cdots, r\}} (\prod \limits_{i \in \{k+1, \cdots, r\} - I} (1 - t_i) \prod \limits_{i \in I} t_i) = 1.$$
		
		In particular, this computation tells us that, for any $I \subseteq \{k+1, \cdots, r\}$, the vertex $\omega^{\Delta \cup \{\alpha_i: i \in I\}}_{\Pi}$ of $X^{\Delta}_{\Pi}$ and the vertex $\omega^{\Delta' \cup \{\alpha_i: i \in I\}}_{\Pi'}$ of $X^{\Delta'}_{\Pi'}$ are related by
		\begin{equation} \label{vertex}
			\Xi(\omega^{\Delta \cup \{\alpha_i: i \in I\}}_{\Pi}) = \Xi(\omega^{\Delta' \cup \{\alpha_i: i \in I\}}_{\Pi'}) + (\rho_{\Pi} - \rho_{\Pi'}).
		\end{equation}
		
		Note that $$\Xi(\omega^{\Delta \cup \{\alpha_i: i \in I\}}_{\Pi}) = \rho_{\Pi} - \rho^{\Delta \cup \{\alpha_i: i \in I\}}_{\Pi},$$ which is the centre of the face $F^{\Delta \cup \{\alpha_i: i \in I\}}_{\Pi}$ of $P$ by Lemma \ref{center}.  Since $F^{\Delta \cup \{\alpha_i: i \in I\}}_{\Pi}$ is a face of $F^{\Delta}_{\Pi}$, we see that $$\Xi(\omega^{\Delta \cup \{\alpha_i: i \in I\}}_{\Pi}) \in F^{\Delta}_{\Pi}.$$  In particular, we have $$\Xi(X^{\Delta}_{\Pi}) \subseteq F^{\Delta}_{\Pi}.$$  In fact, our argument proves that $\Xi(X^{\Delta}_{\Pi})$ is the ``corner'' of $F^{\Delta}_{\Pi}$ near the vertex $\rho_{\Pi}$.  It follows that the minimal affine subspace containing $\Xi(X^{\Delta}_{\Pi})$ is equal to that containing $F^{\Delta}_{\Pi}$.  The same conclusions hold if we replace $\Pi$ and $\Delta$ with $\Pi'$ and $\Delta'$.  Then, by (\ref{vertex}), we have $$F^{\Delta}_{\Pi} = F^{\Delta'}_{\Pi'} + (\rho_{\Pi} - \rho_{\Pi'}).$$  From this and (\ref{translation}) it follows that $$\Xi(x) \sim \Xi(x').$$
		
		\underline{If}) Suppose $\Xi(x) \sim \Xi(x')$.  Choose a simple system $\Pi$ and a subset $\Delta \subseteq \Pi$ such $X^{\Delta}_{\Pi}$ is the minimal face of $X$ that contains $x$.  Choose $\Pi'$ and $\Delta'$ similarly.  We have proved above that $$\Xi(X^{\Delta}_{\Pi}) \subseteq F^{\Delta}_{\Pi} ~ \text{and} ~ \Xi(X^{\Delta'}_{\Pi'}) \subseteq F^{\Delta'}_{\Pi'}.$$  Minimality of $X^{\Delta}_{\Pi}$ and $X^{\Delta'}_{\Pi'}$, together with the assumption that $\Xi(x) \sim \Xi(x')$, implies that $F^{\Delta}_{\Pi}$ and $F^{\Delta'}_{\Pi'}$ are translations of each other.  Since vertices of a face of $P$ are Weyl vectors, there exists a simple system $\Pi''$ such that $$F^{\Delta}_{\Pi} = F^{\Delta'}_{\Pi'} + (\rho_{\Pi} - \rho_{\Pi''}).$$  Since $\Xi(x)$ is in the ``corner'' of $F^{\Delta}_{\Pi}$ near $\rho_{\Pi}$, namely the convex hull $$\text{conv} \{\rho_{\Pi} - \rho^D_{\Pi}: \Delta \subseteq D \subseteq \Pi\},$$ and similarly for $\Xi(x')$, we must have $$\{\rho_{\Pi} - \rho^D_{\Pi}: \Delta \subseteq D \subseteq \Pi\} = \{\rho_{\Pi'} - \rho^{D'}_{\Pi'}: \Delta' \subseteq D' \subseteq \Pi'\} + (\rho_{\Pi} - \rho_{\Pi''}).$$
		
		Take $D = \Pi$.  Then there exists $\Delta' \subseteq D' \subseteq \Pi'$ such that $$\rho_{\Pi} - \rho^{\Pi}_{\Pi} = \rho_{\Pi'} - \rho^{D'}_{\Pi'} + (\rho_{\Pi} - \rho_{\Pi''}),$$ equivalently $$\rho_{\Pi''} = \rho_{\Pi'} - \rho^{D'}_{\Pi'}.$$  The left-hand side of the last equality is a vertex of $P$, while the right-hand side is the centre of the face $F^{D'}_{\Pi'}$.  This forces $\Pi' = \Pi''$ and $D' = \Pi'$.  In particular, we get $$\{\rho^D_{\Pi}: \Delta \subseteq D \subseteq \Pi\} = \{\rho^{D'}_{\Pi'}: \Delta' \subseteq D' \subseteq \Pi'\}.$$
		
		Now take $D = (\Pi - \Delta) - \{\alpha\}$ for some $\alpha \in \Pi - \Delta$.  We see that $$\frac 12 \alpha \in \text{Span}_{\mathbb R} (\Pi' - \Delta').$$  Varying $\alpha$, we have $$\text{Span}_{\mathbb R} (\Pi - \Delta) \subseteq \text{Span}_{\mathbb R} (\Pi' - \Delta').$$  By symmetry, we in fact have $$\text{Span}_{\mathbb R} (\Pi - \Delta) = \text{Span}_{\mathbb R} (\Pi' - \Delta').$$
		
		Notice that $\Pi - \Delta$ (resp. $\Pi' - \Delta'$) is a simple system for the root subsystem $\Phi \cap \text{Span}_{\mathbb R} (\Pi - \Delta)$ (resp. $\Phi \cap \text{Span}_{\mathbb R} (\Pi' - \Delta')$) of $\Phi$.  So if we can show that $\rho^{\Delta}_{\Pi} = \rho^{\Delta'}_{\Pi'}$, then we have $\Pi - \Delta = \Pi' - \Delta'$.  Take $D = \Delta$.  Then there exists $\Delta' \subseteq D' \subseteq \Pi'$ such that $$\rho^{\Delta}_{\Pi} = \rho^{D'}_{\Pi'}.$$  The left-hand side is a vertex of the permutahedron of the root system $\Phi \cap \text{Span}_{\mathbb R} (\Pi - \Delta)$.  The right-hand side is a vertex of the permutahedron of the root system $\Phi \cap \text{Span}_{\mathbb R} (\Pi' - \Delta')$ only if $D' = \Delta'$.  Hence we get $\rho^{\Delta}_{\Pi} = \rho^{\Delta'}_{\Pi'}$, proving $\Pi - \Delta = \Pi' - \Delta'$.
		
		Now we use the same notation as in the proof of the only if part, except that $$x' = \omega^{\Delta'}_{\Pi'} + \sum \limits_{i=1}^k t'_i \omega'_i.$$  It follows from $$\Xi(x) = \Xi(x') + (\rho_{\Pi} - \rho_{\Pi'})$$ that
		\begin{align*}
			\sum \limits_{I \subseteq \{k+1 \cdots, r\}} (\prod \limits_{i \in \{k+1, \cdots, r\} - I} (1 - t_i) \prod \limits_{i \in I} t_i) \rho^{\Delta \cup \{\alpha_i: i \in I\}}_{\Pi} \\ = \sum \limits_{I \subseteq \{k+1 \cdots, r\}} (\prod \limits_{i \in \{k+1, \cdots, r\} - I} (1 - t'_i) \prod \limits_{i \in I} t'_i) \rho^{\Delta' \cup \{\alpha_i: i \in I\}}_{\Pi'}.
		\end{align*}
		Since $\Pi - \Delta = \Pi' - \Delta'$, the first line is equal to $$\sum \limits_{I \subseteq \{k+1 \cdots, r\}} (\prod \limits_{i \in \{k+1, \cdots, r\} - I} (1 - t_i) \prod \limits_{i \in I} t_i) \rho^{\Delta' \cup \{\alpha_i: i \in I\}}_{\Pi'}.$$  By linear independence, it follows that $$\prod \limits_{i \in \{k+1, \cdots, r\} - I} (1 - t_i) \prod \limits_{i \in I} t_i = \prod \limits_{i \in \{k+1, \cdots, r\} - I} (1 - t'_i) \prod \limits_{i \in I} t'_i$$ for all $I \subseteq \{k+1, \cdots, r\}$.  Hence $t_i = t'_i$ for all $i \in \{k+1, \cdots, r\}$.
		
	\end{proof}
	
	Equip the set $\widehat X := X/\sim$ of equivalence classes in $X$ with the quotient topology.  We have proved:
	
	\begin{prop}
		The map $\Xi: X \rightarrow P$ induces a homeomorphism $$\widehat X \rightarrow \widehat P.$$
	\end{prop}
	
	By abuse of notation, the homeomorphism above will also be denoted by $\Xi$.

	\subsection{Mapping the parallelepiped $X_{\Pi}$ to the real locus}
	Recall, from section \ref{se:diffeof}, the increasing diffeomorphism $ f : [-1,1] \rightarrow [-\infty, \infty] $ such that $ f(0) = 0$.
	
	For any simple system $\Pi$, we define
	\begin{align*}
		\Theta_{\Pi}: X_{\Pi} & \longrightarrow \bhr, \\
		\sum \limits_{\omega \in \text{Fund}(\Pi)} t_{\omega} \omega & \longmapsto (z_{\alpha})_{\alpha \in \Phi}
	\end{align*}
	where $z_{\alpha} := \sum \limits_{\alpha_i \in \Pi} \langle \omega_i^{\vee}, \alpha \rangle f(t_{\omega_i})$ and $\omega_i$ is the fundamental weight corresponding to the simple root $\alpha_i$.  Note that, by definition, the image of $X_{\Pi}$ under the map $\Theta_{\Pi}$ lies in the fundamental Weyl chamber determined by $\Pi$.
	
	\subsection{Gluing the maps $\Theta_{\Pi}$}
	
	Retain the notation from the proof of Lemma \ref{glue}.  We would like to show that $$\Theta_{\Pi}(x) = \Theta_{\Pi'}(x).$$
	
	Write
	\begin{align*}
		& K = \{\omega_1, \cdots, \omega_k\}, \\ & \text{Fund}(\Pi) = \{\omega_1, \cdots, \omega_k, \omega_{k+1}, \cdots, \omega_r\}, ~ \text{Fund}(\Pi') = \{\omega_1, \cdots, \omega_k, \omega'_{k+1}, \cdots, \omega'_r\}, \\ & \Pi = \{\alpha_1, \cdots, \alpha_r\}, ~ \Pi' = \{\alpha'_1, \cdots, \alpha'_r\}.
	\end{align*}
	By definition, we have $f(t_{\omega_i}) = 0$ for all $k + 1 \le i \le r$.
	
	It is obvious that, for any $1 \le i \le k$, we have $$\alpha_i - \alpha'_i \in \{\omega_1, \cdots, \omega_k\}^{\perp}.$$
	
	Hence, for any $\alpha \in \Phi$, if $\alpha = \sum \limits_{i=1}^r n_i \alpha_i$, then we have
	\begin{equation*}
		\alpha = \sum \limits_{i=1}^k n_i \alpha_i + \sum \limits_{i=k+1}^r n_i \alpha_i  = \sum \limits_{i=1}^k n_i (\alpha'_i + \beta_i) + \sum \limits_{i=k+1}^r n_i \alpha_i
	\end{equation*}
	where the $\beta_i$'s are in $\{\omega_1, \cdots, \omega_k\}^{\perp}$.  Hence $ \alpha \in \sum \limits_{i=1}^k n_i \alpha'_i + \text{Span}(\alpha'_{k+1}, \cdots, \alpha'_r) $.
	
	It then follows from definition that the $\alpha$-component of $\Theta_{\Pi}(x)$ and $\Theta_{\Pi'}(x)$ are both equal to $\sum \limits_{i=1}^k n_i f(t_{\omega_i})$.  Varying $\alpha$, we see that $\Theta_{\Pi}(x) = \Theta_{\Pi'}(x)$.
	
	Therefore, there exists a map $$\Theta: X \rightarrow \bhr$$ whose restriction to $X_{\Pi}$ is $\Theta_{\Pi}$ for every simple system $\Pi$.
	
	\subsection{Surjectivity of $\Theta$}
	
	Let $z = (z_{\alpha})_{\alpha \in \Phi}$ be in $\bhr$.  Define $\Psi := \{\alpha \in \Phi: z_{\alpha} \neq \infty\}$.  It is easy to verify that $\Psi$ is a \textbf{closed root subsystem} of $\Phi$ which is maximal, with respect to inclusion, in its own $\mathbb R$-span.   From \cite[Lemma 3.2.3(a)]{Krammer}, we know that such a root subsystem is \textbf{parabolic}, namely, $\Psi$ has a simple system which extends to a simple system for $\Phi$ (we thank Dyer for this reference).
	
	Choose a simple system $E$ for $\Psi$ such that $z_{\alpha} \ge 0$ for all $\alpha \in E$.  Extend $E$ to a simple system $\Pi$ for $\Phi$.  Write $E = \{\alpha_1, \cdots, \alpha_k\}$, $\Pi = \{\alpha_1, \cdots, \alpha_k, \alpha_{k+1}, \cdots, \alpha_r\}$, and $\text{Fund}(\Pi) = \{\omega_1, \cdots, \omega_r\}$.  Then, for each $1 \le i \le k$, we have $f^{-1}(z_{\alpha_i}) \in [0, 1)$; and for each $k+1 \le i \le r$, we have $f^{-1}(z_{\alpha_i}) = 1$.  So $$\sum \limits_{i=1}^k f^{-1}(z_{\alpha_i}) \omega_i$$ is a point of $X_{\Pi}$.  By the definition of $\Theta_{\Pi}$, it is clear that the $\alpha_i$-component of $\Theta_{\Pi} (\sum \limits_{i=1}^k f^{-1}(z_{\alpha_i}) \omega_i)$ is equal to $z_{\alpha_i}$.  Notice that $z$ is in the closure of the fundamental Weyl chamber determined by $\Pi$.  Hence $z$ is determined by $z_{\alpha_i}$, $1 \le i \le r$, and it follows that $$\Theta_{\Pi} (\sum \limits_{i=1}^k f^{-1}(z_{\alpha_i}) \omega_i) = z,$$ proving surjectivity of $\Theta$.
	
	\subsection{Injectivity of $\Theta$}
	
	Let $X^{\circ}$ be the interior of $X$.  We have $$X = X^{\circ} \sqcup \bigcup \limits_{\substack{\Pi ~ \text{is a simple system} \\ \Delta \subseteq \Pi}} X^{\Delta}_{\Pi}.$$  By the definition of $\Theta$, it is clear that no component of the image under $\Theta$ of a point of $X^{\circ}$ is equal to $\infty$, whereas at least one component of the image under $\Theta$ of a point of $\bigcup \limits_{\substack{\Pi ~ \text{is a simple system} \\ \Delta \subseteq \Pi}} X^{\Delta}_{\Pi}$ is equal to $\infty$.  It follows that $\Theta(X^{\circ})$ and $\Theta(\bigcup \limits_{\substack{\Pi ~ \text{is a simple system} \\ \Delta \subseteq \Pi}} X^{\Delta}_{\Pi})$ are disjoint.
	
	On $X^{\circ}$, since $\Theta_{\Pi}$ is a combinatorial isomorphism from $X^{\circ} \cap X_{\Pi}$ to $\fh(\mathbb R) \cap \mathcal C_{\Pi}$ for all simple systems $\Pi$, we have that $\Theta$ is injective on $X^{\circ}$.
	
	Suppose $x, x' \in \bigcup \limits_{\substack{\Pi ~ \text{is a simple system} \\ \Delta \subseteq \Pi}} X^{\Delta}_{\Pi}$ are such that $$\Theta(x) = \Theta(x').$$  Define $$\Psi := \{\alpha \in \Phi: ~ \text{the} ~ \alpha \text{-component of} ~ \Theta(x) = \Theta(x') ~ \text{is not equal to} ~ \infty\}.$$  Again, this is a parabolic subsystem of $\Phi$.  Hence, there exist simple systems $\Pi, \Pi'$, subsets $\Delta \subseteq \Pi, \Delta' \subseteq \Pi'$ such that $\Pi - \Delta$ (resp. $\Pi' - \Delta'$) is a simple system for $\Psi$ and $x \in X^{\Delta}_{\Pi}$ (resp. $x \in X^{\Delta'}_{\Pi'}$).
	
	Write $$\text{pr}: (\mathbb{RP}^1)^{\Phi} \rightarrow (\mathbb{RP}^1)^{\Psi}$$ for the projection onto the $\alpha$-components, $\alpha \in \Psi$.  Applying the argument in the second paragraph of this section to the root system $\Psi$, we see that there exists a subset $S$ of $(\Pi - \Delta) \cap (\Pi' - \Delta')$ such that, if $\{\omega_1, \cdots, \omega_l\}$ are the fundamental weights for the root system $\Psi$ corresponding to elements of $S$, then $$\text{pr}(x) = \text{pr}(x') = \sum \limits_{i=1}^l t_{\omega_i} \omega_i.$$  Moreover, we may assume that $\Pi - \Delta = \Pi' - \Delta'$.
	
	It follows that
	\begin{align*}
		& x = \omega^{\Delta}_{\Pi} + \sum \limits_{i=1}^l t_{\omega_i} \omega_i, ~ \text{and} \\
		& x' = \omega^{\Delta'}_{\Pi'} + \sum \limits_{i=1}^l t_{\omega_i} \omega_i,
	\end{align*}
	and, hence, that $x \sim x'$.
	
	Conversely, if $x, x' \in X$ are such that $x \sim x'$, it is clear from definition that $\Theta(x) = \Theta(x')$.  Therefore, we have proved:
	
	\begin{thm}
		The map $\Theta: X \rightarrow \bhr$ induces a homeomorphism $$ \widehat X \longrightarrow \bhr.$$
	\end{thm}
	
	By abuse of notation, the homeomorphism above will also be denoted by $\Theta$.
	
	Putting everything together, we have proved:
	
	\begin{thm} \label{th:Bothmaps}
		Both maps below are homeomorphisms:
		\begin{equation*}
			\begin{tikzcd}
				\widehat P = P/\sim & \widehat X = X/\sim \arrow[r, "\Theta"] \arrow[l, swap, "\Xi"] & \bhr.
			\end{tikzcd}
		\end{equation*}
		
		In particular, $\bhr$ is homeomorphic to the permutahedron $P$ modulo the identification of parallel faces.
	\end{thm}
	
	Recently, Leo Jiang and the third author \cite{JL} have generalized the homeomorphism $ \widehat P \cong \bhr $ to any real matroid Schubert variety, where the permutahedron is replaced by a zonotope.
	
\end{appendix}

\end{document}